\documentclass[10pt,a4paper,reqno]{amsart}
\usepackage[utf8]{inputenc}
\usepackage{amsmath}
\usepackage{mathpazo}
\usepackage{enumitem}
\usepackage{amsthm}
\usepackage{indentfirst}
\usepackage{amsfonts}
\usepackage{verbatim}
\usepackage{amssymb}
\usepackage[hidelinks,linktocpage]{hyperref}
\hypersetup{
	colorlinks   = true, %Colours links instead of ugly boxes
	urlcolor     = blue, %Colour for external hyperlinks
	linkcolor    = blue, %Colour of internal links
	citecolor   = green %Colour of citations
}
\usepackage{array}
\usepackage{stmaryrd}
\usepackage{dsfont}
\usepackage{graphicx} 
\usepackage[left=2.5cm,right=2.5cm,top=3cm,bottom=3cm]{geometry}
\usepackage{color}
\usepackage[all]{xy}
\usepackage[english, french]{babel}
\usepackage{mathrsfs}
\theoremstyle{plain}
\newtheorem{prop}{Proposition}[section]
\newtheorem{lem}[prop]{Lemma}
\newtheorem*{lem*}{Lemma}
\newtheorem{thm}[prop]{Theorem}

\newtheorem{cor}[prop]{Corollary}
\theoremstyle{definition}
\newtheorem{defi}[prop]{Definition}
\newtheorem{ex}[prop]{Example}
\theoremstyle{remark}
\newtheorem{rem}[prop]{Remark}
\newtheorem*{conv}{Convention}

\numberwithin{equation}{section}
\begin{document}

\selectlanguage{english}

	\title{Categorification of the adjoint action of quantum groups}
	
	\author{Laurent Vera}
	
	\maketitle
	
	\begin{abstract}
		Let $U$ be a quantized enveloping algebra. We consider the adjoint action of an $\mathfrak{sl}_2$-subalgebra of $U$ on a subalgebra of $U^+$ that is maximal integrable for this action. We categorify this representation in the context of quiver Hecke algebras. We obtain an action of the 2-category associated with $\mathfrak{sl}_2$ on a category of modules over certain quotients of quiver Hecke algebras. Our approach is similar to that of Kang-Kashiwara \cite{kk} for categorifications of highest weight modules via cyclotomic quiver Hecke algebras. One of the main new features is a compatibility of the categorical action with the monoidal structure, categorifying the notion of derivation on an algebra. As an application of some of our results, we categorify the higher order quantum Serre relations, extending results of Sto\v{s}i\'{c} \cite{stosic} to the non simply-laced case.
	\end{abstract}
	
	\tableofcontents
	
	\section{Introduction}	
	The study of categorified quantum groups began, in its current form, with the work of Chuang and Rouquier \cite{ChR}. While categorifications of representations of Lie algebras and their quantized versions had already appeared (for instance in \cite{ariki}, \cite{bfk} or \cite{huerf} among others), the key novelty in \cite{ChR} was the introduction of Hecke algebra actions at the level of natural transformations in the axiomatic. The resulting notion is called an $\mathfrak{sl}_2$-categorification. This was later generalized to arbitrary symmetrizable Kac-Moody types by Khovanov and Lauda in \cite{catII}, \cite{catIII}, and Rouquier in \cite{2km}. More precisely, let $C$ be a symmetrizable generalized Catan matrix and $U$ the corresponding quantum group. In \cite{catII} and \cite{2km}, a family of graded algebra $(H_{\beta})_{\beta \in Q^+}$ attached to $C$ is introduced, where $Q^+$ denotes the cone of linear combinations of simple roots with coefficients in $\mathbb{Z}_{\geqslant 0}$. These algebras are now known as the Khovanov-Lauda-Rouquier algebras (or simply KLR algebras, or quiver Hecke algebras). It is shown by Khovanov and Lauda in \cite{catII} that over a field, the direct sum over $\beta \in Q^+$ of the Grothendieck groups of the category of finitely generated graded projective $H_{\beta}$-modules is isomorphic to the integral positive part of $U$. The multiplication of $U$ corresponds to an induction product on the KLR algebras side. Furthermore, Khovanov-Lauda and Rouquier introduce 2-categories which categorify Beilinson-Lusztig-MacPherson's idempotent version of the quantum group $\dot{U}$. While the definitions of the 2-categories in \cite{catIII} and \cite{2km} differ on a few points, it was shown by Brundan in \cite{brun} that they are actually isomorphic. Hence there is an essentially unique 2-quantum group $\mathcal U$ associated to $C$. Rouquier proved in \cite{2km} that in the case $C=(2)$, the 2-representations of the 2-category $\mathcal U$ recover the $\mathfrak{sl}_2$-categorifications of \cite{ChR}.
	
	In \cite{catI}, Khovanov and Lauda also conjecture that the irreducible module of $U$ of highest weight $\Lambda$ is categorified by the cyclotomic KLR algebras $(H_{\beta}^{\Lambda})_{\beta \in Q^+}$, which are certain quotients of the KLR algebras. This was proved by Kang and Kashiwara in \cite{kk} (see also \cite{webs}). We recall briefly the results and the strategy of proof of \cite{kk}. Kang and Kashiwara prove that the 2-category $\mathcal U$ acts on the direct sum of the categories of $H_{\beta}^{\Lambda}$-modules, for $\beta \in Q^+$. Given a simple root $i$, the Chevalley generators $F_i$ and $E_i$ of $\mathcal U$ act as some induction and restriction functors $F_i^{\Lambda}$ and $E_i^{\Lambda}$ between the algebras $H_{\beta}^{\Lambda}$. Their proof that these functors yield an action of $\mathcal U$ is based on the following key result. Given a module $M$ over $H_{\beta}^{\Lambda}$, they show that there is an exact sequence
	\begin{equation}\label{seskk}
		0 \rightarrow \overline{F}_i(M) \rightarrow F_i(M) \rightarrow F_i^{\Lambda}(M) \rightarrow 0.
	\end{equation}
	where $\overline{F}_i$ is a ``left $i$-induction" functor and $F_i$ a ``right $i$-induction" functor between KLR algebras. Furthermore this exact sequence is natural in $M$. From this they can recover many of the properties needed to construct a representation of $\mathcal U$, such as the exactness of the functors $F_i^{\Lambda}$ and $E_i^{\Lambda}$, and a categorification of the Lie algebra relation $[e_i,f_i]=h_i$. They also prove that at the Grothendieck group level, this action of $\mathcal U$ gives the irreducible module of $U$ of highest weight $\Lambda$.
	
	In this paper, we categorify part of the adjoint action of $U$ using a similar approach. By adjoint action, we mean the left adjoint action of $U$ on itself arising from the Hopf algebra structure. Explicitly, for a simple root $i$, the adjoint action of the Chevalley generators $e_i,f_i,k_i$ on a element $y$ of $U$ of weight $\beta \in Q^+$ takes the form
	\begin{equation}\label{adj}
		\mathrm{ad}_{e_i}(y)= e_iy - q_i^{\left<i^{\vee},\beta\right>}ye_i, \quad \mathrm{ad}_{f_i}(y)=(f_iy-yf_i)k_i, \quad \mathrm{ad}_{k_i}(y) = q_i^{\left<i^{\vee},\beta\right>}y.
	\end{equation}
	where $i^{\vee}$ denotes the coroot associated to $i$, and $\left< ,\right>$ is the pairing between the dual weight and the weight lattices. Note that the adjoint action of $U$ on itself is not integrable. Our approach does not yield a categorification of the complete adjoint action of $U$, but rather of the action of a given $\mathfrak{sl}_2$-subalgebra $U_i$ of $U$ on a subalgebra $U^+[i]$ of $U^+$ which is integrable for the action of $U_i$, and maximal for this property. More explicitly, for a fixed simple root $i$, $U_i$ is the subalgebra of $U$ generated by $e_i,f_i$ and $k_i$. The subalgebra $U^+[i]$ of $U^+$ is generated by the $\mathrm{ad}_{e_i}^{(n)}(e_j)$ for $n \geqslant 0$ and $j$ a simple root not equal to $i$. Here $\mathrm{ad}_{e_i}^{(n)}$ denotes the $n^{\mathrm{th}}$ divided power of $\mathrm{ad}_{e_i}$. The subalgebra $U^+[i]$ is studied by Lusztig in \cite[Chapter~38]{Lu}, as part of his study of the braid group action. We prove in Proposition \ref{int} that the adjoint action induces an integrable representation of $U_i$ on $U^+[i]$, and that $U^+[i]$ is the largest subspace of $U^+$ with this property. It is this representation that we categorify. To do so, we start by categorifying the algebra $U^+[i]$. For $\beta \in Q^+$, we define an algebra $H_{\beta}^i$ as the quotient of $H_{\beta}$ by the two-sided ideal generated by the idempotent $1_{\beta-i,i}$. These algebras are the analogues of the cyclotomic KLR algebras in our context. However, they behave quite differently: for instance, they are typically infinite dimensional (see Proposition \ref{infgen}) while the cyclotomic KLR algebras are always finite dimensional. Then, we define the category $\mathcal H[i]$ as the direct sum of the categories of finitely generated graded $H_{\beta}^i$-modules, for $\beta \in Q^+$. Our category $\mathcal H[i]$ is a Serre and monoidal full subcategory of the category of all modules over the KLR algebras, and serves as a categorical analogue of the subalgebra $U^+[i]$. On $\mathcal H[i]$ we consider an endofunctor $\mathrm{ad}_{E_i}$, which can be defined as an induction functor between the algebras $H_{\beta}^i$, in a way similar to the functors $F_i^{\Lambda}$ of \cite{kk}. In Proposition \ref{adei}, we endow the powers of $\mathrm{ad}_{E_i}$ with an action of the affine nil Hecke algebras. In particular, we obtain well-defined divided powers $\mathrm{ad}_{E_i}^{(n)}$. This action is one of the axioms to establish a structure of 2-representation of $\mathfrak{sl}_2$ on $\mathcal H[i]$. The other axioms are more delicate to check, and we prove them using a similar approach to \cite{kk}. Our first main result is the following.
	
	\begin{thm}[Theorem \ref{tauinj}]\label{first}
		For all $M \in \mathcal{H}[i]$ of weight $\beta \in Q^+$, there is a short exact sequence
		\[
			0 \rightarrow q_i^{\left< i^{\vee},\beta\right>}ME_i \rightarrow E_iM \rightarrow \mathrm{ad}_{E_i}(M) \rightarrow 0.
		\]
		Furthermore, this sequence is natural in $M$.
	\end{thm}

	Here, the coefficient $q_i^{\left< i^{\vee},\beta\right>}$ denotes a grading shift. The short exact sequence of Theorem \ref{first} can be seen as a categorification of equation (\ref{adj}), and it shows in particular that the functor $\mathrm{ad}_{E_i}$ lifts the operator $\mathrm{ad}_{e_i}$ to the categorical setting. It is an analogue of the short exact sequence (\ref{seskk}) from \cite{kk} in our framework. From this, we can deduce as in \cite{kk} that the functor $\mathrm{ad}_{E_i}$ is exact (Corollary \ref{exact}). At this stage, the main axiom left to check is the categorification of the relation $[e_i,f_i]=h_i$. This is done in Theorem \ref{weyl}. The conclusion of this work is that the action of $U_i$ on $U^+[i]$ lifts to an $\mathfrak{sl}_2$ categorical action on $\mathcal H[i]$. This is our second main result.

	\begin{thm}[Theorem \ref{2rep}]\label{second}
		The endofunctor $\mathrm{ad}_{E_i}$ induces a 2-representation of $\mathfrak{sl}_2$ on $\mathcal H[i]$.
	\end{thm} 
	
	A key new feature of our work is a compatibility of the 2-representation with the monoidal structure. In general, the adjoint action of a Hopf algebra on itself is compatible with the multiplication. In the case of the quantum group $U$, this takes the simple form $\mathrm{ad}_{e_i}(yz) = \mathrm{ad}_{e_i}(y)z+  q_i^{\left< i^{\vee},\beta\right>}y\mathrm{ad}_{e_i}(z)$ for $y,z \in U$ and $y$ of weight $\beta \in Q^+$. We prove that this formula has a categorical analogue, in the form of a short exact sequence.
	\begin{thm}[Corollary \ref{derivation}]\label{third}
		For $M,N \in \mathcal H[i]$ with $M$ of weight $\beta \in Q^+$, there is a short exact sequence
		\[
			0 \rightarrow q_i^{\left< i^{\vee},\beta\right>}M \ \mathrm{ad}_{E_i}(N) \rightarrow \mathrm{ad}_{E_i}(MN) \rightarrow \mathrm{ad}_{E_i}(M)N \rightarrow 0.
		\]
		Furthermore, this sequence is natural in $M,N$.
	\end{thm}
	This feature does not appear for cyclotomic KLR algebras (there is no monoidal structure in that case) and, to the author's knowledge, this is the first example of a 2-representation with such structure. We use this in a crucial way to simplify the computations in our proofs. Iterating the short exact sequence of Theorem \ref{third}, we obtain an interesting filtration of $\mathrm{ad}_{E_i}^n(MN)$. We describe it in detail in Proposition \ref{nderivation}. As one of the consequences of Theorem \ref{third}, we can prove the following result, which categorifies the fact that $U^+[i]$ is generated by the $\mathrm{ad}_{e_i}^{(n)}(e_j)$ for $n\geqslant 0$ and $j \in I \setminus \lbrace i \rbrace$ as a subalgebra of $U^+$.
	\begin{thm}[Theorem \ref{genforhi}, Corollary \ref{gen2}]
		The category $\mathcal H[i]$ is generated by 
		\begin{itemize}
			\item the modules $\mathrm{ad}_{E_i}^{(n)}(E_{j_1}\ldots E_{j_r})$ for $n\geqslant 0$ and $j_1,\ldots,j_r \in I \setminus \lbrace i \rbrace$ as a Serre subcategory,
			\item the modules $\mathrm{ad}_{E_i}^{(n)}(E_j)$ for $n\geqslant 0$ and $j \in I \setminus \lbrace i \rbrace$, as a Serre and monoidal subcategory.
		\end{itemize}
	\end{thm}
	Thanks to this theorem, we are able to reduce some explicit computations to modules of the form $\mathrm{ad}_{E_i}^{(n)}(E_{j_1}\ldots E_{j_r})$, which can be understood quite well.
	
	As we mentioned above, the action of $U_i$ on $U^+[i]$ is an integrable representation. As another consequence of Theorems \ref{first} and \ref{third}, we prove in Corollary \ref{nilp} that the functor $\mathrm{ad}_{E_i}$ is locally nilpotent. More precisely, we prove in Proposition \ref{vanish} that the algebra $H_{\beta}^i$ is zero precisely when $s_i(\beta) \notin Q^+$, where $s_i$ is the reflection of the root lattice corresponding the the simple root $i$. In particular, the algebra $H_{\beta+ni}^i$ is zero when $n$ is large enough, which proves that $\mathrm{ad}_{E_i}$ is locally nilpotent. 
		
	Finally, as an application of the above results, we construct projective resolutions of the modules $\mathrm{ad}_{E_i}^{(n)}(E_{j_1}\ldots E_{j_r})$ and prove a categorification of the higher order quantum Serre relations. This generalizes results of Sto\v{s}i\'{c} \cite{stosic} to the non simply laced case. At the decategorified level, the higher order quantum Serre relations state that $\mathrm{ad}_{e_i}^{(n)}(e_j^m)=0$ for all simple roots $i\neq j$ and $n>-m\left< i^{\vee},j\right>$. One can write this more explicitly as
	\[
		\sum_{k=0}^n (-1)^k q_i^{k(n+m\left< i^{\vee},j\right>-1)} e_i^{(n-k)}e_j^me_i^{(k)} =0.
	\]
	This takes the following categorical form. For a simple root $i$, $n\geqslant 0$ and $M$ an $H^i_{\beta}$-module, we define a complex of $H_{\beta+ni}$-modules of the form
	\[
		\mathrm{Ad}_{E_i}^{(n)}(M) = 0 \rightarrow q_i^{n(n+\left<i^{\vee},\beta\right>-1) } ME_i^{(n)} \rightarrow \ldots\rightarrow q_i^{k(n+\left<i^{\vee},\beta\right>-1)}E_i^{(n-k)}ME_i^{(k)} \rightarrow \ldots \rightarrow E_i^{(n)}M \rightarrow 0,
	\]
	where the term $E_i^{(n)}M$ is in cohomological degree 0. We prove in Theorem \ref{cohodp} that the cohomology of $\mathrm{Ad}_{E_i}^{(n)}(M)$ is concentrated in degree 0, and equal to $\mathrm{ad}_{E_i}^{(n)}(M)$. By our vanishing criterion on the algebras $H_{\beta+ni}^i$, we also know that this cohomology is zero when $n$ is large enough. When $M=E_j^m$, the bound for the vanishing of the cohomology is simply $n>-m\left<i^{\vee},j\right>$. In that case, we get a complex of projective modules with zero cohomology, hence a null-homotopic complex (Theorem \ref{Serre}). This categorifies the higher order quantum Serre relation. Our method is completely different from that of \cite{stosic}, where the result is proved by constructing explicit homotopies, using the thick diagrammatic calculus for KLR algebras (see \cite{thick}, \cite{thick2}). Our approach has the drawback of not providing explicit homotopies, but has the advantage of working for any type and minimizing the amount of computations done in the homotopy category of KLR algebras. Similar projective resolutions also appear in \cite{bkm} and recently in \cite{resol}.
	
	We now describe the structure of the paper. In Section 2, we recall the main definitions regarding quantum groups following \cite{Lu}, and we briefly discuss their adjoint representation. In Section 3, we define the KLR algebras, and recall a few basic results (including the PBW theorem and a description of the center). We also explain how they categorify the positive part of quantum groups. In Section 4, we define and study the category $\mathcal H[i]$ and the functor $\mathrm{ad}_{E_i}$. After proving some elementary properties, we state our main theorem \ref{tauinj} and deduce from it various consequences, such as the exactness of $\mathrm{ad}_{E_i}$ (Corollary \ref{exact}), the compatibility with the monoidal structure (Corollary \ref{derivation}) and the vanishing criterion for the algebras $H_{\beta}^i$ (Proposition \ref{vanish}). The end of Section 4 is devoted to the proof of Theorem \ref{tauinj}. In Section 5, we complete the proof that $\mathcal H[i]$ is endowed with a structure of $\mathfrak{sl}_2$-categorification. This entails checking the categorification of the Lie algebra relation $[e_i,f_i]=h_i$, which amounts to proving some Mackey-type decompositions for the algebras $H_{\beta}^i$. Finally, Section 6 is devoted to the construction of projective resolutions and to the proof of the categorical higher quantum Serre relations. Section 6 does not rely on Section 5.
	
	\subsection*{Acknowledgments} I would like to thank my advisor, Rapha\"{e}l Rouquier for his support and many helpful discussions.
		
	\bigskip
	
	\section{Quantum groups}	
	\subsection{Definitions}
	\subsubsection{Root datum} We refer to \cite{Lu} for a detailed introduction to quantum groups. For the rest of this paper, we fix a \textit{Cartan datum} $(I,\cdot)$. This means that $I$ is a non empty set and ${\cdot:\mathbb{Z}I \otimes_{\mathbb{Z}} \mathbb{Z}I \rightarrow \mathbb{Z}}$ is a symmetric bilinear form such that
	\begin{itemize}
		\item for all $i \in I$, we have $i\cdot i \in 2\mathbb{Z}_{>0}$,
		\item for all $i\neq j \in I$, we have $\frac{2 i\cdot j}{i\cdot i} \in \mathbb{Z}_{\leqslant 0}$.		
	\end{itemize}

	We also fix a \textit{root datum} of type $(I,\cdot)$. This is the data of finitely generated free abelian groups $X$ (the \textit{weight latice}) and $Y$ (the \textit{dual weight latice}), together with a perfect pairing $\left< \cdot, \cdot \right> : Y\otimes_{\mathbb Z} X \rightarrow \mathbb Z$ and injective set maps $(I \hookrightarrow X, i \mapsto i)$, $(I \hookrightarrow Y, i \mapsto i^{\vee})$ satisfying the condition
	\[
		\left< i^{\vee}, j \right> = 2\frac{ i \cdot j}{i\cdot i}
	\]
	for all $i,j \in I$. Let $c_{i,j}=\left<i^{\vee}, j \right>$ and $d_i=\frac{i\cdot i}{2}$. For all $i \in I$, there is an automorphism $s_i$ of the lattice $X$ defined by $s_i(\lambda) = \lambda -\left< i^{\vee},\lambda\right>i$ for $\lambda \in X$. The \textit{root lattice} is $Q = \oplus_{i \in I} \mathbb{Z}i$, and we let $Q^+ = \oplus_{i \in I} \mathbb{Z}_{\geqslant 0} i$. Given an element $\beta = \sum_{i\in I}n_ii$ of $Q^+$, the integer $n=\sum_{i\in I}n_i$ is called the \textit{height} of $\beta$, and denoted $\vert\beta\vert$. We also let
	\[
		I^{\beta} = \Big \{ (j_1,\ldots,j_n) \in I^n, \, \sum_{k=1}^n j_k = \beta \Big \}.
	\]
	
	\subsubsection{Quantum groups} Let us start by introducing some notation in $\mathbb{Q}(q)$. For all $i \in I$, we let $q_i = q^{d_i}$.
	For $k\in \mathbb{Z}$, the quantum integer $\left[k\right]_i$ is given by
	\[
		\left[k\right]_i = \frac{q_i^k - q_i^{-k}}{q_i-q_i^{-1}}.
	\]
	If $k \geqslant 0$, the quantum factorial $\left[k\right]_i!$ is given by
	\[
		\left[k\right]_i! = \prod_{l=1}^{k} \left[l\right]_i.
	\]
	
	\begin{defi} The \textit{quantum group} $U$ associated to the above root datum is the unital $\mathbb Q(q)$-algebra on the generators $e_i, f_i$ for $i \in I$, and $k_{\mu}$ for $\mu \in Y$ subject to the relations
	\begin{enumerate}
		\item $k_0=1$ and $k_{\mu}k_{\mu'} = k_{\mu + \mu'}$ for all $\mu,\mu' \in Y$,\\
		
		\item $k_{\mu}e_i = q^{\left< \mu,i \right>} e_ik_{\mu}$ and $k_{\mu}f_i = q^{-\left< \mu,i \right>} f_ik_{\mu}$ for all $i \in I$ and $\mu \in Y$,\\
		
		\item for all $i,j \in I$
		\[
			e_if_j-f_je_i = \delta_{i,j} \frac{k_i - k_i^{-1}}{q_i-q_i^{-1}}
		\]
		where $k_i = k_{d_ii^{\vee}}$,\\
		
		\item for all $i \neq j \in I$, the quantum Serre relations:
		\[
			\sum_{\ell=0}^{-c_{i,j}+1} (-1)^{\ell}e_i^{(\ell)}e_je_i^{(-c_{i,j} +1-\ell)}=0, \quad
			\sum_{\ell=0}^{-c_{i,j}+1} (-1)^{\ell}f_i^{(\ell)}f_jf_i^{(-c_{i,j} +1-\ell)}=0,
		\]
		where the elements $e_i^{(\ell)}$ and $f_i^{(\ell)}$ are called the \textit{divided powers} and are defined by
		\[
			e_i^{(\ell)} = \frac{1}{[\ell]_i!}e_i^{\ell}, \quad f_i^{(\ell)} = \frac{1}{[\ell]_i!}f_i^{\ell}.
		\]
	\end{enumerate}
	\end{defi}

	The algebra $U$ is $Q$-graded with $e_i$ in degree $i$ and $f_i$ in degree $-i$ for all $i\in I$, and $k_{\mu}$ in degree 0 for all $\mu \in Y$. Furthermore, $U$ has a structure of $Q$-graded Hopf algebra, with coproduct $\Delta$ and antipode $S$ defined by the following formulas:
	\[
		\begin{array}{lclcl}
		\Delta(e_i) = e_i\otimes 1 + k_i\otimes e_i, & & \Delta(f_i) = f_i \otimes k_i^{-1} + 1\otimes f_i, & & \Delta(k_{\mu}) = k_{\mu} \otimes k_{\mu},\\
		S(e_i) = -k_i^{-1}e_i, & & S(f_i) = -f_ik_i, & & S(k_{\mu}) = k_{\mu}^{-1},
		\end{array}
	\]
	for all $i \in I$ and $\mu \in Y$.
	
	For $i \in I$, we denote by $U_i$ the subalgebra of $U$ generated by $e_i, f_i$ and $k_i$. We also denote by $U^+$ the subalgebra of $U$ generated by the $e_j$ for $j\in I$. There is a non-degenerate bilinear form $y\otimes z \mapsto (y,z)$ on $U^+$. To describe it, we start by endowing $U^+ \otimes_{\mathbb{Q}(q)} U^+$ with an algebra structure by defining the multiplication
	\[
		(y\otimes z)(y'\otimes z') = q^{\beta\cdot\gamma} (yy')\otimes(zz')
	\]
	for $y,y',z,z' \in U^+$ with $z,y'$ homogeneous of respective degrees $\beta,\gamma \in Q^+$. There is a morphism of algebras $r : U^+ \rightarrow U^+ \otimes_{\mathbb{Q}(q)} U^+$ defined by $r(e_i) = e_i\otimes 1+1\otimes e_i$ for all $i \in I$. Then there is a unique bilinear form $y\otimes z \mapsto (y,z)$ on $U^+$ satisfying the following properties (see \cite[Proposition~1.2.3]{Lu}):
	\begin{enumerate}
		\item $(1,1)=1$,
		\item $(e_i,e_j) = \delta_{i,j}\frac{1}{1-q_i^{2}}$ for all $i,j \in I$,
		\item $(y,zz') = (r(y),z\otimes z')$ for all $y,z,z' \in U^+$,
		\item $(yy',z) = (y\otimes y',r(z))$ for all $y,y',z \in U^+$.
	\end{enumerate}
	
	Finally, we define the integral form of $U$. Let $\mathcal{A} = \mathbb{Z}\left[q,q^{-1}\right]$. We let $_{\mathcal A}U$ be the sub-$\mathcal{A}$-algebra of $U$ generated by the $k_{\mu}$ for $\mu \in Y$, and $e_i^{(n)},f_i^{(n)}$ for $i\in I$ and $n\geqslant 0$. We have subalgebras of $_{\mathcal A}U$ defined as above: $_{\mathcal A}U^+$ and $_{\mathcal A}U_i$ for $i \in I$.
	
	\subsubsection{Representations of quantum groups} A \textit{weight representation} of $U$ is a $U$-module $V$ that decomposes as
	\[
		V = \bigoplus_{\lambda \in X} V_{\lambda}
	\]
	where $V_{\lambda} = \big \{ v \in V \vert \, \forall \mu \in Y, k_{\mu} v = q^{\left< \lambda,\mu\right>}v \big \}$.
	
	An \textit{integrable representation} of $U$ is a weight representation $V$ on which the action of $e_i$ and $f_i$ is locally nilpotent for all $i \in I$. More explictly, a weight representation $V$ is an integrable representation if and only if for all $v \in V$ and $i \in I$, there exists an integer $n\geqslant 0$ such that $e_i^nv=f_i^nv=0$.
	
	\medskip
	
	\subsection{Adjoint representation}If $A$ is a Hopf algebra with coproduct $\Delta$ and antipode $S$, the (left) adjoint representation of $A$ on itself is defined by
	\[
		\mathrm{ad}_z(y) = \sum z_{(1)}yS(z_{(2)})
	\]
	for all $y,z \in A$. Here we have used Sweedler's notation for the coproduct
	\[
		\Delta(z) = \sum z_{(1)} \otimes z_{(2)}.
	\]
	The adjoint action is compatible with the product of $A$, in the following sense: for all $y,y',z \in A$ we have
	\[
		\mathrm{ad}_z(yy') = \sum \mathrm{ad}_{z_{(1)}}(y)\mathrm{ad}_{z_{(2)}}(y').
	\]
	
	In the case of the quantum group $U$, there are simple formulas for the adjoint action of the algebra generators of $U$. For $y \in U$ homogeneous of degree $\beta \in Q^+$, $i \in I$ and $\mu \in Y$ we have
	\[
		\mathrm{ad}_{e_i}(y) = e_iy - q_i^{\left<i^{\vee},\beta\right>} ye_i, \quad
		\mathrm{ad}_{f_i}(y) = (f_iy-yf_i)k_i, \quad
		\mathrm{ad}_{k_{\mu}}(y) = q^{\left<\mu,\beta\right>}y.
	\]	
	The compatibility with the product takes the form of the following ``$q$-Leibniz formulas"
	\begin{equation}\label{qder}
		\begin{array}{l}\mathrm{ad}_{e_i}(yz) = \mathrm{ad}_{e_i}(y)z+q_i^{\left<i^{\vee},\beta\right>} y\mathrm{ad}_{e_i}(z), \\
		\mathrm{ad}_{f_i}(yz) = q_i^{\left<i^{\vee},\gamma\right>}\mathrm{ad}_{f_i}(y)z+y\mathrm{ad}_{e_i}(z), \end{array}
	\end{equation}
	for all $i\in I$ and $y,z \in U$ homogeneous of respective degrees $\beta,\gamma \in Q^+$. Hence $\mathrm{ad}_{e_i}$ and $\mathrm{ad}_{f_i}$ can be thought of as ``$q$-derivations" of the algebra $U$. We can define divided powers for $\mathrm{ad}_{e_i}$ and $\mathrm{ad}_{f_i}$ as above:
	\[
		\mathrm{ad}_{e_i}^{(n)} = \frac{1}{[n]_i!}\mathrm{ad}_{e_i}^n, \quad \mathrm{ad}_{f_i}^{(n)} = \frac{1}{[n]_i!}\mathrm{ad}_{f_i}^n.
	\]	
	The quantum Serre relations take a particularly simple form in terms of the adjoint representation. Namely, we have
	\[
		\mathrm{ad}_{e_i}^{(-c_{i,j}+1)}(e_j) = 0, \quad \mathrm{ad}_{f_i}^{(-c_{i,j}+1)}(f_j) = 0,
	\]
	for all $i,j \in I$. These can be generalized to higher order quantum Serre relations (see \cite[Proposition~7.1.5]{Lu}). We have
	\[
		\mathrm{ad}_{e_i}^{(n)}(e_j^m) = 0, \quad \mathrm{ad}_{f_i}^{(n)}(f_j^m) = 0,
	\]
	for all $i\neq j \in I$ and $n,m \in \mathbb Z_{\geqslant 0}$ such that $n> -mc_{i,j}$.\\
	
	The adjoint representation is a weight representation of $U$. However, it is not a integrable representation of $U$ because the operators $\mathrm{ad}_{e_i}$ and $\mathrm{ad}_{f_i}$ are not locally nilpotent. For instance, for every integer $m\geqslant 0$ we have
	\begin{equation}\label{eiei}
		\mathrm{ad}_{e_i}^m(e_i) = \bigg(\prod_{k=1}^{m}(1-q_i^{2m})\bigg)e_i^{m+1} \neq 0.
	\end{equation}
	Nevertheless, an integrable representation can be obtained by looking at the action of a given $\mathfrak{sl}_2$-subalgebra on a certain subalgebra of $U^+$. More precisely, we fix $i\in I$ and consider the adjoint action of the subalgebra $U_i$ on $U$. Following Lusztig \cite[Chapter~38]{Lu}, we let $U^+[i]$ be the subalgebra of $U^+$ generated by the elements $\mathrm{ad}_{e_i}^{(n)}(e_j)$, for $j \in I \setminus \lbrace i \rbrace$ and $n \geqslant 0$. The subalgebra $U^+[i]$ can also be described in terms of the inner product of $U^+$, as follows.
	
	\begin{prop}[{\cite[Proposition~38.1.6]{Lu}}]\label{rad}
		The subalgebra $U^+[i]$ consists of the elements $y\in U^+$ such that $(ze_i,y)=0$ for all $z \in U^+$.
	\end{prop}

	We also have the following useful result.

	\begin{prop}[{\cite[Lemmas 38.1.2 and 38.1.5]{Lu}}]\label{dec}
		The multiplication map $U^+[i] \otimes_{\mathbb{Q}(q)} \mathbb{Q}(q)[e_i] \rightarrow U^+$ is an isomorphism.
	\end{prop}

	We now prove the main property of $U^+[i]$ regarding the adjoint action.
	
	\begin{prop}\label{int}
		The subalgebra $U^+[i]$ is stable under the adjoint action of $U_i$. The adjoint action of $U_i$ on $U^+[i]$ is an integrable representation. Furthermore, $U^+[i]$ is maximal in the following sense: if $V\subseteq U^+$ is an integrable $U_i$-submodule of $U$ for the adjoint action, then $V\subseteq U^+[i]$.
	\end{prop}

	\begin{proof}
		We need to prove that $U^+[i]$ is stable under $\mathrm{ad}_{e_i}$ and $\mathrm{ad}_{f_i}$. Since $\mathrm{ad}_{e_i}$ and $\mathrm{ad}_{f_i}$ satisfy formulas (\ref{qder}), it suffices to prove that $\mathrm{ad}_{e_i}(y),\mathrm{ad}_{f_i}(y) \in U^+[i]$ for $y$ an algebra generator of $U^+[i]$.
		
		Let $y=\mathrm{ad}_{e_i}^{n}(e_j)$ for $n\geqslant 0$ and $j \in I \setminus \lbrace i\rbrace$. Then $\mathrm{ad}_{e_i}(y) = \mathrm{ad}_{e_i}^{n+1}(e_j) \in U^+[i]$. Hence $U^+[i]$ is stable under $\mathrm{ad}_{e_i}$. To prove that $\mathrm{ad}_{f_i}(x) \in U^+[i]$, we remark that
		\begin{align*}
			&\mathrm{ad}_{f_i}(e_j) = 0, \\
			&\mathrm{ad}_{f_i}(y) = \mathrm{ad}_{e_i}(\mathrm{ad}_{f_i}(\mathrm{ad}_{e_i}^{n-1}(e_j))) - [c_{i,j}+2(n-1)]_i\mathrm{ad}_{e_i}^{n-1}(e_j),
		\end{align*}
		the second equation following from the relation $e_if_i-f_ie_i=\frac{k_i-k_i^{-1}}{q_i-q_i^{-1}}$. The result now follows by an immediate induction on $n$.
		
		Hence we have a representation of $U_i$ on $U^+[i]$. Let us check that this representation is integrable. The set of elements $y\in U^+[i]$ on which $\mathrm{ad}_{e_i}$ and $\mathrm{ad}_{f_i}$ act nilpotently is a sub-$U_i$-module, and also a subalgebra given relations (\ref{qder}). By the quantum Serre relations, $\mathrm{ad}_{e_i}$ and $\mathrm{ad}_{f_i}$ act nilpotently on every $e_j$ for $j \in I\setminus \lbrace i \rbrace$. Since these generate $U^+[i]$ as an algebra and $U_i$-module, we conclude that the action of $U_i$ on $U^+[i]$ is integrable.
		
		Finally, we prove the maximality of $U^+[i]$ with respect to these properties. Let $y\in U^+$ be such that $\mathrm{ad}_{e_i}^m(y)=0$ for some $m\geqslant 0$. By Proposition \ref{dec}, we can write $y$ in the form
		\[
			y = \sum_{\ell=0}^N y_{\ell}e_i^{\ell}
		\] 
		for some elements $y_0,\ldots,y_{N}$ of $U^+[i]$ with $y_N\neq0$. Assume $N>0$. Using equations (\ref{qder}) and (\ref{eiei}), we see that for all $\ell>0$ we have
		\[
			\mathrm{ad}_{e_i}^m(y_{\ell}e_i^{\ell}) \in a_{\ell}y_{\ell}e_i^{\ell+m} +\sum_{k<\ell+m} U^+[i]e_i^k
		\]
		for some $a_{\ell} \in \mathbb{Q}(q)$ non-zero. Hence
		\[
			0 =\mathrm{ad}_{e_i}^m(y) \in a_Ny_Ne_i^{N+m} +\sum_{k<N+m} U^+[i]e_i^k.
		\]
		By the uniqueness part of Proposition \ref{dec}, we have $y_N=0$ which is a contradiction. Thus $N=0$, and the result follows.
	\end{proof}
	
	\bigskip
	
	\section{KLR algebras} In this section, we define the KLR algebras and recall how they categorify the positive part of the quantum group. We follow the approach and definitions of Rouquier \cite{2km}, \cite{qha}. We also refer to \cite{brundansurvey} for a survey of the subject. At the end of the section, we prove some elementary computational results about KLR algebras that will be useful in the proofs of our main results.\\ 
	
	For the rest of this paper, we fix a commutative ring $K$. We use the symbol $\otimes$ for $\otimes_K$. If $M$ is a graded $K$-module with degree $d$ component $M_d$, we denote by $qM$ the shifted module whose grading is defined by $(qM)_d=M_{d+1}$. For $a = \sum_{\ell \in \mathbb Z} a_{\ell}q^{\ell} \in \mathbb{N}[q,q^{-1}]$ we let
	\[
		aM = \bigoplus_{\ell \in \mathbb Z} (q^{\ell}M)^{\oplus a_{\ell}}.
	\]
	
	\subsection{Polynomial rings and nil Hecke algebras}
	\subsubsection{Symmetric groups} The symmetric group on $n$ letters is denoted by $\mathfrak S_n$. For $k \neq \ell$ two integers in $\lbrace1,\ldots,n\rbrace$, we let $s_{k,\ell}$ be the permutation $(k \, \ell) \in \mathfrak S_n$, and we put $s_k=s_{k,k+1}$. Given a finite sequence $\underline{k}=(k_1,\ldots,k_r)$ of elements of $\lbrace1,\ldots,n-1\rbrace$, we let $s_{\underline{k}}=s_{k_1}\ldots s_{k_r}$. If $k\leqslant \ell$, we denote by $\left[k\uparrow \ell\right]$ the sequence $(k,k+1,\ldots,\ell)$ and by $\left[\ell \downarrow k\right]$ the sequence $(\ell,\ell-1,\ldots k)$. If $k>\ell$, $\left[k\uparrow \ell\right]$ and $\left[\ell \downarrow k\right]$ are understood to be the empty sequence $\varnothing$, in which case $s_{\varnothing} =1$.
	
	For $\omega \in \mathfrak S_n$, a \textit{reduced decomposition} of $\omega$ is a finite sequence $\underline{k} = (k_1,\ldots,k_r)$ such that $s_{\underline{k}} = \omega$ and $r$ is minimal (the integer $r$ is the \textit{length} of $\omega$, denoted by $l(\omega)$). The longest element of $\mathfrak S_n$ is denoted by $\omega_0[1,n]$. If $k<\ell \in \lbrace 1,\ldots,n\rbrace$, we denote by $\omega_0[k,\ell]$ the longest element of the parabolic subgroup of $\mathfrak{S}_n$ generated by $s_k,\ldots,s_{\ell-1}$. The elements $\omega_0[k,\ell]$ satisfy the relations
	\begin{equation}\label{longest}
		\begin{array}{ll}
		\omega_0[k,\ell+1]& = \omega_0[k,\ell]s_{[\ell\downarrow k]} = s_{[k\uparrow \ell]}\omega_0[k,\ell] \\
		&=\omega_0[k+1,\ell+1]s_{[k\uparrow \ell]} = s_{[\ell\downarrow k]}\omega_0[k+1,\ell+1].
		\end{array}		
	\end{equation}
	
	\subsubsection{Demazure operators} Let $P_n =K[x_1,\ldots x_n]$. We consider $P_n$ graded with $x_k$ in degree 2 for all $k\in \lbrace1,\ldots,n\rbrace$. The symmetric group $\mathfrak{S}_n$ acts on $P_n$ by permuting $x_1,\ldots,x_n$. For $k\neq \ell$ two integers in $\lbrace1,\ldots,n-1\rbrace$, we define the \textit{Demazure operator}, or \textit{divided difference operator}, $\partial_{k,\ell}$ on $P_n$ by
	\[
		\partial_{k,\ell}(f) = \frac{f-s_{k,\ell}(f)}{x_{\ell}-x_k}.
	\]
	For $k \in \lbrace1,\ldots,n-1\rbrace$, we put $\partial_k = \partial_{k,k+1}$. The Demazure operators are $P_n^{\mathfrak S_n}$-linear. They are also ``skew derivations" of the algebra $P_n$ for the automorphism $s_{k,\ell}$, in the sense that the satisfy the following skewed Leibniz formula:
	\begin{equation}\label{twistedder}
		\partial_{k,\ell}(fg) = \partial_{k,\ell}(f)g + s_{k,\ell}(f)\partial_{k,\ell}(g)
	\end{equation}
	for all $k\neq \ell \in \lbrace1,\ldots, n-1\rbrace$ and $f,g \in P_n$. The Demazure operators are all conjugate under the action of $\mathfrak S_n$: for $k\neq \ell$ and $\omega \in \mathfrak S_n$ we have $\omega \partial_{k,\ell} \omega^{-1} = \partial_{\omega(k),\omega(\ell)}$. We also have the following relations:
	\[
		\begin{array}{ll}
			\partial_k^2 = 0 & \text{ for all } k \in \lbrace1,\ldots n-1\rbrace, \\
			\partial_k\partial_{\ell} = \partial_{\ell}\partial_k & \text{ for all } k,\ell \in \lbrace1,\ldots n-1\rbrace \text{ such that } \vert k-\ell \vert >1, \\
			\partial_{k+1}\partial_k\partial_{k+1} = \partial_k\partial_{k+1}\partial_k & \text{ for all } k \in \lbrace1,\ldots n-2\rbrace.
		\end{array}
	\]
	
	\medskip
	
	\subsubsection{Affine nil Hecke algebras} We follow \cite[Subsection 3.1]{2km}.
	\begin{defi}\label{anha} The \textit{affine nil Hecke algebra} of rank $n$ is the $K$-algebra $H_n^0$ on the generators $x_1,\ldots,x_n$ and $\tau_1,\ldots,\tau_{n-1}$ subject to the following relations
	\begin{enumerate}
		\item $x_kx_{\ell}=x_{\ell}x_k$ for all $k,\ell \in \lbrace 1,\ldots,n\rbrace$,\\
		
		\item\label{tausq} $\tau_k^2=0$ for all $k \in \lbrace1,\ldots,n-1\rbrace$,\\
		
		\item\label{taux} $\tau_kx_{\ell} - x_{s_k(\ell)}\tau_k = \left \{ \begin{array}{rl}
				1 & \text{ if } \ell=k+1 \\
				-1 & \text{ if } \ell =k \\
				0 & \text{ otherwise}
			\end{array} \right.
			$ for all for all $k\in \lbrace1,\ldots,n-1\rbrace$ and $\ell \in \lbrace 1,\ldots,n\rbrace$,\\
			
		\item\label{taucom} $\tau_k\tau_{\ell} = \tau_{\ell}\tau_k$ for all $k,\ell \in \lbrace1,\ldots,n-1\rbrace$ such that $\vert k-\ell\vert >1$,\\
		
		\item\label{taubraid} $\tau_{k+1}\tau_k\tau_{k+1} - \tau_k\tau_{k+1}\tau_{k}=0$ for all $k \in \lbrace1,\ldots,n-2\rbrace$.
	\end{enumerate}
	\end{defi}
	The algebra $H_n^0$ is graded, with $x_{\ell}$ in degree 2 for all $\ell \in \lbrace 1,\ldots,n\rbrace$ and $\tau_k$ in degree $-2$ for all $k \in \lbrace1,\ldots,n-1\rbrace$. There is an isomorphism of graded $K$-algebras (\cite[Proposition~3.4]{2km})
	\begin{equation}\label{isonil}
		\left \{ \begin{array}{rcl}
			H_n^0 & \xrightarrow{\sim} & \mathrm{End}_{P_n^{\mathfrak S_n}}(P_n), \\
			x_k & \mapsto & x_k, \\
			\tau_k & \mapsto & \partial_k.
		\end{array} \right.
	\end{equation}	
	Given a finite sequence $\underline{k}=(k_1,\ldots,k_r)$ of elements of $\lbrace1,\ldots,n-1\rbrace$, we put $\tau_{\underline k} = \tau_{k_1}\ldots \tau_{k_r}$ and $\partial_{\underline k}=\partial_{k_1}\ldots\partial_{k_r}$. For all $\omega \in \mathfrak S_n$, we can define an element $\tau_{\omega} \in H_n^0$ in the following way: take $\underline{k}$ a reduced expression of $\omega$, and let $\tau_{\omega}=\tau_{\underline{k}}$. By relations (\ref{taucom}) and (\ref{taubraid}) in Definition \ref{anha}, this does not depend on the choice of the reduced expression. There is a similarly defined operator $\partial_{\omega}$ on $P_n$. By relation (\ref{tausq}) in Definition \ref{anha}, for all $\omega_1,\omega_2 \in \mathfrak S_n$ we have
	\[
		\tau_{\omega_1}\tau_{\omega_2} = \left \{ \begin{array}{ll}
			\tau_{\omega_1\omega_2} & \text{if } l(\omega_1\omega_2) = l(\omega_1)+l(\omega_2), \\
			0 & \text{otherwise.}
		\end{array} \right.
	\]
		
	Since $P_n$ is a free graded $P_n^{\mathfrak{S}_n}$-module of graded rank $q^{-\frac{n(n-1)}{2}}[n]!$, the $K$-algebra $H_n^0$ is a matrix algebra over $P_n^{\mathfrak S_n}$. Let $e_n =  x_2x_3^2\ldots x_n^{n-1}\tau_{\omega_0[1,n]}$. Then $e_n$ is a primitive idempotent of $H_n^0$, and there is an isomorphism of graded $(H_n^0,P_n^{\mathfrak S_n})$-bimodules
	\begin{equation}\label{en}
		\left \{
			\begin{array}{rcl}
				P_n & \xrightarrow{\sim} & q^{-n(n-1)}H_n^0e_n \\
				P & \mapsto & P\tau_{\omega_0[1,n]}
			\end{array}
		\right.
	\end{equation}
	From this, we deduce that there is an isomorphism of graded $(H_n^0,P_n^{\mathfrak S_n})$-bimodules
	\begin{equation}\label{divpow}
		H_n^0 \simeq q^{-\frac{n(n-1)}{2}}[n]! (H_n^0e_n).
	\end{equation}
	
	\medskip
		
	\subsection{KLR algebras}
	\subsubsection{Definitions} For all $i,j \in I$, we fix a polynomial $Q_{i,j}(u,v) \in K[u,v]$. Assume that this data satisfies the following conditions:
	\begin{itemize}
		\item for all $i \in I$, $Q_{i,i}(u,v) = 0$,
		\item for all $i,j \in I$, $Q_{i,j}(u,v) = Q_{j,i}(v,u)$,
		\item for all $i\neq j \in I$, there are some invertible elements $t_{i,j}$ and $t_{j,i}$ of $K$ such that $Q_{i,j}(u,v)$ has the form
		\[
			Q_{i,j}(u,v) \in t_{i,j}u^{-c_{i,j}} + t_{j,i}v^{-c_{j,i}} + \sum_{\substack{s,t >0 \\ d_is+d_jt=- i\cdot j}} Ku^sv^t.
		\]
	\end{itemize}
	The KLR algebras are a family of $K$-algebras attached to the data of these polynomials.

	\begin{defi}\label{klr} Let $\beta$ be an element of $Q^+$ of height $n$. The \textit{Khovanov-Lauda-Rouquier algebra} $H_{\beta}$ is the unital $K$-algebra with generators $1_{\nu}$ for $\nu \in I^{\beta}$, $x_1,\ldots, x_n$ and $\tau_1,\ldots,\tau_{n-1}$ subject to the following relations:
	\begin{enumerate}
		\item $1_{\nu}1_{\nu'} = \delta_{\nu,\nu'}1_{\nu}$ for all $\nu,\nu' \in I^{\beta}$, and $\sum_{\nu \in I^{\beta}} 1_{\nu} = 1_{\beta}$, where $1_{\beta}$ is the unit of $H_n$,\\
		
		\item $x_k1_{\nu} = 1_{\nu}x_k$ for all $\nu \in I^{\beta}$ and $k \in \lbrace1,\ldots,n\rbrace$,\\
		
		\item $\tau_k1_{\nu} = 1_{s_k(\nu)}\tau_k$ for all $\nu \in I^{\beta}$ and $k \in \lbrace1,\ldots,n-1\rbrace$,\\
		
		\item $x_kx_{\ell}=x_{\ell}x_k$ for all $k,\ell \in \lbrace1,\ldots,n\rbrace$, \\
		
		\item\label{Tausq} $\tau_k^21_{\nu} = Q_{\nu_k,\nu_{k+1}}(x_k,x_{k+1})1_{\nu}$ for all $\nu \in I^{\beta}$ and $k \in \lbrace1,\ldots,n-1\rbrace$,\\
		
		\item\label{Taux} $\left(\tau_k x_{\ell} - x_{s_k(\ell)}\tau_k\right)1_{\nu} = \left \{ \begin{array}{ll}
			1_{\nu} & \text{ if } \ell=k+1 \text{ and } \nu_k=\nu_{k+1} \\
			-1_{\nu} & \text{ if } \ell=k \text{ and } \nu_k=\nu_{k+1}\\
			0 & \text{ otherwise }
			\end{array} \right.$ for all $k \in \lbrace1,\ldots,n-1\rbrace$, $\ell\in \lbrace1,\ldots n\rbrace$ and $\nu \in I^{\beta}$,\\
			
		\item $\tau_k\tau_{\ell} = \tau_{\ell}\tau_k$ for all $k,\ell \in \lbrace1,\ldots,n-1\rbrace$ such that $\vert k-\ell \vert >1$,\\
		
		\item\label{Taubraid} $\left(\tau_{k+1}\tau_k\tau_{k+1} - \tau_k\tau_{k+1}\tau_k\right)1_{\nu} = \delta_{\nu_k,\nu_{k+2}} \partial_{k,k
			+2}\left(Q_{\nu_k,\nu_{k+1}} (x_{k+2},x_{k+1}) \right)1_{\nu}$ for all $k\in \lbrace1,\ldots,n-2\rbrace$ and $\nu \in I^{\beta}$.
	\end{enumerate}
	\end{defi}

	\begin{conv} We will number components of tuples $\nu \in I^{n}$ from the right to the left, to match with the graphical calculus interpretation of the KLR algebras. Namely, we will write $\nu = (\nu_n,\ldots,\nu_1)$ for $\nu \in I^n$ and we let the symmetric group $\mathfrak{S}_n$ act on $I^n$ accordingly. For instance, if $\nu = (c,b,a) \in I^3$, we have $s_1(\nu) = (c,a,b)$ and $s_2(\nu) = (b,c,a)$.
	\end{conv}

	For a sequence $\underline{k} = (k_1,\ldots,k_r)$ of elements of $\lbrace1,\ldots,n-1\rbrace$, we put $\tau_{\underline{k}}=\tau_{k_1}\ldots\tau_{k_r}$. We also define a (possibly non-unital) $K$-algebra $H_n$ by
	\[
		H_n = \bigoplus_{\substack{\beta \in Q^+ \\ \vert \beta \vert =n}} H_{\beta}.
	\]
	The algebra $H_n$ is unital only when $I$ is finite. There is a grading on $H_n$ defined in the following way:
	\begin{itemize}
		\item $1_{\nu}$ is in degree 0,
		\item $x_k1_{\nu}$ is in degree $\nu_k\cdot\nu_k$,
		\item $\tau_{\ell}1_\nu$ is in degree $-\nu_l\cdot \nu_{l+1}$,
	\end{itemize}
	for all $\nu \in I^{n}$, $k \in \lbrace 1,\ldots, n\rbrace$ and $\ell \in \lbrace1,\ldots,n-1 \rbrace$. There is an anti-automorphism $\mathrm{rev}_{n}$ of $H_{n}$ defined by
	\begin{align*}
	\mathrm{rev}_{n}\big(1_{(\nu_n,\ldots,\nu_1)}\big) = 1_{(\nu_1,\ldots,\nu_n)}, \quad		
	\mathrm{rev}_{n}(x_k) = x_{n-k}, \quad
	\mathrm{rev}_{n}(\tau_{\ell}) = \tau_{n-1-{\ell}},
	\end{align*}
	for all $(\nu_n,\ldots,\nu_1) \in I^{n}$, $k\in \lbrace 1,\ldots,n \rbrace$ and $\ell \in \lbrace1,\ldots,n-1\rbrace$. Hence $H_{n} \simeq H_{n}^{\mathrm{op}}$.
	
	Let
	\[
		\mathcal{H} = \bigoplus_{\beta \in Q^+}H_{\beta}\mathrm{-mod}
	\]
	be the category of finitely generated graded modules over the KLR algebras. The morphisms in $\mathcal H$ are given by degree preserving module maps. The category $\mathcal H$ is $K$-linear and abelian. Furthermore it is graded, in the sense that it is equipped with the grading shift functor $M \mapsto qM$.

	\subsubsection{PBW theorem} The KLR algebras satisfy a PBW type theorem.	
	\begin{thm}[{\cite[Theorem~3.7]{2km}}]\label{pbw}
		For all $\beta \in Q^+$, the $K$-algebra $H_{\beta}$ is free as a $K$-module. Let $n=\vert \beta\vert$ and let $T$ be a complete set of reduced decompositions of elements of $\mathfrak S_{n}$. Then the following sets are bases over $K$ for $H_{\beta}$:
		\begin{align*}
			&\left \{  \tau_{\underline{\omega}}x_1^{a_1}\ldots x_n^{a_n} 1_{\nu} \vert \, \nu \in I^{\beta}, \, a_1,\ldots,a_n \geqslant 0, \, \underline{\omega}\in T \right \},\\
			&\left \{ x_1^{a_1}\ldots x_n^{a_n} \tau_{\underline{\omega}}1_{\nu} \vert \, \nu \in I^{\beta}, \, a_1,\ldots,a_n \geqslant 0, \, \underline{\omega} \in T \right \}.
		\end{align*}
	\end{thm}

	Let $P_{\beta}$ be the subalgebra of $H_{\beta}$ generated by $x_1,\ldots,x_n$ and $1_{\nu}$ for $\nu \in I^{\beta}$. As a consequence of Theorem \ref{pbw}, there is an isomorphism of algebras
	\[
		P_{\beta} \simeq \bigoplus_{\nu \in I^{\beta}} P_n1_{\nu}.
	\]
	Furthermore, $H_{\beta}$ is free of rank $n!$ as a left (or right) module over $P_{\beta}$, with basis $\lbrace \tau_{\underline{\omega}} \vert \, \underline{\omega} \in T \rbrace$, for $T$ any set of reduced expressions for the elements of $\mathfrak S_n$.
	
	\subsubsection{Center} Let $\beta \in Q^+$ of height $n$. The symmetric group $\mathfrak S_{n}$ acts naturally on $I^{\beta}$ by permutation of the components. Thus it also acts on the algebra $P_{\beta}$ by permuting $x_1,\ldots,x_{n}$ and the $1_{\nu}$, $\nu \in I^{\beta}$.
	
	\begin{prop}[{\cite[Proposition 3.9]{2km}}]\label{center}
		The center of $H_{\beta}$ is $P_{\beta}^{\mathfrak S_{n}}$.
	\end{prop}

	\medskip
		
	\subsection{Monoidal structure} The category $\mathcal H$ has a monoidal structure, which we describe now.
	\subsubsection{Inductions between KLR algebras} Assume that $I$ is finite. For $n,m \geqslant 0$, there is a morphism of $K$-algebras $r_{n}^{n+m} : H_{n} \rightarrow H_{n+m}$ (called \textit{right inclusion}) defined by
	\[
		r_{n}^{n+m}\left(1_{\nu}\right) = \sum_{\mu \in I^{m}} 1_{\mu\nu}, \quad
		r_{n}^{n+m}\left(x_k\right) = x_k, \quad
		r_{n}^{n+m}\left(\tau_{\ell}\right) = \tau_{\ell},
	\]
	for all $\nu \in I^{n}$, $k \in \lbrace1,\ldots,n\rbrace$ and $\ell \in \lbrace1,\ldots,n-1\rbrace$. Here, $\mu\nu$ denotes the concatenation of $\mu \in I^m$ with $\nu \in I^n$. We also have a morphism of $K$-algebras $l_{n}^{n+m}: H_n \rightarrow H_{n+m}$ (called \textit{left inclusion}) defined by
	\[
		l_{n}^{n+m}\left(1_{\nu}\right) = \sum_{\mu \in I^{m}} 1_{\nu\mu}, \quad
		l_{n}^{n+m}\left(x_k\right) = x_{m+k}, \quad
		l_{n}^{n+m}\left(\tau_{\ell}\right) = \tau_{m+\ell},
	\]
	for all $\nu \in I^{n}$, $k \in \lbrace1,\ldots,n\rbrace$ and $\ell \in \lbrace1,\ldots,n-1\rbrace$. By Theorem \ref{pbw}, right and left inclusion are injective. Let $H_{m,n} = H_{m} \otimes H_{n}$. There is an injective morphism of $K$-algebras
	\[
		\left \{ \begin{array}{rcl}
			H_{m,n} & \rightarrow & H_{n+m}, \\
			y\otimes z & \mapsto & y \diamond z = l_{m}^{n+m}(y)r_{n}^{n+m}(z).
		\end{array} \right.			
	\]
	This defines an associative binary operation $\diamond$ on $\bigoplus_{k \geqslant 0} H_{k}$ with unit element $1_0 \in H_0 = K$. This morphism also endows $H_{n+m}$ with a structure of $H_{m,n}$-bimodule. By Theorem \ref{pbw}, $H_{n+m}$ is free of rank $\frac{(n+m)!}{n!m!}$ as a left (resp. right) $H_{m,n}$-module, and there are decompositions
	\begin{equation}\label{coset}
		H_{n+m} = \bigoplus_{\underline{\omega} \in T} H_{m,n}\tau_{\underline{\omega}}= \bigoplus_{\underline{\omega} \in T'} \tau_{\underline{\omega}} H_{m,n},
	\end{equation}
	where $T$ (resp. $T'$) is a complete set of reduced decompositions for minimal length representatives of right (resp. left) cosets of $\mathfrak S_{m}\times \mathfrak S_{n}$ in $\mathfrak S_{n+m}$. Hence there is a biexact bifunctor defined by
	\[
		\left \{ \begin{array}{rcl}
			H_{m}\mathrm{-mod} \times H_{n}\mathrm{-mod} & \rightarrow & H_{n+m}\mathrm{-mod}, \\
			(M,N) & \mapsto & MN = H_{n+m}\otimes_{H_{m,n}} \left(M\otimes N\right).
		\end{array} \right.
	\]	
	This defines the monoidal structure of $\mathcal H$.\\
	
	It will be convenient to work with the algebras $H_{\beta}$ rather than with $H_n$. Let us describe the monoidal structure at their level. For $\beta_1,\ldots,\beta_r \in Q^+$, we put $1_{\beta_1,\ldots,\beta_r} = 1_{\beta_1} \diamond \ldots \diamond 1_{\beta_r}$, an idempotent of $H_{\beta_1+\ldots+\beta_r}$.	Let $\beta,\gamma \in Q^+$ of respective heights $m,n$, and let $H_{\beta,\gamma}=H_{\beta}\otimes H_{\gamma}$. There is a non-unital morphism of $K$-algebras
	\[
		\left \{ \begin{array}{rcl}
		H_{\beta,\gamma} &\rightarrow &H_{\beta+\gamma},\\
		y\otimes z &\mapsto &y \diamond z.
		\end{array}\right.
	\]
	The image of the unit of $H_{\beta,\gamma}$ by this morphism is the idempotent $1_{\beta,\gamma}$ of $H_{\beta+\gamma}$. This endows $H_{\beta+\gamma}1_{\beta,\gamma}$ (resp. $1_{\beta,\gamma}H_{\beta+\gamma}$) with a structure of $(H_{\beta+\gamma},H_{\beta,\gamma})$-bimodule (resp. $(H_{\beta,\gamma},H_{\beta+\gamma})$-bimodule). It is free as a right (resp. left) $H_{\beta,\gamma}$-module, with basis $\lbrace\tau_{\underline{\omega}}, \, \underline{\omega} \in T \rbrace$ (resp. $\lbrace\tau_{\underline{\omega}}, \, \underline{\omega} \in T' \rbrace$) with $T$ (resp. $T'$) as above. If $M \in H_{\beta}\mathrm{-mod}$ and $N\in H_{\gamma}\mathrm{-mod}$, then $MN$ is an $H_{\beta+\gamma}$ module given by
	\[
		MN = H_{\beta+\gamma}1_{\beta,\gamma}\otimes_{H_{\beta,\gamma}} (M\otimes N).
	\]
	This also defines the monoidal structure in the case of $I$ infinite.
	
	Finally, we describe some useful induction and restriction functors. For $i\in I$, we denote by $E_i$ the free $H_i$-module of rank 1. The \textit{left $i$-induction functor} is the functor defined by
	\[
		E_i(-) : \left \{ \begin{array}{rcl}
			\mathcal H & \rightarrow & \mathcal H \\
			M & \mapsto & E_iM
		\end{array} \right.
	\]
	It is exact. Since $E_i$ is free of rank 1, for $M \in H_{\beta}\mathrm{-mod}$ we simply have
	\[
		E_iM \simeq H_{\beta+i}1_{i,\beta} \otimes_{H_{\beta}}M.
	\]
	The \textit{left $i$-restriction functor} $F_i$ is the right adjoint to the left $i$-induction functor. For $M \in H_{\beta}\mathrm{-mod}$, $F_i(M) = 1_{i,\beta-i}M$ (viewed as an $H_{\beta-i}$-module, via the right inclusion $H_{\beta-i} \rightarrow 1_{i,\beta-i}H_{\beta}1_{i,\beta-i}$). When $\beta-i \notin Q^+$, we understand $1_{i,\beta-i}=0$. The functor $F_i$ is also exact. There is a similarly defined pair of adjoint functors called right $i$-induction and right $i$-restriction, but they will not be as useful to us.
	
	\subsubsection{Divided powers} Let $n\geqslant 0$ and $i\in I$. The algebra $H_{ni}$ is isomorphic to the affine nil Hecke algebra $H_n^0$ (up to a grading dilatation by $d_i$). The $H_{ni}$-module $E_i^n$ is free of rank 1. Its endomorphism ring (as an $H_{ni}$-module without grading) is $H_{ni}^{\mathrm{opp}}$, acting by right multiplication. Recall the primitive idempotent $e_n = x_2\ldots x_{n}^{n-1}\tau_{\omega_0[1,n]}$ of $H_n^0$. We define the $n^{\text{th}}$ divided power $E_i^{(n)}$ by
	\[
		E_i^{(n)} = q_i^{-\frac{n(n-1)}{2}}E_i^ne_n.
	\]
	As a graded $H_{ni}$-module, we have simply $E_i^{(n)} \simeq q_i^{-n(n-1)}P_{ni}$ by the isomorphism (\ref{en}). By the structure theory of the affine nil Hecke algebra (\ref{divpow}) there is an isomorphism of graded $H_{ni}$-modules
	\[
		E_i^{n} \simeq [n]_i! E_{i}^{(n)}.
	\]
	
	\medskip
	
	\subsection{Categorified quantum groups} Consider the full subcategory $\mathcal H\mathrm{-proj}$ of $\mathcal H$ consisting of projective modules. This is an additive and graded category. The \textit{split Grothendieck group} $K_0(\mathcal H\mathrm{-proj})$ is the quotient of the free abelian group on isomophism classes $[M]$ of objects $M \in \mathcal H\mathrm{-proj}$ by the relations $[M]+[N] = [M\oplus N]$ for all $M,N \in \mathcal H\mathrm{-proj}$. Since $\mathcal H\mathrm{-proj}$ is graded, we can endow $K_0(\mathcal H\mathrm{-proj})$ with a structure of $\mathbb{Z}[q,q^{-1}]$-module as follows: for $p \in \mathbb{Z}[q,q^{-1}]$ and $M \in \mathcal H\mathrm{-proj}$, we let $p[M] = [pM]$. Furthermore, there is a product on $K_0(\mathcal{H}\mathrm{-proj})$ induced from the moinoidal structure: if $M,N \in \mathcal H\mathrm{-proj}$, we define $[M][N] = [MN]$. Hence $K_0(\mathcal H\mathrm{-proj})$ has the structure of a $\mathbb{Z}[q,q^{-1}]$-algebra. There is also a non-degenerate bilinear form on $K_0(\mathcal H\mathrm{-proj})$ defined as follows: for $M,N \in \mathcal{H}\mathrm{-proj}$ let
	\[
		\left([M],[N]\right) = \sum_{k\in \mathbb{Z}} \mathrm{rk}(\mathrm{Hom}_{\mathcal H}(M,q^kN)) q^{k}.
	\]
	Khovanov and Lauda proved that $\mathcal{H}\mathrm{-proj}$ categorifies half of the integral quantum group $_{\mathcal A}U^+$.
	\begin{thm}[{\cite[Theorem~8]{catII}}]
		Assume that $K$ is a field. Then there is a unique isomorphism of $\mathbb Z[q,q^{-1}]$-algebras
		\[
		 _{\mathcal A}U^+ \xrightarrow{\sim} K_0(\mathcal H\mathrm{-proj})
		\]
		sending $e_i^{(n)}$ to $\big[E_i^{(n)}\big]$ for all $i \in I$ and $n \geqslant 0$. Furthermore, this isomorphism is an isometry for the inner products on $_{\mathcal A}U^+$ and $K_0(\mathcal H\mathrm{-proj})$.
	\end{thm}

	\medskip
	
	\subsection{Some computational lemmas} We now recall some elementary results in KLR algebras that will be used in the proofs below. First, we recall some basic results about computations in affine nil Hecke algebras.
	
	\begin{lem}\label{compnil}\begin{enumerate}
			\item Let $n\geqslant 0$, let $k \in \lbrace1,\ldots,n-1\rbrace$ and let $P\in P_n$. In $H_n^0$ we have
			\[
			\tau_kP - s_k(P)\tau_k = \partial_k(P).
			\]
			\item Let $ n\geqslant 0$, let $\underline{k}$ be a finite sequence of elements of $\lbrace 1,\ldots, n-1\rbrace$ and let $P\in P_n$. In $H_n^0$ we have
			\[
			\tau_{\underline{k}} P\tau_{\omega_0[1,n]} = \partial_{\underline{k}}(P)\tau_{\omega_0[1,n]}.
			\]
		\end{enumerate}
	\end{lem}
	
	\begin{proof}
		For statement (1), consider the two maps
		\[
		\left \{ \begin{array}{rcl} P_n & \rightarrow & H_n^0 \\ P & \mapsto &P\tau_k - s_k(P)\tau_k \end{array} \right., \
		\left \{ \begin{array}{rcl} P_n & \rightarrow & H_n^0 \\ P & \mapsto & \partial_k(P) \end{array} \right..
		\]
		They are both skew derivations, in the sense that they satisfy relation (\ref{twistedder}) with respect to the automorphism $s_k$. Furthermore, by relation (\ref{taux}) in Definition \ref{anha}, they coincide on $x_1,\ldots,x_n$, which are algebra generators of $P_n$. Hence, they are equal and we have $\tau_kP - s_k(P)\tau_k = \partial_k(P)$ for all $P \in P_n$.
		
		For statement (2), by induction on the length of the finite sequence $\underline{k}$, it suffices to treat the case where $\underline{k}$ has length 1. If $k \in \lbrace1,\ldots,n-1\rbrace$, then for all $P \in P_n$ we have
		\begin{align*}
		\tau_k P\tau_{\omega_0[1,n]} &= s_k(P)\tau_k \tau_{\omega_0[1,n]} + \partial_k(P)\tau_{\omega_0[1,n]}
		\end{align*}
		by statement (1) of the lemma. However, $\tau_k\tau_{\omega_0[1,n]} =0$ since $l(s_k\omega_0[1,n]) < l(\omega_0[1,n])$. Hence we have $\tau_k P\tau_{\omega_0[1,n]}=\partial_k(P)\tau_{\omega_0[1,n]}$, which completes the proof.
	\end{proof}
	
	\begin{lem}\label{xtaucom}
		Let $m,n \geqslant 0$, and let $T$ be a complete set of reduced expressions of minimal length representatives of left cosets of $\mathfrak S_m\times \mathfrak S_n$ in $\mathfrak S_{m+n}$. Let $\underline{\omega}_0$ be the longest element in $T$. Then for all $y \in H_n$ and $z \in H_{m}$ we have
		\[
			(y\diamond z)\tau_{\underline{\omega}_0} - \tau_{\underline{\omega}_0}(z\diamond y) \in \sum_{\underline{\omega} \in T \setminus \lbrace \underline{\omega}_0 \rbrace} \tau_{\underline{\omega}} H_{m,n}.
		\]
	\end{lem}
	
	\begin{proof}
		There is a filtration of $H_{n+m}$ with $1_{\nu}$, $x_{\ell}$ in degree 0 for all $\nu \in I^{n+m}$ and $\ell \in \lbrace1,\ldots,n+m\rbrace$, and $\tau_{k}$ in degree 1 for all $k\in \lbrace1,\ldots,n+m-1\rbrace$. In the associated graded $\mathrm{gr}(H_{n+m})$, relations (\ref{Taux}) and (\ref{Taubraid}) of Definition \ref{klr} become
		\[
			\tau_kx_{\ell} - x_{s_k(\ell)}\tau_k=0, \quad
			\tau_{k+1}\tau_k\tau_{k+1} - \tau_k\tau_{k+1}\tau_k=0.
		\]
		From these, we deduce that in $\mathrm{gr}(H_{n+m})$, we have
		\begin{equation}\label{assgr}
			(y\diamond z)\tau_{\underline{\omega}_0} - \tau_{\underline{\omega}_0}(z\diamond y) =0
		\end{equation}
		for all $y \in \lbrace x_1,\ldots, x_n,\tau_1,\ldots,\tau_{n-1},1_{\nu}, \nu \in I^n\rbrace$ and $z\in \lbrace x_1,\ldots,x_m,\tau_1,\ldots,\tau_{m-1}, 1_{\nu}, \nu \in I^m \rbrace$. However, the linear span of the elements $y\otimes z \in H_{n,m}$ for which equation (\ref{assgr}) holds in $\mathrm{gr}(H_{n+m})$ is a subalgebra of $H_{n,m}$. So equation (\ref{assgr}) holds in $\mathrm{gr}(H_{n+m})$ for all $y\otimes z \in H_{n,m}$. Hence for all $y\otimes z \in H_{n,m}$ we have
		\[
			(y\diamond z)\tau_{\underline{\omega}_0} - \tau_{\underline{\omega}_0}(z\diamond y)\in\sum_{l(\underline{\omega})<l(\underline{\omega}_0)} \tau_{\underline{\omega}}(H_{n+m})_{0}
		\]
		where $(H_{n+m})_{0}$ is the piece of the filtration in degree 0 (that is, $(H_{n+m})_{0}=\oplus_{\nu\in I^{n+m}}P_{n+m}1_{\nu}$). Using decomposition (\ref{coset}), we can write the right hand side as wanted.
	\end{proof}
	
	We now generalize relations (\ref{Tausq}) and (\ref{Taubraid}) from Definition \ref{klr}. For this, we need to introduce some notation. For $i\in I$ and $\nu=(\nu_1,\ldots,\nu_n) \in I^{n}$, we define
	\[
		Q_{i,\nu}(u,v_1,\ldots,v_n) = \prod_{\substack{1\leqslant k \leqslant n \\ \nu_k \neq i}} Q_{i,\nu_k}(u,v_k) \in K[u,v_1\ldots,v_n].
	\]
	These polynomials generalize the polynomials $Q_{i,j}(u,v)$, and will appear many times in the following sections.
	
	\begin{lem}\label{multistrand}
		Let $n \geqslant 1$, let $i \in I$ and let $\nu \in I^n$ be such that $\nu_k \neq i$ for all $k \in \lbrace1,\ldots,n\rbrace$. In $H_{n+1}$ the following relations hold:
		\begin{enumerate}
			\item $\tau_{[n\downarrow 1]}\tau_{[1\uparrow n]} 1_{i,\nu} = Q_{i,\nu}(x_{n+1},x_1,\ldots,x_n) 1_{i,\nu}$,\\
			
			\item $\tau_{[1\uparrow n]} \tau_{[n\downarrow1]} 1_{\nu,i} = Q_{i,\nu}(x_1,x_{2},\ldots,x_{n+1}) 1_{\nu,i}$.\\
			
		\end{enumerate}
		In $H_{n+2}$ the following relation holds:
		\begin{enumerate}
			\setcounter{enumi}{2}
			\item $\left(\tau_{[n+1\downarrow2]}\tau_1\tau_{[2\uparrow n+1]} - \tau_{[1\uparrow n]}\tau_{n+1}\tau_{[n\downarrow 1]}\right)1_{i,\nu,i} = \partial_{1,n+2}\left(Q_{i,\nu}(x_{n+2},x_{2},\ldots,x_{n+1})\right)1_{i,\nu,i}$.
		\end{enumerate}
	\end{lem}

	\begin{proof}
		Relation (2) follows from relation (1) by applying the anti-automorphism $\mathrm{rev}_{n+1}$, so it suffices to prove (1) and (3). The proof is by induction on $n$. The case $n=1$ is just the statement of relations (\ref{Tausq}) and (\ref{Taubraid}) in Definition \ref{klr}.
		
		Assume that relation (1) is proved for $\nu$ and consider $\nu'=(j,\nu)$ for some $j\neq i$. We have
		\begin{align*}
			\tau_{[n+1\downarrow 1]}\tau_{[1\uparrow n+1]} 1_{i,\nu'} &= \tau_{n+1}\tau_{[n\downarrow 1]}\tau_{[1\uparrow n]}\tau_{n+1} 1_{i,j,\nu} \\
			&= \tau_{n+1} Q_{i,\nu}(x_{n+1},x_1,\ldots,x_n) \tau_{n+1} 1_{i,j,\nu} 
		\end{align*}
		where the second equality comes from the inductive assumption. Using relation (\ref{Taux}) from Definition \ref{klr}, we can commute the polynomial $Q_{i,\nu}(x_{n+1},x_1,\ldots,x_n)$ to the left of $\tau_{n+1}$ and we get
		\begin{align*}
		\tau_{[n+1\downarrow 1]}\tau_{[1\uparrow n+1]} 1_{i,\nu'} &= Q_{i,\nu}(x_{n+2},x_1,\ldots,x_n) \tau_{n+1}^21_{i,j,\nu} \\
			&= Q_{i,\nu}(x_{n+2},x_1,\ldots,x_n)Q_{i,j}(x_{n+2},x_{n+1}) 1_{i,j,\nu} \\
			&= Q_{i,\nu'}(x_{n+2},x_1,\ldots,x_{n+1})1_{i,\nu'},
		\end{align*}
		where the second equality is relation (\ref{Tausq}) of Definition \ref{klr}. This proves relation (1), and hence also relation (2).
		
		Assume that relation (3) is proved for $\nu$ and consider $\nu'=(j,\nu)$ for some $j\neq i$. We have
		\begin{align*}
			\tau_{[n+2\downarrow2]}\tau_1\tau_{[2\uparrow n+2]} 1_{i,\nu',i} &= \tau_{n+2}\tau_{[n+1\downarrow2]}\tau_1\tau_{[2\uparrow n+1]}\tau_{n+2}  1_{i,\nu',i} \\
			&= \tau_{n+2}\left(\tau_{[1\uparrow n]}\tau_{n+1}\tau_{[n\downarrow 1]} +  \partial_{1,n+2}\left(Q_{i,\nu}(x_{n+2},x_{2},\ldots,x_{n+1})\right) \right)\tau_{n+2}  1_{i,\nu',i}
		\end{align*}
		the second equality coming from the inductive assumption. Let us simplify each of the two terms. For the term $\tau_{n+2}\tau_{[1\uparrow n]}\tau_{n+1}\tau_{[n\downarrow 1]} \tau_{n+2}1_{i,\nu',i}$, both factors $\tau_{n+2}$ can be commuted across the factors $\tau_{[1\uparrow n]}$ and $\tau_{[n\downarrow 1]}$ using relation (\ref{Taux}) from Definition \ref{klr}, which yields 
		\begin{align*}
			\tau_{n+2}\tau_{[1\uparrow n]}\tau_{n+1}\tau_{[n\downarrow 1]} \tau_{n+2}1_{i,\nu',i}& = \tau_{[1\uparrow n]}\tau_{n+2}\tau_{n+1}\tau_{n+2}\tau_{[n\downarrow 1]}1_{i,\nu',i} \\
			&= \tau_{[1\uparrow n]}\left(\tau_{n+1}\tau_{n+2}\tau_{n+1} + \partial_{n+1,n+3}(Q_{i,j}(x_{n+3},x_{n+2})) \right)\tau_{[n\downarrow 1]}1_{i,\nu',i} \\
			&= \tau_{[1\uparrow n+1]}\tau_{n+2} \tau_{[n\downarrow 1]}1_{i,\nu',i} + \partial_{1,n+3}(Q_{i,j}(x_{n+3},x_{n+2})) \tau_{[1\uparrow n]}\tau_{[n\downarrow 1]}1_{i,\nu',i} \\
			&= \tau_{[1\uparrow n+1]}\tau_{n+2} \tau_{[n\downarrow 1]}1_{i,\nu',i} + \partial_{1,n+3}(Q_{i,j}(x_{n+3},x_{n+2}))Q_{i,\nu}(x_1,x_{2},\ldots,x_{n+1})1_{i,\nu',i}
		\end{align*}
		The second equality comes from relation (\ref{Taubraid}) in Definition \ref{klr}. The third is deduced by commuting the polynomial $\partial_{n+1,n+3}(Q_{i,j}(x_{n+3},x_{n+2})$ to the left of $\tau_{[1\uparrow n]}$ using relation (\ref{Taux}) in Definition \ref{klr}. Finally for the last equality we use relation (2) of the lemma which was proved above to compute $\tau_{[1\uparrow n]}\tau_{[n\downarrow 1]}$.
				
		Now for the term $\tau_{n+2}\partial_{1,n+2}\left(Q_{i,\nu}(x_{n+2},x_{2},\ldots,x_{n+1})\right)\tau_{n+2}1_{i,\nu',i}$, we can commute the polynomial to the left of the factor $\tau_{n+2}$ by relation (\ref{Taux}) in Definition \ref{klr} to obtain
		\begin{align*}
			\tau_{n+2}\partial_{1,n+2}\left(Q_{i,\nu}(x_{n+2},x_{2},\ldots,x_{n+1})\right)\tau_{n+2}1_{i,\nu',i} &= s_{n+2}\left(\partial_{1,n+2}\left(Q_{i,\nu}(x_{n+2},x_{2},\ldots,x_{n+1})\right)\right) \tau_{n+2}^2 1_{i,\nu',i}  \\
			&= \partial_{1,n+3}\left(Q_{i,\nu}(x_{n+3},x_{2},\ldots,x_{n+1})\right) Q_{i,j}(x_{n+3},x_{n+2}) 1_{i,\nu',i}.
		\end{align*}
		For the last equality, we have used the conjugation of the Demazure operators to rewrite the polynomial, and relation (\ref{Tausq}) in Definition \ref{klr} to simplify $\tau_{n+2}^2$. With the two terms written as such, we deduce
		\begin{align*}
			\left(\tau_{[n+2\downarrow2]}\tau_1\tau_{[2\uparrow n+2]} - \tau_{[1\uparrow n+1]}\tau_{n+2} \tau_{[n\downarrow 1]}\right)1_{i,\nu',i}	&=\partial_{1,n+3}(Q_{i,j}(x_{n+3},x_{n+2}))Q_{i,\nu}(x_1,x_{2},\ldots,x_{n+1})1_{i,\nu',i} \\ & \quad +  \partial_{1,n+3}\left(Q_{i,\nu}(x_{n+3},x_{2},\ldots,x_{n+1})\right) Q_{i,j}(x_{n+3},x_{n+2}) 1_{i,\nu',i} \\
			&= \partial_{1,n+3} \left( Q_{i,j}(x_{n+3},x_{n+2})Q_{i,\nu}(x_{n+3},x_{2},\ldots,x_{n+1}) \right)1_{i,\nu',i} \\
			&= \partial_{1,n+3} \left( Q_{i,\nu'}(x_{n+3},x_2,\ldots,x_{n+2}\right)1_{i,\nu',i}
		\end{align*}
		where the second equality is the fact that $\partial_{1,n+3}$ is a skewed derivation (\ref{twistedder}). This concludes the proof of (3).
	\end{proof}
	
	\bigskip
	
	\section{Categorification of the adjoint action of $E_i$}\label{cat} We now fix $i \in I$, and explain how to categorify the action of $U_i$ on $U^+[i]$. We start by introducing the category $\mathcal H[i]$, our categorical analogue of $U^+[i]$. Then, we define an endofunctor $\mathrm{ad}_{E_i}$ of $\mathcal H[i]$, which categorifies the action of $\mathrm{ad}_{e_i}$ on $U^+[i]$. In the first two subsections, we prove some elementary properties of these objects. In particular, we prove that the affine nil Hecke algebra acts on powers of $\mathrm{ad}_{E_i}$ in Proposition \ref{adei}. Another important point is the realization of $\mathrm{ad}_{E_i}$ as the cokernel of a natural transformation $\tau_{E_i,(-)}$, which we construct in \ref{tau}. This sets up the necessary preliminaries for Theorem \ref{tauinj} which states that $\tau_{E_i,(-)}$ is injective, providing the key short exact sequence for the study of $\mathrm{ad}_{E_i}$. In Subsection \ref{mainthm}, we give various corollaries of this theorem, such as the exactness of the functor $\mathrm{ad}_{E_i}$ and a compatibility of $\mathrm{ad}_{E_i}$ with the monoidal structure. We give the proof of Theorem \ref{tauinj} in Subsection \ref{proof}.

	\subsection{The category $\mathcal H[i]$} 
	\subsubsection{Definitions}
	\begin{defi} For $\beta \in Q^+$, let $H_{\beta}^i$ be the graded $K$-algebra defined by $H_{\beta}^{i}=H_{\beta}/(1_{\beta-i,i})$, where $(1_{\beta -i,i})$ is the two-sided ideal of $H_{\beta}$ generated by $1_{\beta-i,i}$ (when $\beta-i \notin Q^+$, we put $1_{\beta-i,i}=0$). For $n \geqslant 0$, let $H_{n}^i$ be the graded $K$-algebra defined by
			\[
				H_{n}^i = \bigoplus_{\substack{\beta \in Q^+ \\ \vert\beta\vert =n}} H_{\beta}^i.
			\]
	We define the category $\mathcal H[i]$ by
			\[
				\mathcal H[i] = \bigoplus_{\beta \in Q^+} H_{\beta}^i\mathrm{-mod}.
			\]
	\end{defi}

	Equivalently, for $n \geqslant 1$ we can define $H_n^i$ as the quotient of $H_n$ by the two sided ideal generated by the elements $1_{\nu i}$ where $\nu \in I^{n-1}$. If $I$ is finite and $n,m\geqslant 0$, the right inclusion $r_{n}^{n+m} : H_{n} \rightarrow H_{n+m}$ satisfies
	\[
		r_{n}^{n+m}(1_{\nu i}) = \sum_{\mu \in I^m} 1_{\mu\nu i} \in \left(1_{\rho i}\right)_{\rho \in I^{n+m-1}}.
	\]
	Hence we have an induced morphism $r_{n}^{n+m} : H_{n}^i \rightarrow H_{n+m}^i$. The left inclusions however do not induce maps at the level of the algebras $H_n^i$.
	
	We view $\mathcal H[i]$ as a full subcategory of $\mathcal H$. If $M \in H_{\beta}\mathrm{-mod}$, $M$ is an object of $\mathcal H[i]$ if and only if $1_{\beta-i,i}M=0$.
	
	\begin{ex}
		\begin{enumerate}
			\item  If $j \in I \setminus \lbrace i \rbrace$, we have $H_{i+j}^i \simeq K[x_1,x_2]/Q_{i,j}(x_2,x_1)$. In particular, $H_{i+j}^i$ is either zero (in the case where $i\cdot j=0$) or not finitely generated as a $K$-module. We will see below that this generalizes to every $H_{\beta}^i$ for $\beta\in Q^+$. This contrasts with cyclotomic KLR algebras, which are finitely generated as $K$-modules.
			
			\item If $j \in I \setminus \lbrace i \rbrace$, we have $E_j \in \mathcal H[i]$. Furthermore, if $M \in H_{\beta}^i\mathrm{-mod}$, we have $1_{j,\beta-j}M \in \mathcal H[i]$ since $1_{\beta-i,i}1_{j,\beta-j}M = 1_{j,\beta-j}1_{\beta-i,i}M=0$.
		\end{enumerate}
	\end{ex}	
	
	We say that a subcategory of a graded abelian category is a \textit{Serre subcategory} if it is non-empty, full and closed under subquotients, extensions and degree shifts.
	
	\begin{prop}
		The category $\mathcal{H}[i]$ is a Serre and monoidal subcategory of $\mathcal H$.
	\end{prop}

	\begin{proof}
		Let $\beta \in Q^+$. The functor
		\[
			\left \{ \begin{array}{rcl}
				H_{\beta}\mathrm{-mod} & \mapsto & 1_{\beta-i,i}H_{\beta}1_{\beta-i,i}\mathrm{-mod} \\
				M & \mapsto & 1_{\beta-i,i}M
			\end{array} \right.
		\]
		is exact. Hence its kernel is a Serre subcategory of $H_{\beta}\mathrm{-mod}$. If follows that $\mathcal{H}[i]$ is a Serre subcategory of $\mathcal H$.
				
		Let $M \in H_{\beta}^i\mathrm{-mod}$ and $N \in H_{\gamma}^i\mathrm{-mod}$, where $\beta,\gamma \in Q^+$ have respective heights $m,n$. Then by (\ref{coset})
		\[
			MN = \bigoplus_{\underline{\omega} \in T} \tau_{\underline{\omega}}1_{\beta,\gamma} \otimes_{H_{\beta,\gamma}} (M\otimes N),
		\]
		where $T$ is a complete set of reduced decompositions for minimal length representatives of left cosets of $\mathfrak S_m\times \mathfrak S_n$ in $\mathfrak S_{m+n}$. If $\underline{\omega} \in T$, then $s_{\underline{\omega}}^{-1}(1)\in \lbrace 1,n\rbrace$. Hence we have
		\[
			1_{\beta+\gamma-i,i}\tau_{\underline{\omega}}1_{\beta,\gamma} = \left \{ \begin{array}{ll}
				\tau_{\underline{\omega}}1_{\beta-i,i,\gamma} & \text{if } s_{\underline{\omega}}^{-1}(1) = n, \\
				\tau_{\underline{\omega}}1_{\beta,\gamma-i,i} & \text{if } s_{\underline{\omega}}^{-1}(1) = 1.
			\end{array} \right.
		\]
		Since both $1_{\beta-i,i,\gamma}$ and $1_{\beta,\gamma-i,i}$ act by zero on $M\otimes N$, we conclude that $1_{\beta+\gamma-i,i}MN=0$, and $MN \in \mathcal H[i]$. Thus $\mathcal H[i]$ is a monoidal subcategory of $\mathcal H$.
	\end{proof}

	The following proposition can be seen as the categorification of the description of $U^+[i]$ in terms of the bilinear form (see Proposition \ref{rad}). It is a first justification of why $\mathcal H[i]$ is a categorical analogue of $U^+[i]$. This will be justified further below when we give generators of $\mathcal{H}[i]$ as a Serre and monoidal category.
	
	\begin{prop}
		Let $M \in \mathcal H$. Then $M \in \mathcal H[i]$ if and only if for all $N\in \mathcal H$, $\mathrm{Hom}_{\mathcal H}(NE_i,M)=0$.
	\end{prop}

	\begin{proof}
		Assume that $M$ is an $H_{\beta}$-module for $\beta \in Q^+$. Then the condition $1_{\beta-i,i}M=0$ is equivalent to $\mathrm{Hom}_{H_{\beta-i}}(N,1_{\beta-i,i}M)=0$ for all $N \in H_{\beta-i}\mathrm{-mod}$. By adjunction between right $i$-induction and right $i$-restriction, this last condition is equivalent to $\mathrm{Hom}_{H_{\beta}}(NE_i,M)=0$ for all $N \in H_{\beta-i}\mathrm{-mod}$, and the proposition follows.
	\end{proof}
	
	The inclusion functor $\mathcal{H}[i] \hookrightarrow \mathcal H$ has a left adjoint $\pi_i : \mathcal H \rightarrow \mathcal H[i]$. Explicitly, for a module $M \in H_{\beta}-\mathrm{mod}$, we have
	\[
		\pi_i(M) = M/(1_{\beta-i,i}M)
	\]
	where $(1_{\beta-i,i}M)$ denotes the $H_{\beta}$-submodule of $M$ generated by $1_{\beta-i,i}M$. Note that $\pi_i(M)=M$ if $M \in \mathcal H[i]$. We also have $\mathrm{Hom}_{\mathcal H}(\pi_i(M),\pi_i(M)) \simeq \mathrm{Hom}_{\mathcal H}(M,\pi_i(M))$. So at the decategorified level, $\pi_i$ is an orthogonal projection. The following lemma gives compatibilities between the functor $\pi_i$ and the $j$-induction functors. It will be used repeatedly below.
	
	\begin{lem}\label{piiE}
		For all $j\in I$, there is a canonical isomorphism which is natural in $M \in \mathcal H[i]$
		\[
			\pi_i(E_jM) \simeq \left \{ \begin{array}{ll}
				E_j\pi_i(M) & \text{if } j\neq i, \\
				\pi_i(E_i\pi_i(M)) & \text{if } j=i.
			\end{array} \right.
		\]
	\end{lem}
	
	\begin{proof}
		Assume that $M \in H_{\beta}\mathrm{-mod}$. There is a canonical quotient map $\pi_i(E_jM) \twoheadrightarrow \pi_i(E_j\pi_i(M))$, which we prove is an isomorphism. Let $N \in \mathcal{H}[i]$. We have isomorphisms
		\begin{align*}
		\mathrm{Hom}_{H_{\beta+j}}(\pi_i(E_jM),N)% &\simeq \mathrm{Hom}_{H_{\beta,j}}(M\otimes E_j,1_{\beta,j}N) \\
		%&\simeq \mathrm{Hom}_{H_{\beta}}(M,\mathrm{Hom}_{H_{j}}(E_j,1_{\beta,j}N)) \\
		&\simeq \mathrm{Hom}_{H_{\beta+j}}(E_jM,N) \\
		&\simeq \mathrm{Hom}_{H_{\beta}}(M,1_{j,\beta}N).
		\end{align*}
		the first one coming from the adjunction between $\pi_i$ and the inclusion of $\mathcal H[i]$ in $\mathcal H$, and the second one from the adjunction between left $j$-induction and left $j$-restriction. The same argument applied to $\pi_i(M)$ yields a canonical isomorphism $\mathrm{Hom}_{H_{\beta+j}}(\pi_i(E_j\pi_i(M)),N) \simeq \mathrm{Hom}_{H_{\beta}}(\pi_i(M),1_{j,\beta}N)$. Since $1_{j,\beta}N \in \mathcal H[i]$, we have $\mathrm{Hom}_{H_{\beta}}(\pi_i(M),1_{j,\beta}N) \simeq \mathrm{Hom}_{H_{\beta}}(M,1_{j,\beta}N)$. Hence we have a diagram
		\[
			\xymatrix{
				\mathrm{Hom}_{H_{\beta+j}}(\pi_i(E_j\pi_i(M)),N) \ar[r]^-{\sim} \ar[d]^{\mathrm{can}} & \mathrm{Hom}_{H_{\beta}}(M,1_{j,\beta}N) \\
				\mathrm{Hom}_{H_{\beta+j}}(\pi_i(E_jM),N) \ar[ru]^-{\sim} &
				}
		\]
		in which the isomorphisms are given by adjunctions as explained above. It is straightforward to check that this diagram commutes. Hence the vertical map in the diagram is an isomorphism. So the canonical quotient map $\pi_i(E_jM) \twoheadrightarrow \pi_i(E_j\pi_i(M))$ is an isomorphism by Yoneda's lemma. In the case where $j\neq i$, we have $E_j \in \mathcal H[i]$, so $E_j\pi_i(M) \in \mathcal H[i]$ and $\pi_i(E_j\pi_i(M))=E_j\pi_i(M)$, completing the proof.
	\end{proof}

	\subsubsection{Polynomial annihilators} Let $\beta \in Q^+$. Since $H_{\beta}^i$ is a quotient of $H_{\beta}$, it is naturally an $H_{\beta}$-bimodule, and in particular a $P_{\beta}$-bimodule. A key element to understanding the algebra $H_{\beta}^i$ is to understand which polynomials of $P_{\beta}$ act by zero on $H_{\beta}^i$. The goal of this paragraph is to give explicitly such elements of $P_{\beta}$.
	
	We start with some general results about affine nil Hecke algebras. For a $K$-algebra $A$, consider the $K$-algebra $A \otimes H_{n}^{0}$. It contains the $K$-algebra $A[x_1,\ldots,x_n] = A \otimes P_n$ as a subalgebra. The action of $H_n^{0}$ on $P_n$ given by the isomorphism (\ref{isonil}) induces actions of $H_n^{0}$ and $A \otimes H_{n}^{0}$ on $A[x_1,\ldots,x_n]$.
	
	\begin{lem}\label{annnil}
		Let $J$ be a two-sided ideal of $A\otimes H_n^{0}$. Then $J\cap A[x_1,\ldots,x_n]$ is an $H_n^0$-submodule of $A[x_1,\ldots,x_n]$.	
	\end{lem}
	
	\begin{proof}
		The isomorphism of $H_n^0$-modules
		\begin{equation}\label{lastminute}
			\left \{ \begin{array}{rcl}
				A[x_1,\ldots,x_n] & \xrightarrow{\sim} & \left(A\otimes H_n^0\right)e_n, \\
				P & \mapsto & Pe_n,
			\end{array} \right.
		\end{equation}
		induces an injection $J\cap A[x_1,\ldots,x_n] \hookrightarrow Je_n$. Let us prove that this is also a surjection. Let $h \in J$. We have $h=\sum_{\omega \in \mathfrak{S}_n} h_{\omega} \tau_{\omega}$ for some $h_{\omega} \in P_n$. Since $J$ is an ideal, we have $h_{\omega}\tau_{\omega_0[1,n]} =h\tau_{\omega^{-1}\omega_0[1,n]} \in J$. However, in $H_n^0$ we have $1\in P_n\tau_{\omega_0[1,n]}P_n$, so $h_{\omega} \in P_n h_{\omega}\tau_{\omega_0[1,n]}P_n \subseteq J$. In particular, $he_n=h_1e_n$ is in the image of $J\cap A[x_1,\ldots,x_n]$. So the isomorphism of $H_n^0$-modules (\ref{lastminute}) maps $J\cap A[x_1,\ldots,x_n]$ bijectively to $Je_n$. Since $Je_n$ is an $H_n^0$-submodule of $H_n^0e_n$, it follows that $J\cap A[x_1,\ldots,x_n]$ is an $H_n^0$-submodule of $A[x_1,\ldots,x_n]$.
	\end{proof}
	
	\begin{cor}\label{annilmod}
		\begin{enumerate}
			\item Let $p \in Z(A)[x_1,\ldots,x_n]$, and let $J$ be the two-sided ideal of $A\otimes H_n^0$ generated by $p$. Then as a left $A[x_1,\ldots,x_n]$-module we have
			\[
				J = \bigoplus_{\omega \in \mathfrak{S}_n} ((A\otimes H_n^0)\cdot p)\tau_{\omega}
			\]
			where $(A\otimes H_n^0)\cdot p$ denotes the $\left(A\otimes H_n^0\right)$-submodule of $A[x_1,\ldots,x_n]$ generated by $p$.
			\item Let $M$ be a left (resp. right) $\left(A\otimes H_n^{0}\right)$-module, and let $P \in A[x_1,\ldots,x_n]$ be such that $PM=0$ (resp. $MP=0$). Then $(H_n^0\cdot P)M=0$ (resp. $M(H_n^0\cdot P)=0$), where $H_n^0\cdot P$ denotes the $H_n^0$ submodule of $A[x_1,\ldots,x_n]$ generated by $P$.
		\end{enumerate}
	\end{cor}
	
	\begin{proof}
		For (1), we saw in the proof of Lemma \ref{annnil} that
		\[
			J = \bigoplus_{\omega \in \mathfrak{S}_n} (J\cap A[x_1,\ldots,x_n]) \tau_{\omega}.
		\]
		It also follows from Lemma \ref{annnil} that $J\cap A[x_1,\ldots,x_n]$ contains $(A\otimes H_n^0)\cdot p$. Conversely, given $h_1,h_2 in H_n^0$, we see easily that the coefficient of $h_1ph_2$ on $A[x_1,\ldots,x_n]$ is in $(A\otimes H_n^0)\cdot p$. Hence, $J\cap A[x_1,\ldots,x_n] = (A\otimes H_n^0)\cdot p$, which completes the proof.
		
		For (2), let $J$ be the annihilator of $M$ in $A\otimes H_n^{0}$. Then $J$ is a two-sided ideal of $A\otimes H_n^{0}$. Hence $J\cap A[x_1,\ldots,x_n]$ is an $H_n^0$-submodule of $A[x_1,\ldots,x_n]$ by Lemma \ref{annnil}. The result follows.
	\end{proof}
	
	Recall that for $\nu \in I^{n}$, we have defined a polynomial
	\[
		Q_{i,\nu}(u,v_1,\ldots,v_n) = \prod_{\substack{1\leqslant k \leqslant n \\ \nu_k \neq i}} Q_{i,\nu_k}(u,v_k).
	\]

	\begin{prop}\label{annilpol}
		Let $n\geqslant 1$, let $\nu \in I^n$ and let $a \in \lbrace 1,\ldots,n\rbrace$ be such that $\nu_a=i$. Then $Q_{i,\nu}(x_a,x_1,\ldots,x_n)1_{\nu}$ acts by zero on $H_n^i$. 
	\end{prop}

	\begin{proof}
		The proof is by induction on $n$. Assume that $n=1$. If $\nu_1=i$, $1_i$ is zero in $H_1^i$, so the result holds. If $\nu_1 \neq i$, there is no $a$ such that $\nu_a=i$, and the result holds too.
		
		Assume that the result is proved for $n \geqslant 1$. Let $\nu \in I^{n+1}$. We write $\nu$ in the form $\nu = (\nu_{n+1},\nu')$ with $\nu'\in I^n$. We consider two cases depending on whether $\nu_{n+1}\neq i$ or $\nu_{n+1}=i$.
		
		Case 1: if $\nu_{n+1}\neq i$, then we have $a<n+1$ and
		\begin{align*}
			Q_{i,\nu}(x_a,x_1,\ldots,x_{n+1})1_{\nu} &=Q_{i,\nu_{n+1}}(x_a,x_{n+1}) Q_{i,\nu'}(x_a,x_1,\ldots,x_n)1_{\nu} \\
			&= Q_{i,\nu_{n+1}}(x_a,x_{n+1}) r_n^{n+1}(Q_{i,\nu'}(x_a,x_1,\ldots,x_n)1_{\nu'})1_{\nu}.
		\end{align*}
		By induction, $Q_{i,\nu'}(x_a,x_1,\ldots,x_n)1_{\nu'}$ is zero in $H_n^i$. Hence $Q_{i,\nu}(x_a,x_1,\ldots,x_{n+1})1_{\nu}$ is zero in $H_{n+1}^i$.
		
		Case 2: If $\nu_{n+1}=i$, then we consider three subcases: (i) $a<n+1$, (ii) $a=n+1$ and $\nu_n=i$, and (iii) $a=n+1$ and $\nu_n\neq i$. 
		\begin{itemize}
			\item Sub-case (i): if $a<n+1$, we have
		\begin{align*}
			Q_{i,\nu}(x_a,x_1,\ldots,x_{n+1})1_{\nu} &= Q_{i,\nu'}(x_a,x_1,\ldots,x_n)1_{\nu} \\
			&= r_{n}^{n+1}(Q_{i,\nu'}(x_a,x_1,\ldots,x_n)1_{\nu'})1_{\nu}.
		\end{align*}
		By induction, $Q_{i,\nu'}(x_a,x_1,\ldots,x_n)1_{\nu'}$ is zero in $H_n^i$ so $Q_{i,\nu}(x_a,x_1,\ldots,x_{n+1})1_{\nu}$ is zero in $H_{n+1}^i$.
		
			\item Sub-case (ii): if $a=n+1$ and $\nu_n=i$, we have
		\begin{align*}
			Q_{i,\nu}(x_{n+1},x_1,\ldots,x_{n+1})1_{\nu} &= Q_{i,\nu'}(x_{n+1},x_1,\ldots,x_{n-1},x_{n+1})1_{\nu}\\ 
			&= s_{n}(Q_{i,\nu'}(x_n,x_1,\ldots,x_n))1_{\nu}.
		\end{align*}
		By induction $Q_{i,\nu'}(x_n,x_1,\ldots,x_n)1_{\nu}=r_n^{n+1}(Q_{i,\nu'}(x_n,x_1,\ldots,x_n)1_{\nu'})1_{\nu}$ is zero in $H_{n+1}^i$. Since $\nu_{n+1}=\nu_n=i$, there is a natural structure of left $\left(H_2^0\otimes K[x_1,\ldots,x_{n-1}]\right)$-module on $1_{\nu}H_{n+1}^i$, for which the polynomial $Q_{i,\nu'}(x_n,x_1,\ldots,x_n)$ acts by zero. By Corollary \ref{annilmod}, $s_{n}(Q_{i,\nu'}(x_n,x_1,\ldots,x_n))$ also acts by zero, thus $Q_{i,\nu}(x_{n+1},x_1,\ldots,x_{n+1})1_{\nu}$ is zero in $H_{n+1}^i$.
		
			\item Sub-case (iii): if $a=n+1$ and $\nu_n \neq i$, let $\nu''= (i,\nu_{n-1},\ldots,\nu_1)$. Then we have
		\begin{align*}
			Q_{i,\nu}(x_{n+1},x_1,\ldots,x_{n+1})1_{\nu} &= Q_{i,\nu_n}(x_{n+1},x_n) Q_{i,\nu''}(x_{n+1},x_1,\ldots,x_{n-1},x_{n+1})1_{\nu} \\
			&= \tau_n^2Q_{i,\nu''}(x_{n+1},x_1,\ldots,x_{n-1},x_{n+1})1_{\nu} \\
			&= \tau_n Q_{i,\nu''}(x_{n},x_1,\ldots,x_{n-1},x_{n})\tau_n 1_{\nu} \\
			&= \tau_n r_{n}^{n+1}(Q_{i,\nu''}(x_n,x_1,\ldots,x_n)1_{\nu''}) \tau_n 1_{\nu}
		\end{align*}
		By induction, $Q_{i,\nu''}(x_n,x_1,\ldots,x_n)1_{\nu''}$ is zero in $H_n^i$, hence $Q_{i,\nu}(x_{n+1},x_1,\ldots,x_{n+1})1_{\nu}$ is zero in $H_{n+1}^i$.
		\end{itemize}
	\end{proof}

	The polynomials $Q_{i,\nu}(u,v_1,\ldots,v_n)$ have important symmetry properties. For $\beta \in Q^+$ of height $n$, let
	\[
		Q_{i,\beta}(u,x_1,\ldots,x_n) = \sum_{\nu \in I^{\beta}} Q_{i,\nu}(u,x_1,\ldots,x_n)1_{\nu} \in P_{\beta}[u].
	\]
		
	\begin{lem}\label{annilcenter}
		Let $\beta \in Q^+$. As an element of $P_{\beta}[u]$, the polynomial $Q_{i,\beta}$ has coefficients in the center of $H_{\beta}$.	
	\end{lem}

	\begin{proof}
		The symmetric group $\mathfrak S_n$ acts on $P_{\beta}[u]$ by acting on the coefficients. By Proposition \ref{center}, $Z(H_{\beta}) = P_{\beta}^{\mathfrak{S}_n}$ so it suffices to check that that $Q_{i,\beta}(u,x_1,\ldots,x_n)$ is invariant under the action of $\mathfrak{S}_n$. Let $\omega \in \mathfrak S_n$. We have
		\begin{align*}
			\omega(Q_{i,\beta}(u,x_1,\ldots,x_n)) &= \sum_{\nu \in I^{\beta}} \bigg(\prod_{\substack{1\leqslant k \leqslant n \\ \nu_k \neq i}} \omega(Q_{i,\nu_k}(u,x_k)) \bigg) \omega(1_{\nu}) \\
			&= \sum_{\nu \in I^{\beta}} \bigg(\prod_{\substack{1\leqslant k \leqslant n \\ \nu_k \neq i}} Q_{i,\nu_k}(u,x_{\omega(k)}) \bigg) 1_{\omega(\nu)}.
		\end{align*}
		Doing the re-indexation $\nu' = \omega(\nu)$ in this last sum yields
		\begin{align*}
			\omega(Q_{i,\beta}(u,x_1,\ldots,x_n)) &= \sum_{\nu' \in I^{\beta}}\bigg(\prod_{\substack{1\leqslant k \leqslant n \\ \nu'_{\omega(k)} \neq i}} Q_{i,\nu'_{\omega(k)}}(u,x_{\omega(k)}) \bigg) 1_{\nu'} \\
			&= \sum_{\nu' \in I^{\beta}}\bigg(\prod_{\substack{1\leqslant \ell \leqslant n \\ \nu'_{\ell} \neq i}} Q_{i,\nu'_{\ell}}(u,x_{\ell}) \bigg) 1_{\nu'}
		\end{align*}
		where the second equality follows from doing the re-indexation $\ell=\omega(k)$ in the product. Hence $Q_{i,\beta}(u,x_1,\ldots,x_n)$ is invariant under the action of $\mathfrak S_n$ and the result is proved.
	\end{proof}

	As another application of Corollary \ref{annilmod}, we now prove that the algebras $H_{\beta}^i$ are not finitely generated as $K$-modules in general.
	
	\begin{prop}\label{infgen}
		Let $\beta \in Q^+ \setminus \lbrace 0 \rbrace$ be such that $s_i(\beta) \in Q^+$. Then $H_{\beta}^i$ is not finitely generated as a $K$-module.
	\end{prop}

	\begin{proof}
		Write $\beta$ in the form $\beta = \gamma+ni$ with $\gamma \in \bigoplus_{j\neq i} \mathbb{Z}_{\geqslant 0}j$ of height $r$ and $n\geqslant 0$. Since $s_i(\beta) \in Q^+$, we have $\gamma \neq 0$ and $n \leqslant \left< i^{\vee},\gamma \right>$. We prove that under these conditions, $1_{ni,\gamma}H_{ni+\gamma}^i1_{ni,\gamma}$ is not finitely generated as a $K$-module.
		
		Let $T$ be a set of reduced decompositions for minimal length representatives of left cosets of $\mathfrak S_r\times\mathfrak S_n$ in $\mathfrak S_{n+r}$. The two-sided ideal $1_{ni,\gamma}H_{ni,\gamma}1_{(n-1)i+\gamma,i}H_{ni,\gamma}1_{ni,\gamma}$ of $1_{ni,\gamma}H_{ni,\gamma}1_{ni,\gamma}$ is generated by the elements $\tau_{\underline{\omega}^{-1}}\tau_{\underline{\omega}}1_{ni,\gamma}$ for $\underline{\omega} \in T$ such that $s_{\underline{\omega}}(r+1)=1$, where $\underline{\omega}^{-1}$ denotes the reverse sequence of $\underline{\omega}$. Such $\underline{\omega}$ can be written in the form $\underline{\omega} = \underline{\omega}'\left[1\uparrow r\right]$. Hence $\tau_{\underline{\omega}^{-1}}\tau_{\underline{\omega}}1_{ni,\gamma}$ is an element of $P_{ni+\gamma}1_{ni,\gamma}$, multiple of $\tau_{[r\downarrow 1]}\tau_{[1\uparrow r]}1_{ni,\gamma} = Q_{i,\gamma}(x_{r+1},x_1,\ldots,x_r)$ (Lemma \ref{multistrand}). So $1_{ni,\gamma}H_{ni,\gamma}1_{(n-1)i+\gamma,i}H_{ni,\gamma}1_{ni,\gamma}$ is generated by $Q_{i,\gamma}(x_{r+1},x_1,\ldots,x_r)$ as a two-sided ideal of $1_{ni,\gamma}H_{ni,\gamma}1_{ni,\gamma}$.
		
		Furthermore, since $\gamma$ has no component on $i$, we have $H_{ni}\otimes H_{\gamma} \simeq 1_{ni,\gamma}H_{ni,\gamma}1_{ni,\gamma}$, the isomorphism being given by the $\diamond$ operation. Hence, if we denote by $J$ the two-sided ideal of $H_{ni}\otimes H_{\gamma}$ generated by $Q_{i,\gamma}(x_{r+1},x_1,\ldots,x_r)$, we have $1_{ni,\gamma}H_{ni+\gamma}^i1_{ni,\gamma} \simeq (H_{ni}\otimes H_{\gamma})/J$. By Proposition \ref{annilcenter}, $Q_{i,\gamma}(x_{r+1},x_1,\ldots,x_r) \in Z(H_{\gamma})[x_{r+1}]$. Let $\tilde{J}$ be the the $\left(H_{ni}\otimes H_{\gamma}\right)$-submodule of $H_{\gamma}[x_{r+1},\ldots,x_n]$ generated by $Q_{i,\gamma}(x_{r+1},x_1\ldots,x_r)$. By Corollary \ref{annilmod}, we have
		\[
			1_{ni,\gamma}H^i_{ni+\gamma}1_{ni,\gamma} \simeq \bigoplus_{\omega \in \mathfrak{S}_n} (H_{\gamma}[x_{r+1},\ldots,x_n]/\tilde{J}) (\tau_{\omega}\diamond 1_{\gamma})
		\]
		Since $Q_{i,\gamma}(x_{r+1},x_1,\ldots,x_r)$ does not involve the variables $x_{r+2},\ldots,x_{n+r}$ (and in particular, is symmetric in them), $\tilde{J}$ is generated as a sub-$H_{\gamma}[x_{r+1},\ldots,x_n]$-module by the $\partial_{[k\downarrow r+1]}(Q_{i,\gamma}(x_{r+1},x_1,\ldots,x_r))$ for $k \in \lbrace r+1,\ldots,r+n-1 \rbrace$. However, we have
		\[
			\partial_{[k\downarrow r+1]}(Q_{i,\gamma}(x_{r+1},x_1,\ldots,x_r)) \in tx_{k+1}^{-\left< i^{\vee},\gamma\right>-k+r} + \sum_{\ell < -\left< i^{\vee},\gamma\right>-k+r} x_{k+1}^{\ell}K[x_1,\ldots,x_k]
		\]
		with $t$ an invertible element of $K^{\times}$. Since $n \leqslant -\left< i^{\vee},\gamma\right>$, these polynomials have in particular positive degree. Thus $H_{\gamma}[x_{r+1},\ldots,x_n]/\tilde{J}$ is free of positive rank as an $H_{\gamma}$-module. In particular, it is not finitely generated over $K$.
	\end{proof}

	\medskip

	\subsection{Adjoint action of $E_i$}
	\subsubsection{Definition} Consider the functor $\mathrm{ad}_{E_i} : \mathcal{H}[i] \rightarrow \mathcal{H}[i]$ defined by:
	\[
		\mathrm{ad}_{E_i}(M) = E_iM / (1_{\beta,i} E_iM) = \pi_i(E_iM).
	\]
	for $M \in H_{\beta}^i\mathrm{-mod}$. So $\mathrm{ad}_{E_i}$ is the composition of the left $i$-induction functor with the functor $\pi_i$. Left $i$-induction is an exact endofunctor of $\mathcal H$, and $\pi_i$ is right exact since it is a left adjoint. Hence, the functor $\mathrm{ad}_{E_i}$ is right exact. We will prove below it is actually exact.

	\begin{prop}\label{adei}
		\begin{enumerate}
			\item For $n\geqslant 0$, there is a canonical isomorphism which is natural in $M\in \mathcal H[i]$
			\[
				\mathrm{ad}_{E_i}^n(M) \simeq \pi_i(E_i^nM).
			\]
			\item For all $n\geqslant 0$, there is an algebra morphism
			\[
				H_{ni}^{\mathrm{op}} \rightarrow \mathrm{End}(\mathrm{ad}_{E_i}^n).
			\]
			Hence the affine nil Hecke algebra $H_{ni}$ acts on $\mathrm{ad}_{E_i}^n$.
		\end{enumerate}
	\end{prop}

	\begin{proof}
		We construct the isomorphism in (1) by induction on $n$. For the case $n=0$, we have $\pi_i(M) = M$ for $M \in \mathcal H[i]$, and the canonical isomorphism is just the identity. Assume that we have constructed the natural isomorphism $\mathrm{ad}_{E_i}^n(M) \xrightarrow{\sim} \pi_i(E_i^nM)$ for some $n \geqslant 0$. Then we have an isomorphism
		\begin{align*}
			\mathrm{ad}_{E_i}^{n+1}(M) &\xrightarrow{\sim} \mathrm{ad}_{E_i}(\pi_i(E_i^nM)) = \pi_i(E_i\pi_i(E_i^nM)).
		\end{align*}
		By Lemma \ref{piiE}, there is a canonical isomorphism $\pi_i(E_i\pi_i(E_i^nM)) \xrightarrow{\sim} \pi_i(E_i^{n+1}M)$, which completes the construction of (1).
		
		For (2), notice that we have an action of $H_{ni}^{\mathrm{op}}$ on $E_i^n$ by right multiplication, so we get an action of $H_{ni}^{\mathrm{op}}$ on the functor $M \mapsto E_i^nM$. By vertical composition, we get an action of $H_{ni}^{\mathrm{op}}$ on the functor $M \mapsto \pi_i(E_i^nM)$, which is canonically isomorphic to $\mathrm{ad}_{E_i}^n$ by (1), whence the result.
	\end{proof}

	For the rest of this paper, we fix the action of $H_{ni}^{\mathrm{op}}$ on $\mathrm{ad}_{E_i}^n$ to be the one constructed in the proof of Proposition \ref{adei} via the canonical isomorphism $\mathrm{ad}_{E_i}^n(M) \simeq \pi_i(E_i^nM)$. In particular, one can define the \textit{divided power} $\mathrm{ad}_{E_i}^{(n)}$. Recall that the element $e_n=x_2\ldots x_n^{n-1}\tau_{\omega_0[1,n]}$ is a primitive idempotent of $H_{ni}$. Then the functor $\mathrm{ad}_{E_i}^{(n)}$ is defined by
	\[
		\mathrm{ad}_{E_i}^{(n)} = q_i^{-\frac{n(n-1)}{2}}e_n\big(\mathrm{ad}_{E_i}^{n} \big).
	\]
	From the isomorphism (\ref{divpow}), we have an isomorphism
	\[
		\mathrm{ad}^n_{E_i} \simeq [n]_i!\mathrm{ad}_{E_i}^{(n)}.
	\]	
	Using the definition of our chosen action of $H_{ni}$ on $\mathrm{ad}_{E_i}^n$, we see that for all $M \in \mathcal H[i]$ there is an isomorphism
	\[
		\mathrm{ad}_{E_i}^{(n)}(M) \simeq \pi_i\big(E_i^{(n)}M\big)
	\]
	which is natural in $M$.
	
	\subsubsection{Induction} Let us now give another description of the functor $\mathrm{ad}_{E_i}$, as an induction functor. For $\beta \in Q^+$, the right inclusion induces a (non-unital) morphism $H_{\beta}^{i} \rightarrow H_{\beta+i}^{i}$. This endows $H_{\beta+i}^i1_{i,\beta}$ with a structure of $(H_{\beta+i}^i,H_{\beta}^i)$-bimodule, and we get an induction functor
	\[
		\mathrm{ind}_{H_{\beta}^i}^{H_{\beta+i}^i} : \left \{ \begin{array}{rcl}
			H_{\beta}^i\mathrm{-mod} & \rightarrow & H_{\beta+i}^i\mathrm{-mod}, \\
			M & \mapsto & H_{\beta+i}^i1_{i,\beta} \otimes_{H_{\beta}^i} M.
		\end{array} \right.
	\]
	
	\begin{prop}\label{adind}
		There is an isomorphism
		\[
			\mathrm{ad}_{E_i} \simeq \bigoplus_{\beta \in Q^+} \mathrm{ind}_{H_{\beta}^i}^{H_{\beta+i}^i}.
		\]
	\end{prop}

	\begin{proof}
		Let $\beta \in Q^+$. The canonical quotient morphism $H_{\beta+\alpha_i} \twoheadrightarrow H_{\beta+\alpha_i}^{i}$ endows $H_{\beta+\alpha_i}^{i}$ with a structure of $(H_{\beta+\alpha_i}^{i},H_{\beta+\alpha_i})$-bimodule, and we have a natural isomorphism	$\pi_i(E_iM) \simeq H_{\beta+i}^{i} \otimes_{H_{\beta+i}}(E_iM)$ for $M \in H_{\beta}^i\mathrm{-mod}$. Now recall that
		\[
			E_iM = H_{\beta+i}1_{i,\beta} \otimes_{H_{i,\beta}} (E_i\otimes M) \simeq H_{\beta+i}1_{i,\beta} \otimes_{H_{\beta}} M,
		\]
		the second isomorphism coming from the fact that $E_i$ is a free module of rank 1 over $H_{i}$. Hence we have natural isomorphisms
		\begin{align*}
			\mathrm{ad}_{E_i}(M) &= \pi_i(E_{i}M) \\
			&\simeq H_{\beta+i}^{i} \otimes_{H_{\beta+i}} (E_iM) \\
			&\simeq  H_{\beta+i}^{i} \otimes_{H_{\beta+i}} H_{\beta+i}1_{i,\beta} \otimes_{H_{\beta}} M \\
			&\simeq H_{\beta+i}^{i}1_{i,\beta} \otimes_{H_{\beta}} M.
		\end{align*}
		In general, if $A$ is a $K$-algebra, $J$ a 2-sided ideal of $A$, $L$ a left $A$-module, $R$ a right $A$-module, such that $JL=RJ=0$, we have an obvious isomorphism $R\otimes_A L \simeq R\otimes_{A/J} L$. Here, we get
		\[
			H_{\beta+i}^{i}1_{\beta,i} \otimes_{H_{\beta}} M \simeq H_{\beta+i}^i 1_{\beta,i}\otimes_{H_{\beta}^i} M,
		\]
		which concludes the proof.
	\end{proof}
	
	\subsubsection{The morphism $\tau_{E_i,M}$}\label{tau} We now describe of the functor $\mathrm{ad}_{E_i}$ as the cokernel of a natural transformation. For $\beta \in Q^+$ of height $r$ and $M \in H_{\beta}^i-\mathrm{mod}$, we define a morphism $\tau_{E_i,M} :ME_i \rightarrow E_iM$ as follows:
	\[
		\tau_{E_i,M} : \left \{ \begin{array}{rcl}
			ME_i & \rightarrow & E_iM, \\
			h 1_{\beta,i} \otimes_{H_{\beta}} m & \mapsto & h\tau_{[1\uparrow r]} 1_{i,\beta}\otimes_{H_{\beta}} m.
		\end{array} \right.
	\]
	
	\begin{prop}\label{welldef}
		The map $\tau_{E_i,M}$ is a well-defined morphism of $H_{\beta+i}$-modules of degree $-i\cdot\beta$.
	\end{prop}
	
	\begin{proof}
		By adjunction between left $i$-induction and left $i$-restriction, it suffices to check that the map
		\[
			\left \{ \begin{array}{rcl}
				M & \rightarrow & 1_{\beta,i}H_{\beta+i} 1_{i,\beta} \otimes_{H_{\beta}} M, \\
				 m & \mapsto & \tau_{[1\uparrow r]}1_{i,\beta} \otimes_{H_{\beta}} m,
			\end{array} \right.
		\]
		is a morphism of $H_{\beta}$-modules. To check this, we need to prove that for all $y\in H_{\beta}$ we have
		\[
			\left((y\diamond 1_i)\tau_{[1\uparrow r]}-\tau_{[1\uparrow r]}(1_i \diamond y)\right)1_{i,\beta}\otimes_{H_{\beta}}M=0.
		\]
		Let $Z = \left((y\diamond 1_i)\tau_{[1\uparrow r]}-\tau_{[1\uparrow r]}(1_i \diamond y)\right)1_{i,\beta}$. By Lemma \ref{xtaucom}, the element $Z$ can be written in the form
		\[
			Z = \sum_{k=2}^{r+1} \tau_{[k\uparrow r]} z_k
		\]
		for some $z_k \in H_{i,\beta}$. However, $Z= 1_{\beta,i}Z1_{i,\beta}$ and $1_{\beta,i}\tau_{[k\uparrow r]}=\tau_{[k\uparrow r]}1_{\beta,i}$ when $k>2$. So
		\begin{align*}
			Z &= 1_{\beta,i} Z 1_{i,\beta} \\
			&= \sum_{k=2}^{n+1} \tau_{[k\uparrow r]}1_{\beta,i}z_k 1_{i,\beta}\\
			&= \sum_{k=2}^{n+1} \tau_{[k\uparrow r]} 1_{i,\beta-i,i}z_k 1_{i,\beta}.
		\end{align*}
		Since $1_{\beta-i,i}M=0$, we deduce that $Z \otimes_{H_{\beta}}M=0$. Hence $\tau_{E_i,M}$ is well-defined. The statement about the degree follows from the definition of the grading on the KLR algebras.
	\end{proof}

	Our morphism $\tau_{E_i,M}$ is similar to the morphism $P$ in \cite[Theorem~4.7]{kk}. Similar morphisms also appear in \cite{bkm}, for instance in Lemma 4.9.
	
	Given the definition, it is clear that the morphism $\tau_{E_i,M}$ is natural in $M$, \textit{i.e.} for every arrow $f : M \rightarrow M'$ in $\mathcal{H}[i]$, we have a commutative diagram					
	\[
	\xymatrix{
		ME_i \ar[d]_{fE_i} \ar[r]^{\tau_{E_i,M}} & E_iM \ar[d]^{E_if} \\
		M'E_i \ar[r]_{\tau_{E_i,M'}} & E_iM'
	}
	\]
	Hence, $\tau_{E_i,(-)} : (-)E_i \rightarrow E_i(-)$ is a natural transformation of functors $\mathcal H[i] \rightarrow \mathcal H$.
	
	We now relate $\tau_{E_i,(-)}$ and $\mathrm{ad}_{E_i}$.
	
	\begin{prop}\label{coker}
	For $\beta \in Q^+$ of height $r$ and $M\in H_{\beta}^i\mathrm{-mod}$, we have $\mathrm{ad}_{E_i}(M) = \mathrm{coker}(\tau_{E_i,M})$. Hence, we have an exact sequence which is natural in $M$
			\[
				q_i^{\left<i^{\vee},\beta\right>}ME_i \xrightarrow{\tau_{E_i,M}} E_iM \rightarrow \mathrm{ad}_{E_i}(M) \rightarrow 0.
			\]
	\end{prop}
	
	\begin{proof}
		We show that $\mathrm{im}(\tau_{E_i,M})=H_{\beta+i}1_{\beta,i}(E_iM)$. Let $h \in H_{\beta+i}$ and $m \in M$. We have
			\[
				\tau_{E_i,M}(h1_{\beta,i}\otimes_{H_{\beta}} m) = h\tau_{[1\uparrow r]}1_{i,\beta} \otimes_{H_{\beta}} m\in H_{\beta+i}1_{\beta,i}(E_iM).
			\]
			So $\mathrm{im}(\tau_{E_i,M}) \subseteq H_{\beta+i}1_{i,\beta}(E_iM)$. Conversely, by (\ref{coset}) we have
			\[
				E_iM = \bigoplus_{k=1}^{r+1} \tau_{[k\uparrow r]}1_{i,\beta} \otimes_{H_{i,\beta}} (E_i\otimes M).
			\]
			Note that
			\[
				1_{\beta,i}\tau_{[k\uparrow r]}1_{i,\beta} = \left \{ \begin{array}{ll}
					\tau_{[k\uparrow r]}1_{i,\beta-i,i} & \text{if } k>1, \\
					\tau_{[1\uparrow r]}1_{\beta,i} & \text{if } k=1.
				\end{array} \right. 
			\]
			Since $1_{\beta-i,i}M=0$, we deduce that
			\[
				1_{\beta,i}(ME_i)=\tau_{[1\uparrow r]}1_{i,\beta}\otimes_{H_{i,\beta}} (E_i\otimes M) \subseteq \mathrm{im}(\tau_{E_i,M}).
			\]
			Hence $H_{\beta+i}1_{\beta,i}(E_iM) \subseteq \mathrm{im}(\tau_{E_i,M})$. So we have proved that $\mathrm{im}(\tau_{E_i,M})=H_{\beta+i}1_{\beta,i}(E_iM)$, from which we deduce that $\mathrm{ad}_{E_i}(M) = \mathrm{coker}(\tau_{E_i,M})$.
	\end{proof}
	
	A key property of $\tau_{E_i,(-)}$ is that it is compatible with the monoidal structure of $\mathcal H[i]$, in the following sense.
	
	\begin{prop}\label{tauprod}
		For all $M,N \in \mathcal{H}[i]$, we have $\tau_{E_i,MN} = \tau_{E_i,M}N \circ M\tau_{E_i,N}$.
	\end{prop}

	\begin{proof}
		Assume that $M \in H_{\beta}^i\mathrm{-mod}$ with $\beta$ of height $r$, and that $N\in H_{\gamma}^i\mathrm{-mod}$ with $\gamma$ of height $s$. Let \[y = h1_{\beta,\gamma}\otimes_{H_{\beta,\gamma}} (a\otimes b) \in MN,\] where $h\in H_{\beta+\gamma}$, $a \in M$ and $b \in N$, and \[z=1_{\beta,\gamma,i}\otimes_{H_{\beta,\gamma}} y \in MNE_i.\] Such elements $z$ generate $MNE_i$ as an $H_{\beta+\gamma+i}$-module. So it suffices to prove that $\tau_{E_i,MN}(z) = \tau_{E_i,M}N(M\tau_{E_i,N}(z))$. On the one hand, we have
		\begin{align*}
			\tau_{E_i,MN}(z) &= \tau_{[1\uparrow r+s]}1_{i,\beta+\gamma}\otimes_{H_{\beta+\gamma}} y \\
			&= \tau_{[1\uparrow r+s]}(1_i \diamond h)1_{i,\beta,\gamma} \otimes_{H_{\beta,\gamma}} (a\otimes b).
		\end{align*}
		On the other hand, we have
		\begin{align*}
			M\tau_{E_i,N}(z) & = M\tau_{E_i,N}\left((h\diamond 1_i) 1_{\beta,\gamma,i}\otimes_{H_{\beta,\gamma}} (a\otimes b)\right) \\
				&= (h \diamond 1_i)\tau_{[1\uparrow s]}1_{\beta,i,\gamma} \otimes_{H_{\beta,\gamma}} (a \otimes b).
		\end{align*}
		So
		\begin{align*}
			\tau_{E_i,M}N(M\tau_{E_i,N}(z)) &= (h\diamond 1_i) \tau_{[1\uparrow s]}\tau_{[s+1\uparrow r+s]} 1_{i,\beta,\gamma} \otimes_{H_{\beta,\gamma}} (a\otimes b) \\
			&= (h\diamond 1_i) \tau_{[1\uparrow r+s]} 1_{i,\beta,\gamma} \otimes_{H_{\beta,\gamma}} (a\otimes b)
		\end{align*}
		Hence to get the result, it suffices to prove that
		\[
			\left((h\diamond 1_i)\tau_{[1\uparrow r+s]} - \tau_{[1\uparrow r+s]}(1_i\diamond h) \right)1_{i,\beta,\gamma} \otimes_{H_{\beta,\gamma}} (M\otimes N) =0.
		\]		
		Let $Z=\left((h\diamond 1_i)\tau_{[1\uparrow r+s]} - \tau_{[1\uparrow r+s]}(1_i\diamond h) \right)1_{i,\beta,\gamma}$. By Lemma \ref{xtaucom}, we have
		\[
			Z \in \sum_{k=2}^{r+s+1} \tau_{[k\uparrow r+s]} H_{i,\beta+\gamma}1_{i,\beta,\gamma}.
		\]
		Let $S$ be a complete set of reduced expressions for minimal length representatives of left cosets of $\mathfrak S_r\times \mathfrak S_s$ in $\mathfrak S_{r+s}$, then by (\ref{coset}) we have
		\[
			H_{\beta+\gamma} 1_{\beta,\gamma} = \bigoplus_{\underline{\omega} \in S} \tau_{\underline{\omega}}H_{\beta,\gamma}.
		\]
		Hence $Z$ can be written in the form
		\[
			Z=\sum_{\substack{1<k\leqslant r+s+1 \\ \underline{\omega} \in S}} \tau_{[k\uparrow r+s]}(1_i\diamond \tau_{\underline{\omega}} )h_{k,\underline{\omega}}
		\]		
		for some $h_{k,\underline{\omega}} \in H_{i,\beta,\gamma}$. However $Z=1_{\beta+\gamma,i}Z$, so
		\begin{align*}
		Z &= 1_{\beta+\gamma,i}Z \\
		&= \sum_{\substack{1<k\leqslant r+s+1 \\ \underline{\omega} \in S}} 1_{\beta+\gamma,i}\tau_{[k\uparrow r+s]}(1_i \diamond \tau_{\underline{\omega}})h_{k,\underline{\omega}} \\
		&=\sum_{\substack{1<k\leqslant r+s+1 \\ \underline{\omega} \in S}} \tau_{[k\uparrow r+s]} 1_{\beta+\gamma,i}(1_i\diamond\tau_{\underline{\omega}})h_{k,\underline{\omega}} 
		\end{align*}
		For $\underline{\omega} \in S$, we have $s_{\underline{\omega}}^{-1}(1) \in \lbrace 1,s+1\rbrace$. It follows that
		\begin{align*}
		Z = & \sum_{\substack{1<k\leqslant r+s+1 \\ \underline{\omega} \in S, \, s_{\underline{\omega}}^{-1}(1)=1}} \tau_{[k\uparrow r+s]}(1_i \diamond \tau_{\underline{\omega}})h_{k,\underline{\omega}}1_{i,\beta,\gamma-i,i} + \sum_{\substack{1<k\leqslant r+s+1 \\ \underline{\omega} \in S, \, s_{\underline{\omega}}^{-1}(1)=s+1}} \tau_{[k\uparrow r+s]}(1_i \diamond \tau_{\underline{\omega}})h_{k,\underline{\omega}}1_{i,\beta-i,i,\gamma}.
		\end{align*}
		Since both $1_{\beta-i,i,\gamma}$ and $1_{\beta,\gamma-i,i}$ act by zero on $M\otimes N$, we deduce $Z\otimes_{H_{\beta,\gamma}} (M\otimes N)=0$, which completes the proof.
	\end{proof}
	
	\subsubsection{Example: adjoint action on Chevalley generators}\label{chevgen} In this section, we describe the modules $\mathrm{ad}_{E_i}^{(n)}(E_j)$ for $n\geqslant 0$ and $j\in I \setminus \lbrace i \rbrace$.
	
	\begin{prop}
		As left graded $P_{n+1}$-modules, we have
		\[
			H_{j+ni}1_{j+(n-1)i,i}(E_i^nE_j) = \left( \bigoplus_{\omega \in \mathfrak S_n} J_{n+1}(\tau_{\omega}\diamond 1_j)1_{ni,j}\right) \bigoplus \Bigg(\bigoplus_{\substack{\omega \in \mathfrak S_n\\ 1\leqslant k \leqslant n}} P_{n+1}\tau_{[k\downarrow 1]}(\tau_{\omega}\diamond 1_j)1_{ni,j}\Bigg),
		\]
		where $J_{n+1}$ is the ideal of $P_{n+1}$ generated by the $\partial_{[k\downarrow 2]}(Q_{i,j}(x_2,x_1))$ for $k\in \lbrace 1,\ldots,n \rbrace$.
	\end{prop}

	\begin{proof}		
		Call $L$ the left hand-side of the equality, and $L'$ the right-hand side. Let us start by proving that $L'$ is a $H_{j+ni}$-submodule of $H_{j+ni}1_{ni,j}$. It is clear that $L'$ is stable by the left action of $P_{n+1}$. Furthermore, note that $J_{n+1}$ is the $\left(H_n^0 \otimes K[x_1]\right)$-submodule of $P_{n+1}$ generated by $Q_{i,j}(x_2,x_1)$. Using this and the fact that $\tau_1^21_{ni,j}=Q_{i,j}(x_2,x_1)$, we see easily that $L'$ is also stable by left multiplication by $\tau_1,\ldots,\tau_{n-1}$. Hence $L'$ is a sub-$H_{j+ni}$-module.
		
		By the same argument as in the proof of Proposition \ref{coker}, $L$ is the $H_{j+ni}$-submodule of $H_{j+ni}1_{ni,j}$ generated by the $\tau_{[1\uparrow k]}1_{ni,j}$ for $k \in \lbrace 1,\ldots,n\rbrace$. Since $\tau_{[1\uparrow k]}1_{ni,j} \in L'$ for $k \in \lbrace 1,\ldots,n\rbrace$, we have $L\subseteq L'$.
		
		Conversely, it is clear that
		\[
			\bigoplus_{\substack{\omega \in \mathfrak S_n\\ 1\leqslant k \leqslant n}} P_{n+1}\tau_{[k\downarrow 1]}(\tau_{\omega}\diamond 1_j)1_{ni,j} \subseteq L.
		\]
		Furthermore, $Q_{i,j}(x_2,x_1)1_{ni,j} = \tau_1^21_{ni,j} \in L$. By Corollary \ref{annilmod}, we deduce that $J_{n+1}1_{ni,j} \subseteq L$, and it follows that $L'\subseteq L$.
	\end{proof}

	\begin{cor}
		As left graded $P_{n+1}$-modules, we have
		\begin{align*}
			\mathrm{ad}_{E_i}^{n}(E_j) &= \bigoplus_{\omega \in S_n} \left(P_{n+1}/J_{n+1}\right) (\tau_{\omega}\diamond 1_j) 1_{ni,j},\\
			\mathrm{ad}_{E_i}^{(n)}(E_j) &= q_i^{-\frac{n(n-1)}{2}}\left(P_{n+1}/J_{n+1}\right) (\tau_{\omega_0[1,n]}\diamond 1_j) 1_{ni,j}.
		\end{align*}
	\end{cor}

	In particular, we have a simple formula for the graded ranks of $\mathrm{ad}_{E_i}^{n}(E_j)$ and $\mathrm{ad}_{E_i}^{(n)}(E_j)$. They are given by
	\begin{align*}
		&\mathrm{grk}(\mathrm{ad}_{E_i}^{n}(E_j)) = \frac{1}{(1-q_i^2)^n(1-q_j^2)}\left(\prod_{k=0}^{n-1} \left(1-q_i^{2(-c_{i,j}-k)} \right)\right)\left( \sum_{\omega \in S_n} q_i^{-2l(\omega)}\right),\\
		&\mathrm{grk}(\mathrm{ad}_{E_i}^{(n)}(E_j)) = \frac{q_i^{\frac{n(n-1)}{2}}}{(1-q_i^2)^n(1-q_j^2)}\left(\prod_{k=0}^{n-1} \left(1-q_i^{2(-c_{i,j}-k)} \right)\right).
	\end{align*}
	A consequence of these formulas is the following vanishing criterion:
	\begin{equation}\label{eqvanish}
		\mathrm{ad}_{E_i}^n(E_j) = 0 \, \Leftrightarrow \, n > -c_{i,j}.
	\end{equation}
	We extend that result below.
	
	\medskip

	\subsection{Main theorem and its consequences}\label{mainthm} The first main result is the following.	
	\begin{thm}\label{tauinj}
		For all $\beta \in Q^+$ and $M \in H_{\beta}^i\mathrm{-mod}$, the morphism $\tau_{E_i,M}$ is injective. Hence we have a short exact sequence which is natural in $M$
		\[
			0 \rightarrow q_i^{\left< i^{\vee},\beta\right>} ME_i \xrightarrow{\tau_{E_i,M}} E_iM \rightarrow \mathrm{ad}_{E_i}(M) \rightarrow 0.
		\]
	\end{thm}

	In particular, this theorem gives a categorification of the relation $\mathrm{ad}_{e_i}(y) = e_iy - q_i^{\left< i^{\vee},\beta\right>}ye_i$ for $y \in U^+$ of degree $\beta \in Q^+$. It shows that the functor $\mathrm{ad}_{E_i}$ is indeed a catgeorical lift of the operator $\mathrm{ad}_{e_i}$. The proof of Theorem \ref{tauinj} is done in Subsection \ref{proof}. For the rest of this subsection, we give some corollaries of the theorem.
	
	\smallskip
	
	\subsubsection{Exactness of $\mathrm{ad}_{E_i}$} Theorem \ref{tauinj} is the analogue of \cite[Theorem 4.7]{kk} for cyclotomic KLR algebras. It can be used formally in the same way to prove that the functor $\mathrm{ad}_{E_i}$ is exact.
	
	\begin{cor}\label{proj}
		Let $\beta \in Q^+$. The right $H_{\beta}^{i}$-module $H_{\beta+i}^{i} 1_{i,\beta}$ is projective.
	\end{cor}

	\begin{proof}
		As in \cite{kk}, we use the following lemma.
		\begin{lem*}[{\cite[Lemma 4.18]{kk}}]
			Let $A$ be a $K$-algebra, and $N$ be an $A[t]$-module such that:
			\begin{enumerate}
				\item the projective dimension of $N$ as an $A[t]$-module is at most 1,
				\item there exists $p(t) \in Z(A)[t]$ with invertible leading coefficient such that $p(t)N=0$.
			\end{enumerate}
		Then $N$ is projective as an $A$-module.
		\end{lem*} 
		We use this lemma with $A=H_{\beta}^i$ and $N=H_{\beta+i}^i1_{i,\beta}$. By Theorem \ref{tauinj} applied to $M=H_{\beta}^i$, we have a short exact sequence of $(H_{\beta+i},H_i\otimes H_{\beta}^{i})$-bimodules:
		\[
			0 \rightarrow H_{\beta+i} 1_{\beta,i}\otimes_{H_{\beta}} H_{\beta}^{i} \rightarrow H_{\beta+i} 1_{i,\beta}\otimes_{H_{\beta}} H_{\beta}^{i} \rightarrow H_{\beta+i}^{i} 1_{i,\beta} \rightarrow 0.
		\]
		Note that the compatibility with the right $\big(H_{i} \otimes H_{\beta}^{i}\big)$-module structure comes from the naturality. Since $H_{i} \otimes H_{\beta}^i\simeq H_{\beta}^i[t]$, we can view this sequence as an exact sequence of right $H_{\beta}^i[t]$-modules.
		
		Since $H_{\beta+i} 1_{\beta,i}$ is free as a right $H_{\beta,i}$-module, $H_{\beta+\alpha_i} 1_{\beta,i}\otimes_{H_{\beta}} H_{\beta}^{i}$ is free as a right $H_{\beta}^{i}[t]$-module. Similarly, $H_{\beta+i} 1_{i,\beta}\otimes_{H_{\beta}} H_{\beta}^{i}$ is free as a right $H_{\beta}^{i}[t]$-module. Hence, $H_{\beta+i}^{i} 1_{i,\beta}$ has projective dimension at most 1 as a right $H_{\beta}^{i}[t]$-module. Let
		\[
			p(x_{n+1}) =  Q_{i,\beta}(x_{n+1},x_1,\ldots,x_n) 1_{i,\beta} = \sum_{\nu \in I^{\beta}} \Bigg( \prod_{\substack{1\leqslant k \leqslant n \\ \nu_k \neq i}} Q_{i,\nu_k}(x_{n+1},x_k) \Bigg) 1_{i,\nu} \in P_{\beta+i}.
		\]
		When considered as an element in $P_{\beta}[x_{n+1}]$, $p(x_{n+1})$ has coefficients in $Z(H_{\beta})$ by Proposition \ref{annilcenter}. Furthermore, $H_{\beta+i}^ip(x_{n+1})=0$ by Proposition \ref{annilpol}. Hence the lemma applies and the result follows.
	\end{proof}
	
	\begin{cor}\label{exact}
		The functor $\mathrm{ad}_{E_i}$ is exact.
	\end{cor}

	\begin{proof}
		This follows immediately from Proposition \ref{adind} and Corollary \ref{proj}.
	\end{proof}

	\smallskip

	\subsubsection{Compatibility with products} We now state and prove the compatibility between the monoidal structure of $\mathcal H[i]$ and the action of the functor $\mathrm{ad}_{E_i}$. We can think of the following result as expressing the fact that $\mathrm{ad}_{E_i}$ is a ``categorical derivation". This categorifies the compatibility of the adjoint action of a Hopf algebra with its product.

	\begin{cor}\label{derivation}
		For $M \in H_{\beta}^i\mathrm{-mod}$ and $N\in H_{\gamma}^i\mathrm{-mod}$ with $\beta,\gamma \in Q^+$, there is a short exact sequence which is natural in $M,N$
		\[
			0 \rightarrow q_i^{\left< i^{\vee},\beta\right>}M \ \mathrm{ad}_{E_i}(N) \rightarrow \mathrm{ad}_{E_i}(MN) \rightarrow \mathrm{ad}_{E_i}(M)N \rightarrow 0.
		\]
	\end{cor}

	\begin{proof}
		This a consequence of the following elementary statement: given arrows $f=f_2\circ f_1$ in an abelian category, we have an exact sequence
		\[
			\mathrm{coker}(f_1) \rightarrow \mathrm{coker}(f) \rightarrow \mathrm{coker}(f_2) \rightarrow 0.
		\]
		The first map is induced by $f_2$, the second map is the canonical quotient map. Furthermore, if $f_2$ is injective, the map $\mathrm{coker}(f_1) \rightarrow \mathrm{coker}(f)$ is injective as well. By Proposition \ref{tauprod} and Theorem \ref{tauinj}, the statement applies to $f=\tau_{E_i,MN}$, $f_1 = M\tau_{E_i,N}$ and $f_2=\tau_{E_i,M}N$, yielding the result.
	\end{proof}

	\begin{rem}
		In general, the short exact sequence in Corollary \ref{derivation} does not split. Let us give an example. Let $j\in I \setminus \lbrace i \rbrace$ and $c=-c_{i,j}$. Assume that $c>0$. Let $S_j$ be the simple $H_j$-module equal to $K$ as a graded $K$-module, with $x_1$ acting by 0. We have
		\[
			\mathrm{ad}_{E_i}(S_j) = K[x_2]/x_2^c 1_{ij},
		\]
		where $x_2$ acts by multiplication and $x_1,\tau_1$ and $\tau_2$ act by zero. In particular, the element $x_2^c1_{ij}$ of $H_{i,j}$ acts by zero. From this we deduce that the central element
		\[
			\chi = x_1^c1_{jji} + x_2^c1_{jij} + x_3^c1_{ijj} \in H_{i+2j}
		\]
		acts by zero on both $S_j\mathrm{ad}_{E_i}(S_j)$ and $\mathrm{ad}_{E_i}(S_j)S_j$. We can also compute $\mathrm{ad}_{E_i}(S_j^2)$, where $S_j^2 = S_j S_j$, and find that it decomposes as a $P_{2j+i}$-module as
		\[
			\mathrm{ad}_{E_i}(S_j^2) = 1_{i,2j}\left( \left(K[x_3]/x_3^{2c}\right) \otimes S_j^2\right) \oplus \left( \tau_21_{i,2j}\otimes_{H_{i,2j}} \left(\left(K[x_3]/x_3^{c}\right) \otimes S_j^2 \right)\right).
		\]
		Since $c>0$ we have $2c>c$, so $\chi$ does not act by zero on the first summand. Hence the short exact sequence
		\[
			0 \rightarrow q_i^{-c}S_j \ \mathrm{ad}_{E_i}(S_j) \rightarrow \mathrm{ad}_{E_i}(S_j^2) \rightarrow \mathrm{ad}_{E_i}(S_j)S_j \rightarrow 0
		\]
		does not split.
	\end{rem}

	More generally, Corollary \ref{derivation} can be applied repeatedly to construct an interesting filtration of $\mathrm{ad}_{E_i}^{n}(MN)$. To describe it, let us set up some notation. Given an integer $k \geqslant 0$, and its binary expansion
	\[
		k = \sum_{\ell \geqslant 0} k_{\ell}2^{\ell}
	\]
	where $k_{\ell} \in \lbrace 0,1\rbrace$ are almost all 0, we define its \textit{binary weight} $\zeta(k)$ as the number of $k_{\ell}$ that are equal to 1. There are inductive relations for the binary weight function: for all $k \geqslant 0$ we have
	\begin{equation}\label{reczeta}
		\left \{ \begin{array}{l}
			\zeta(2k) = \zeta(k), \\
			\zeta(2k+1) = \zeta(k)+1.
		\end{array} \right.
	\end{equation}
	We also define a function $\sigma$ by
	\[
		\sigma(k) = \sum_{\ell\geqslant 0} \ell k_{\ell} - \frac{\zeta(k)(\zeta(k)-1)}{2}.
	\]
	There are simple recursive relations for the function $\sigma$: for all $k \geqslant 0$ we have
	\begin{equation}\label{recsigma}
	\left \{ \begin{array}{l}
		\sigma(2k) = \sigma(k) + \zeta(k), \\
		\sigma(2k+1) = \sigma(k).
	\end{array} \right.
	\end{equation}	
	With these notations we can now state and prove the result.
	
	\begin{prop}\label{nderivation}
		Let $M \in H_{\beta}^i\mathrm{-mod}$, $N\in H_{\gamma}^i\mathrm{-mod}$ with $\beta,\gamma \in Q^+$ and $n \geqslant 0$. Then there is a filtration of $\mathrm{ad}_{E_i}^n(MN)$
		\[
			0 = V_{-1} \subseteq V_0 \subseteq \cdots \subseteq V_{2^n-1} = \mathrm{ad}_{E_i}^n(MN),
		\]
		with quotients given by
		\[
			V_k / V_{k-1} \simeq q_i^{(n-\zeta(k))\left< i^{\vee},\beta\right> + 2\sigma(k)}\mathrm{ad}_{E_i}^{\zeta(k)}(M)\mathrm{ad}_{E_i}^{n-\zeta(k)}(N).
		\]
	\end{prop}

	\begin{proof}
		We proceed inductively on $n$. The case $n=0$ is clear. Assume that for some $n$ we have a filtration of $\mathrm{ad}_{E_i}^n(MN)$
		\[
			0 = V_{-1} \subseteq V_0 \subseteq \cdots \subseteq V_{2^n-1} = \mathrm{ad}_{E_i}^n(MN),
		\]
		with quotients given by
		\[
			V_k / V_{k-1} \simeq q_i^{(n-\zeta(k))\left< i^{\vee},\beta\right> + 2\sigma(k)} \mathrm{ad}_{E_i}^{\zeta(k)}(M)\mathrm{ad}_{E_i}^{n-\zeta(k)}(N).
		\]
		Since $\mathrm{ad}_{E_i}$ is exact by Corollary \ref{exact}, we can apply $\mathrm{ad}_{E_i}$ to the filtration above to get a filtration of $\mathrm{ad}_{E_i}^{n+1}(MN)$
		\[
			0 = \mathrm{ad}_{E_i}(V_{-1}) \subseteq \mathrm{ad}_{E_i}(V_0) \subseteq \cdots \subseteq \mathrm{ad}_{E_i}(V_{2^n-1}) = \mathrm{ad}_{E_i}^{n+1}(MN),
		\]
		with quotients given by
		\[
			\mathrm{ad}_{E_i}(V_k)/ \mathrm{ad}_{E_i}(V_{k-1}) \overset{\sim}{\underset{p_k}\rightarrow} q_i^{(n-\zeta(k))\left< i^{\vee},\beta\right> + 2\sigma(k)} \mathrm{ad}_{E_i}\left(\mathrm{ad}_{E_i}^{\zeta(k)}(M)\mathrm{ad}_{E_i}^{n-\zeta(k)}(N)\right).
		\]
		By Corollary \ref{derivation}, there are short exact sequences
		\[
			0 \rightarrow  q_i^{\left< i^{\vee},\beta\right> + 2\zeta(k)}\mathrm{ad}_{E_i}^{\zeta(k)}(M)\mathrm{ad}_{E_i}^{n+1-\zeta(k)}(N) \xrightarrow{f_k} \mathrm{ad}_{E_i}\left(\mathrm{ad}_{E_i}^{\zeta(k)}(M)\mathrm{ad}_{E_i}^{n-\zeta(k)}(N)\right) \rightarrow \mathrm{ad}_{E_i}^{\zeta(k)+1}(M)\mathrm{ad}_{E_i}^{\zeta(k)}(N) \rightarrow 0.
		\]
		We can use them to refine the filtration of $\mathrm{ad}_{E_i}^{n+1}(MN)$. More precisely, for $k \in \lbrace-1,\ldots 2^{n+1}-1\rbrace$ let
		\[
			W_k = \left \{ \begin{array}{ll}
				\mathrm{ad}_{E_i}(V_{(k-1)/2}) & \text{ if } k \text{ is odd}, \\
				p_{k/2}^{-1}(\mathrm{im}(f_{k/2})) & \text{ if } k \text{ is even}.
			\end{array} \right.
		\]
		We get a filtration of $\mathrm{ad}_{E_i}^{n+1}(MN)$ of the form
		\[
			0 = W_{-1} \subseteq W_0 \subseteq \cdots \subseteq W_{2^{n+1}-1} = \mathrm{ad}_{E_i}^{n+1}(XY),
		\]
		with quotients
		\[
			W_k/W_{k-1} \simeq \left \{ \begin{array}{ll}
				q_i^{(n-\zeta((k-1)/2))\left< i^{\vee},\beta\right> + 2\sigma((k-1)/2)}\mathrm{ad}_{E_i}^{\zeta((k-1)/2)+1}(M)\mathrm{ad}_{E_i}^{n-\zeta((k-1)/2)}(N) & \text{ if } k \text{ is odd}, \\ \\
				q_i^{(n-\zeta(k/2))\left< i^{\vee},\beta\right> + 2\sigma(k/2) +\left<i^{\vee},\beta\right> + 2\zeta(k/2)} \mathrm{ad}_{E_i}^{\zeta(k/2)}(M)\mathrm{ad}_{E_i}^{n+1-\zeta(k/2)}(N) & \text{ if } k \text{ is even}.
			\end{array} \right.
		\]
		The conclusion follows from the recursive relations (\ref{reczeta}) for $\zeta$ and (\ref{recsigma}) for $\sigma$.
	\end{proof}	

 	\smallskip
 		
	\subsubsection{A vanishing criterion} We determine for which $\beta \in Q^+$ the algebra $H_{\beta}^i$ is zero, extending the results of Subsection \ref{chevgen}. Let us introduce some notation first. For $\beta \in Q^+$ and $\nu \in I^{\beta}$, we let
	\[
		E_{\nu} = E_{\nu_{\vert \beta \vert}}\ldots E_{\nu_1} = H_{\beta}1_{\nu}. 
	\]
	Note that $H_{\beta}^i1_{\nu}=\pi_i(E_{\nu})$.
	
	\begin{prop}\label{vanish}
		Let $\beta \in Q^+$. The algebra $H_{\beta}^i$ is zero if and only if $s_i(\beta) \notin Q^+$.
	\end{prop}

	\begin{proof}
		We proved in Proposition \ref{infgen} that if $s_i(\beta) \in Q^+$, then $H_{\beta}^i$ is not finitely generated as a $K$-module  (in particular non zero).
		
		Conversely, assume $s_i(\beta) \notin Q^+$. Write $\beta$ in the form $\beta = \gamma +ni$, with $\gamma \in \bigoplus_{j\neq i} \mathbb{Z}_{\geqslant 0}j$ and $n \geqslant 0$. Since $s_i(\beta) \notin Q^+$, we have $n > -\left<i^{\vee},\gamma \right>$. Consider an idempotent of $H_{\beta}^i$ of the form $1_{n_1i,\nu_1,\ldots,n_ki,\nu_k}$ with $n_1+\ldots+n_k=n$, and $\nu_{1}\ldots\nu_k \in I^{\gamma}$. Then
		\begin{align*}
			H_{\beta}^i1_{n_1i,\nu_1,\ldots,n_ki,\nu_k} &= \pi_i(E_i^{n_1}E_{\nu_1}\ldots E_i^{n_k}E_{\nu_k}) \\
			&\simeq \pi_i\left(E_i^{n_1} \pi_i(E_{\nu_1}\ldots E_i^{n_k} E_{\nu_k}) \right)
		\end{align*}
		by Proposition \ref{piiE}. Furthermore by Proposition \ref{adei}, we have isomorphisms
		\begin{align*}
			\pi_i\left(E_i^{n_1} \pi_i(E_{\nu_1}\ldots E_i^{n_k} E_{\nu_k})\right) &\simeq \mathrm{ad}_{E_i}^{n_1}\left(\pi_i\left(E_{\nu_1}\ldots E_i^{n_k} E_{\nu_{k}}\right)\right) \\
			 &\simeq \mathrm{ad}_{E_i}^{n_1}\left(E_{\nu_1}\pi_i\left(E_i^{n_2}E_{\nu_2}\ldots E_i^{n_{k}}E_{\nu_k}\right)\right)
		\end{align*}
		where the second isomorphism is obtained by applying Proposition \ref{piiE} to take the $E_{\nu_1}$ outside of the $\pi_i$. Hence there is an isomorphism
		\[
			H_{\beta}^i1_{n_1i,\nu_1,\ldots,n_ki,\nu_k} \simeq \mathrm{ad}_{E_i}^{n_1}\left(E_{\nu_1}\pi_i\left(E_i^{n_2}E_{\nu_2}\ldots E_i^{n_{k}}E_{\nu_k}\right)\right).
		\]
		This procedure can be repeated inductively and we obtain an isomorphism
		\[
		 H_{\beta}^i1_{n_1i,\nu_1,\ldots,n_ki,\nu_k} \simeq \mathrm{ad}_{E_i}^{n_1}\left(E_{\nu_1}\mathrm{ad}_{E_i}^{n_{2}}\left(E_{\nu_2}\ldots\mathrm{ad}_{E_i}^{n_k}(E_{\nu_k})\ldots \right)\right).
		\]
		Using Proposition \ref{nderivation} repeatedly, we see that $H_{\beta}^i1_{n_1i,\nu_1,\ldots,n_ki,\nu_k}$ has a filtration with quotients being shifts of modules of the form $\mathrm{ad}_{E_i}^{m_r}(E_{j_r}) \ldots \mathrm{ad}_{E_i}^{m_1}(E_{j_1})$ for some $m_1,\ldots,m_r$ such that $m_1 +\ldots+m_r = n$ and $(j_r,\ldots,j_1) \in I^{\gamma}$. We now prove that such quotients are zero. Indeed, for such a quotient, since
		\[
			m_1 + \ldots + m_r =n > -\left<i^{\vee},\gamma\right> = -\left<i^{\vee},j_1\right> - \ldots - -\left<i^{\vee},j_r\right>
		\]
		there exists an index $\ell$ such that $m_{\ell} > -\left<i^{\vee},j_{\ell}\right>$. But then by the vanishing criterion (\ref{eqvanish}), we have $\mathrm{ad}_{E_i}^{m_{\ell}}(E_{j_{\ell}})=0$, and $\mathrm{ad}_{E_i}^{m_r}(E_{j_r}) \ldots \mathrm{ad}_{E_i}^{m_1}(E_{j_1})=0$. Hence all the quotients in the filtration of $H_{\beta}^i1_{n_1i,\nu_1,\ldots,n_ki,\nu_k}$ are zero. So $H_{\beta}^i1_{n_1i,\nu_1,\ldots,n_ki,\nu_k}=0$. Since $1_{\beta}$ is the sum of all idempotents of the form $1_{n_1i,\nu_1,\ldots,n_ki,\nu_k}$, we conclude that $H_{\beta}^i=0$.		
	\end{proof}

	In particular, we deduce from Proposition \ref{vanish} that $H_{\beta+ni}^i$ is zero for $n$ large enough. An immediate consequence is the following corollary.

	\begin{cor}\label{nilp}
		The functor $\mathrm{ad}_{E_i}$ is locally nilpotent.
	\end{cor}

	Propositions \ref{infgen} and \ref{vanish} imply in particular that $H_{\beta}^i$ is either zero, or not finitely generated as a $K$-module. This contrasts with the case of cyclotomic KLR algebras, which are always finitely generated as $K$-modules.

	\medskip
	
	\subsection{Proof of the main theorem}\label{proof}
	\subsubsection{Preliminaries} We start by proving weaker versions of Corollary \ref{exact} and Proposition \ref{nderivation} that will be useful in the proof.
	
	\begin{prop}\label{exactweak}
		Let $M \in \mathcal H[i]$ together with a filtration
		\[
			0=M_{-1} \subseteq M_0 \subseteq \cdots \subseteq M_r=M.
		\]
		For all $k\in \lbrace 0,\ldots, r\rbrace$, put $V_k=M_k/M_{k-1}$ and assume that $\tau_{E_i,V_k}$ is injective. Then $\tau_{E_i,M}$ is injective, and we have a filtration of $\mathrm{ad}_{E_i}(M)$
		\[
			0=\mathrm{ad}_{E_i}(M_{-1}) \subseteq \mathrm{ad}_{E_i}(M_0) \subseteq \cdots \subseteq \mathrm{ad}_{E_i}(M_r) = \mathrm{ad}_{E_i}(M),
		\]
		such that for all $k\in \lbrace 0,\ldots, r\rbrace$, $\mathrm{ad}_{E_i}(M_k)/\mathrm{ad}_{E_i}(M_{k-1}) \simeq \mathrm{ad}_{E_i}(V_k)$.
	\end{prop}

	\begin{proof}
		We proceed by induction on $r$. The result is clear if $r=0$. In general, the short exact sequence
		\[
			0 \rightarrow M_{r-1} \rightarrow M \rightarrow M_r \rightarrow 0
		\]
		gives rise to the following commutative diagram
		\[
			\xymatrix{
				0 \ar[r] & M_{r-1}E_i \ar[r] \ar[d]_{\tau_{E_i,M_{r-1}}} & ME_i \ar[r] \ar[d]_{\tau_{E_i,M}} & V_rE_i \ar[r] \ar[d]_{\tau_{E_i,V_r}} & 0 \\
				0 \ar[r] & E_iM_{r-1} \ar[r] & E_iM \ar[r] & E_iV_{r} \ar[r] &0
			}
		\]
		with exact rows (we omit keeping tracks of the grading shifts here, since they do not play any role). By assumption, $\tau_{E_i,V_r}$ is injective, and by induction $\tau_{E_i,M_{r-1}}$ is injective as well. Hence $\tau_{E_i,M}$ is injective. Furthermore, we have a commutative diagram
		\[
		\xymatrix{
			& 0 \ar[d] & 0 \ar[d] & 0 \ar[d] & \\
			0 \ar[r] & M_{r-1}E_i \ar[r] \ar[d]_{\tau_{E_i,M_{r-1}}} & ME_i \ar[r] \ar[d]_{\tau_{E_i,M}} & V_rE_i \ar[r] \ar[d]_{\tau_{E_i,V_r}} & 0 \\
			0 \ar[r] & E_iM_{r-1} \ar[r] \ar[d] & E_iM \ar[r] \ar[d] & E_iV_r \ar[r] \ar[d] & 0 \\
			0 \ar[r] & \mathrm{ad}_{E_i}(M_{r-1}) \ar[d] \ar[r] & \mathrm{ad}_{E_i}(M) \ar[d] \ar[r] & \mathrm{ad}_{E_i}(V_r) \ar[d] \ar[r] &0 \\
			& 0 & 0 & 0 &
		}
		\]
		in which the first two rows are exact, and the three columns are exact. By the nine lemma, the last row is also exact. By induction, we have a filtration of $\mathrm{ad}_{E_i}(M_{r-1})$ of the form
		\[
			0=\mathrm{ad}_{E_i}(M_{-1}) \subseteq \mathrm{ad}_{E_i}(M_0) \subseteq \cdots \subseteq \mathrm{ad}_{E_i}(M_{r-1}),
		\]
		and such that for all $k$, $\mathrm{ad}_{E_i}(M_k)/\mathrm{ad}_{E_i}(M_{k-1}) \simeq \mathrm{ad}_{E_i}(V_k)$. This filtration together with the third exact row of the diagram above gives the desired filtration of $\mathrm{ad}_{E_i}(M)$.
	\end{proof}

	\begin{prop}\label{derivationweak}
		Let $M\in H_{\beta}^i\mathrm{-mod}$ and let $N \in H_{\gamma}^i\mathrm{-mod}$ with $\beta,\gamma \in Q^+$. Assume that for all $k \geqslant 0$, $\tau_{E_i,\mathrm{ad}_{E_i}^k(M)}$ and $\tau_{E_i,\mathrm{ad}_{E_i}^k(N)}$ are injective. Then there is a filtration
		\[
			0=V_{-1} \subseteq V_{0} \subseteq V_{1} \subseteq \cdots \subseteq V_{2^n-1} = \mathrm{ad}_{E_i}^n(MN),
		\] 
		such that
		\[
			V_k/V_{k-1} \simeq q_i^{(n-\zeta(k))\left<i^{\vee},\beta\right> +2\sigma(k)} \mathrm{ad}_{E_i}^{\zeta(k)}(M)\mathrm{ad}_{E_i}^{n-\zeta(k)}(N).
		\]
	\end{prop}

	\begin{proof}
		The proof is the same inductive proof as that of Proposition \ref{nderivation}, with one difference. Instead of using the exactness of $\mathrm{ad}_{E_i}$ to find a filtration of $\mathrm{ad}^{n+1}_{E_i}(MN)$ from that of $\mathrm{ad}_{E_i}^n(MN)$, we use Proposition \ref{exactweak}. The assumptions of Proposition \ref{exactweak} are satisfied since for all $k \geqslant 0$, $\tau_{E_i,\mathrm{ad}_{E_i}^k(M)}$ and $\tau_{E_i,\mathrm{ad}_{E_i}^k(N)}$ are injective by assumption.
	\end{proof}
	
	\subsubsection{Proof of Theorem \ref{tauinj}} The strategy of the proof is as follows. We start by proving that Theorem \ref{tauinj} holds for modules of the form $\mathrm{ad}_{E_i}^{(n)}(E_{j_r}\ldots E_{j_1})$, where $j_1,\ldots,j_r \in I \setminus \lbrace i \rbrace$ and $n \geqslant 0$, and their subquotients. This part of the proof is similar to the argument of \cite{kk}: it is done by constructing a quasi left inverse to $\tau_{E_i,(-)}$. Then, we prove that as a Serre subcategory of $\mathcal H$, the subcategory $\mathcal H[i]$ is generated by such modules. This is done using the compatibility with the monoidal structure proved in Proposition \ref{derivationweak}. Once these two points are proved, the conclusion follows from Proposition \ref{exactweak}.
	
	\begin{thm}\label{injforgen}
		Let $j_1,\ldots,j_r \in I \setminus \lbrace i \rbrace$ and let $n \geqslant 0$. Let $N$ be a subquotient of $M = \mathrm{ad}_{E_i}^{(n)}(E_{j_r}\ldots E_{j_1})$. Then $\tau_{E_i,N}$ is injective.
	\end{thm}

	\begin{proof}		
		Let $\beta = ni +j_1+ \ldots + j_r$, so that $M$ is an $H_{\beta}^i$-module. We construct a map $\tau_{M,E_i} : E_iM \rightarrow ME_i$ such that the composition $\tau_{M,E_i}\circ \tau_{E_i,M}$ is injective. The map $\tau_{M,E_i}$ is defined by
		\[
			\tau_{M,E_i} : \left \{
			\begin{array}{rcl}
				E_iM & \rightarrow & ME_i \\
				h1_{i,\beta} \otimes_{H_{\beta}} m & \mapsto & hP\tau_{[n+r\downarrow 1]} 1_{\beta,i} \otimes_{H_{\beta}} m
			\end{array}	
			\right.		
		\]
		where
		\[
			P = Q_{i,\beta}(x_{n+r+1},x_1,\ldots,x_{n+r}) = \sum_{\nu \in I^{\beta}} \Bigg(\prod_{\substack{1\leqslant k \leqslant n+r \\ \nu_k \neq i}} Q_{i,\nu_k}(x_{n+r+1},x_{k})\Bigg)1_{i,\nu} \in P_{\beta+i}.
		\]
		To check that this is well-defined, we need to prove that for all $y \in H_{\beta}$, we have
		\[
			\left( (1_i \diamond y)P\tau_{[n+r\downarrow 1]} - P\tau_{[n+r\downarrow 1]}(y\diamond 1_i) \right)1_{\beta,i} \otimes_{H_{\beta}}M=0.
		\]
		Let $Z=\left( (1_i \diamond y)P\tau_{[n+r\downarrow 1]} - P\tau_{[n+r\downarrow 1]}(y\diamond 1_i) \right)1_{\beta,i}$. The element $P$ is central in $H_{i,\beta}$ by Proposition \ref{annilcenter}, so we have $Z= P\left( (1_i\diamond y)\tau_{[n+r\downarrow 1]} - \tau_{[n+r\downarrow 1]}(y\diamond 1_i) \right)1_{\beta,i}$. By Proposition \ref{xtaucom}, there are elements $z_k \in H_{\beta,i}$ such that
		\[
			Z = P\sum_{k=0}^{n+r-1} \tau_{[k\downarrow 1]}z_k 1_{\beta,i}= \sum_{k=0}^{n+r-1} \tau_{[k\downarrow 1]}P z_k1_{\beta,i}.
		\]
		However
		\begin{align*}
			P1_{\beta,i} &= \sum_{\nu \in I^{\beta}} \Bigg(\prod_{\substack{1\leqslant k \leqslant n+r \\ \nu_k \neq i}} Q_{i,\nu_k}(x_{n+r+1},x_{k})\Bigg)1_{i,\nu} 1_{\beta,i} \\
			&= \sum_{\mu \in I^{\beta-i}} \Bigg(\prod_{\substack{1\leqslant k \leqslant n+r-1 \\ \mu_k \neq i}} Q_{i,\mu_k}(x_{n+r+1},x_{k+1})\Bigg)1_{i,\mu,i} \\
			&= \sum_{\substack{\rho \in I^{\beta} \\ \rho_{n+r}=i}} Q_{i,\rho}(x_{n+r+1},x_{2},\ldots,x_{n+r+1})1_{\rho,i} \\
			&=  \Bigg(\sum_{\substack{\rho \in I^{\beta} \\ \rho_{n+r}=i}} Q_{i,\rho}(x_{n+r},x_1,\ldots,x_{n+r})1_{\rho} \Bigg) \diamond 1_i.
		\end{align*}
		By Proposition \ref{annilpol}, $Q_{i,\rho}(x_{n+r},x_1,\ldots,x_{n+r})1_{\rho}$ acts by zero on $M$ for all $\rho \in I^{\beta}$ such that $\rho_{n+r}=i$. So $Z\otimes_{H_{\beta}} M=0$, and the morphism $\tau_{M,E_i}$ is well defined.
		
		Next, we compute the composition $\tau_{M,E_i}\circ \tau_{E_i,M}$. It is given by
		\[
			\tau_{M,E_i}\circ \tau_{E_i,M} : \left \{ \begin{array}{rcl}
				ME_i & \rightarrow & ME_i \\
				h1_{\beta,i} \otimes_{H_{\beta}} m & \mapsto & h\tau_{[1\uparrow n+r]}P\tau_{[n+r\downarrow 1]}1_{\beta,i} \otimes_{H_{\beta}} m
			\end{array} \right.
		\]		
		Note that $ME_i$ is cyclic, generated by the class $c$ of $\tau_{\omega_0[r+2,n+r+1]}1_{ni,j_r,\ldots,j_1,i}$ in $ME_i$. We start by computing the image of $c$ under $\tau_{M,E_i}\circ \tau_{E_i,M}$. In the following computations, we do not write the idempotent $1_{ni,j_r,\ldots,j_1,i}$ on the right to help readability, but all the elements are considered in $H_{\beta+i}1_{ni,j_r,\ldots,j_1,i}$. The detailed explanations are below the computation. We have
		\begin{align*}
			\tau_{[1\uparrow n+r]}P\tau_{[n+r\downarrow 1]}\tau_{\omega_0[r+2,n+r+1]}&= \tau_{[1\uparrow n+r]}P\tau_{\omega_0[r+1,n+r+1]} \tau_{[r\downarrow 1]} \\
			&= \tau_{[1\uparrow r]} \partial_{[r+1\uparrow n+r]}(P) \tau_{\omega_0[r+1,n+r+1]} \tau_{[r\downarrow 1]} \\
			&= s_{[1\uparrow r]}\left(\partial_{[r+1\uparrow n+r]}(P)\right) \tau_{[1\uparrow r]} \tau_{[n+r\downarrow 1]}\tau_{\omega_0[r+2,n+r+1]} \\
			&= s_{[1\uparrow r]}\left(\partial_{[r+1\uparrow n+r]}(P)\right) \tau_{[n+r\downarrow r+2]}\tau_{[1\uparrow r]}\tau_{r+1}\tau_{[r\downarrow 1]}\tau_{\omega_0[r+2,n+r+1]}.
		\end{align*}
		For the first equality, we have commuted $\tau_{[r\downarrow 1]}$ to the right of $\tau_{\omega_0[r+2,n+r+1]}$, and the factor $\tau_{[n+r\downarrow r+1]}\tau_{\omega_0[r+2,n+r+1]}$ that appears is equal to $\tau_{\omega_0[r+1,n+r+1]}$. The second equality comes from applying part (2) of Lemma \ref{compnil}. For the third equality, we have commuted $\partial_{[r+1\uparrow n+r]}(P)$ to the left of $\tau_{[1\uparrow r]}$ using relation (\ref{Taux}) from Definition \ref{klr}, and we have undone the first step. Finally, the fourth equality is obtained by commuting  $\tau_{[1\uparrow r]}$ to the right of $\tau_{[n+r\downarrow r+2]}$. Now by relation (3) in Lemma \ref{multistrand} we have
		\begin{align*}
			\tau_{[1\uparrow r]}\tau_{r+1}\tau_{[r\downarrow 1]} &=\left(\tau_{[r+1\downarrow2]}\tau_1\tau_{[2\uparrow r+1]} + \tilde{P}\right) =\tilde{P} \quad \text{ mod } (1_{\beta-i,i,i})
		\end{align*}
		where $\tilde{P} = \partial_{r+2,1}\left(Q_{i,(j_1,\ldots,j_r)}(x_{r+2},x_{2},\ldots,x_{r+1})\right)$. By part (2) of Lemma \ref{compnil} we have
		\[
			\tau_{[n+r\downarrow r+2]}\tilde{P}\tau_{\omega_0[r+2,n+r+1]} = \partial_{[n+r\downarrow r+2]}(\tilde{P}) \tau_{\omega_0[r+2,n+r+1]}.
		\]
		Hence if we let
		\[
			\theta= s_{[1\uparrow r]}\left(\partial_{[r+1\uparrow n+r]}(P)\right)\partial_{[n+r\downarrow r+2]}(\tilde{P})1_{ni,j_r,\ldots,j_1,i} \in P_{\beta+i}
		\]
		we have proved that
		\[
			\tau_{[1,\uparrow n+r]}P\tau_{[n+r\downarrow 1]}\tau_{\omega_0[r+2,n+r+1]} = \theta \tau_{\omega_0[r+2,n+r+1]} \quad \text{ mod } \left(1_{\beta-i,i,i}\right).
		\]
		This proves that $\tau_{M,E_i}\circ \tau_{E_i,M}(c) = \theta c$. Now let
		\[
			\Theta = \sum_{\omega \in \mathfrak S_{n+r+1}/\mathrm{Stab}(ni,j_r,\ldots,j_1,i)} \omega(\theta) \in P_{\beta+i},
		\]
		where $\mathrm{Stab}(ni,j_r,\ldots,j_1,i)$ denotes the stabilizer of the sequence $(ni,j_r,\ldots,j_1,i)$ in $\mathfrak S_{n+r+1}$. Let us check that $\Theta$ is in the center of $H_{\beta+i}$. By Proposition \ref{center}, this amounts to checking that $\Theta$ is fixed by $\mathfrak S_{n+r+1}$. Given the definition of $\Theta$, it suffices to check that for $\omega \in \mathrm{Stab}(ni,j_r,\ldots,j_1,i)$ we have $\omega(\theta) = \theta$. By Proposition \ref{annilcenter}, we already know that $P,\tilde{P}$ are fixed under those $\omega$ who permute the variables labeled by $j_1,\ldots,j_r$, thus so is $\theta$. But then the Demazure operators in $\partial_{[r+1\uparrow n+r]}(P),\partial_{[n+r\downarrow r+2]}(\tilde{P})$ make these polynomials symmetric in the variables labeled by $i$. Hence $\theta$ is indeed fixed by $\mathrm{Stab}(ni,j_r,\ldots,j_1,i)$, and $\Theta$ is central. Furthermore, we have
		\[
			\Theta c = \Theta 1_{ni,j_r,\ldots,j_1,i} c= \theta c=	\tau_{M,E_i}\circ \tau_{E_i,M}(c).
		\]
		Hence $\tau_{M,E_i}\circ \tau_{E_i,M}$ and multiplication by $\Theta$ (which is a morphism of $H_{\beta+i}$-modules since $\Theta$ is central) agree on $c$. Since $c$ generates $ME_i$, we conclude that $\tau_{M,E_i}\circ \tau_{E_i,M}(x) = \Theta x$ for all $x\in ME_i$. Now remark that $\Theta 1_{\beta,i}$ has the form
		\[
			\Theta 1_{\beta,i} \in tx_{1}^{\ell}1_{\beta,i} + \sum_{k<\ell} x_{1}^k(P_{\beta}\diamond 1_i)1_{\beta,i} 
		\]
		for some integer $\ell$ and $t \in K^{\times}$. It follows that multiplication by $\Theta 1_{\beta,i}$ is an injective endomorphism of $M\otimes E_i$. But since
		\[
			ME_i = \bigoplus_{k=0}^{n+r} \tau_{[k\downarrow 1]} 1_{\beta,i} \otimes_{H_{\beta,i}} (M\otimes E_i),
		\]
		multiplication by $\Theta$ acts diagonally along this decomposition, by multiplication by $\Theta 1_{\beta,i}$ on each summand $M\otimes E_i$. We conclude that multiplication by $\Theta$ is an injective endomorphism of $ME_i$. Hence $\tau_{E_i,M}$ is injective.
		
		We now prove that this implies the result for the subquotient $N$ of $M$. We can fit $N$ in a diagram
		\[
			N \hookrightarrow N' \overset{f}{\twoheadleftarrow} M
		\]
		for some $H_{\beta}$-module $N'$. Given its definition, the map $\tau_{M,E_i}$ restricts to a map $E_i\mathrm{ker}(f) \rightarrow \mathrm{ker}(f)E_i$. So we have an induced map $\tau_{N',E_i} : E_iN' \rightarrow N'E_i$ such that $\tau_{N',E_i}\circ \tau_{E_i,N'}$ is multiplication by $\Theta$. By the same argument as above, multiplication by $\Theta$ is injective on $N'E_i$. Hence $\tau_{E_i,N'}$ is injective. But by naturality $\tau_{E_i,N}$ is the restriction of $\tau_{E_i,N'}$ to $NE_i$, so it is injective as well.
	\end{proof}

	\begin{cor}\label{injforserre}
		Let $N$ be in the Serre subcategory of $\mathcal H[i]$ generated by the modules $\mathrm{ad}_{E_i}^{n}(E_{j_r}\ldots E_{j_1})$ for $n \geqslant 0$, $r \geqslant 1$ and $j_1,\ldots,j_r \in I \setminus \lbrace i \rbrace$. Then $\tau_{E_i,N}$ is injective.
	\end{cor}

	\begin{proof}
		Since $\mathrm{ad}_{E_i}^{n} \simeq [n]_i!\mathrm{ad}_{E_i}^{(n)}$, we know that $N$ is in the Serre subcategory of $\mathcal H[i]$ generated by the modules $\mathrm{ad}_{E_i}^{(n)}(E_{j_r}\ldots E_{j_1})$ for $n \geqslant 0$, $r \geqslant 1$ and $j_1,\ldots,j_r \in I \setminus \lbrace i \rbrace$. There is a filtration of $N$
		\[
			0 \subseteq N_0 \subseteq \cdots \subseteq N_{\ell}=N
		\]
		such that each $N_k/N_{k-1}$ is a subquotient of some $\mathrm{ad}_{E_i}^{(n)}(E_{j_r}\ldots E_{j_1})$. Hence $\tau_{E_i,N_k/N_{k-1}}$ is injective by Theorem \ref{injforgen}. By Proposition \ref{exactweak}, we conclude that $\tau_{E_i,N}$ is injective.
	\end{proof}
	
	\begin{thm}\label{genforhi}
		As a Serre subcategory of $\mathcal H$, the subcategory $\mathcal H[i]$ is generated by the modules $\mathrm{ad}_{E_i}^{n}(E_{j_r}\ldots E_{j_1})$ for $n\geqslant 0$, $r\geqslant 1$ and $j_1,\ldots,j_r \in I \setminus \lbrace i \rbrace$.
	\end{thm}

	\begin{proof}
		Let $\mathcal H'[i]$ be the Serre full subcategory of $\mathcal H$ generated by the modules $\mathrm{ad}_{E_i}^{n}(E_{j_r}\ldots E_{j_1})$ for $n\geqslant 0$, $r\geqslant 1$ and $j_1,\ldots,j_r \in I \setminus \lbrace i \rbrace$. It is clear that $\mathcal H'[i] \subseteq \mathcal H[i]$. We prove the converse inclusion.
		
		Let $M\in \mathcal H[i]$. There exists a projective module $P \in \mathcal H$ that surjects onto $M$. Since $\pi_i$ is right exact, $\pi_i(P)$ surjects onto $\pi_i(M)=M$. Hence it suffices to show that $\pi_i(P) \in \mathcal H'[i]$ for every projective module $P$. Since every projective module is a direct sum of summands of modules of the form $E_{j_m}\ldots E_{j_1}$, it suffices to show $\pi_i(E_{j_m}\ldots E_{j_1}) \in \mathcal H'[i]$ for all $m$ and $j_1,\ldots,j_m \in I$. We proceed inductively on $m$.
		
		For $m=1$, we have $\pi_i(E_{j_1})=E_{j_1}$ if $j_1\neq i$ and $\pi_i(E_i)=0$. In all cases, we have $\pi_i(E_{j_1}) \in \mathcal H'[i]$. Assume now $\pi_i(E_{j_{m-1}}\ldots E_{j_{1}}) \in \mathcal{H}'[i]$, we consider two cases:
		\begin{itemize}
			\item if $j_m=j \neq i$, then $\pi_i(E_{j_m}\ldots E_{j_1}) \simeq E_j\pi_i(E_{j_{m-1}}\ldots E_{j_{1}})$ by Lemma \ref{piiE}. So it suffices to show that $E_j\mathcal H'[i] \subseteq \mathcal{H}'[i]$. If $M \in \mathcal H'[i]$, then we have a filtration
			\[
				0 = M_{-1} \subseteq M_0 \subseteq \cdots \subseteq M_s=M,
			\]
			such that $M_k/M_{k-1}$ is a subquotient of $\mathrm{ad}_{E_i}^{n_k}(N_k)$ for some $n_k \geqslant 0$ and $N_k$ a product of $E_{a}$, $a \in I \setminus \lbrace i \rbrace$. By exactness of multiplication by $E_j$, we have a filtration of $E_jME$
			\[
				0 = E_jM_{-1} \subseteq E_jM_0 \subseteq \cdots \subseteq E_jM_s=E_jM,
			\]
			such that $E_jM_k/E_jM_{k-1}$ is a subquotient of $E_j\mathrm{ad}_{E_i}^{n_k}(N_k)$. By Theorem \ref{injforgen}, $N_k$ and $E_j$ satisfy the conditions of Proposition \ref{derivationweak}, so we have an injection (up to a grading shift)
			\[
				E_j\mathrm{ad}_{E_i}^{n_k}(N_k) \hookrightarrow \mathrm{ad}_{E_i}^{n_k}(E_jN_k).
			\]
			Hence $E_jM$ is also in $\mathcal H'[i]$.
			
			\item if $j_m=i$, we have $\pi_i(E_{j_m}\ldots E_{j_1}) \simeq \mathrm{ad}_{E_i}(\pi_i(E_{j_{m-1}}\ldots E_{j_{1}}))$ by Lemma \ref{piiE}. So it suffices to prove that $\mathcal H'[i]$ is stable under $\mathrm{ad}_{E_i}$. If $M \in \mathcal H'[i]$, then we have a filtration
			\[
				0 = M_{-1} \subseteq M_0 \subseteq \cdots \subseteq M_s=M
			\]
			such that $M_k/M_{k-1}$ is a subquotient of some $\mathrm{ad}_{E_i}^{n_k}(N_k)$ for $N_k$ a product of $E_a$, $a \in I \setminus \lbrace i \rbrace$. By Corollary \ref{injforserre}, $\tau_{E_i,M_k/M_{k-1}}$ is injective for all $k$. So we can apply Proposition \ref{exactweak} to get a filtration of $\mathrm{ad}_{E_i}(M)$
			\[
				0 = \mathrm{ad}_{E_i}(M_{-1}) \subseteq \mathrm{ad}_{E_i}(M_0) \subseteq \cdots \subseteq \mathrm{ad}_{E_i}(M_s)=\mathrm{ad}_{E_i}(M),
			\]
			such that
			\[
				\mathrm{ad}_{E_i}(M_k)/\mathrm{ad}_{E_i}(M_{k-1}) \simeq \mathrm{ad}_{E_i}(M_k/M_{k-1}).
			\]
			By assumption, $M_k/M_{k-1}$ fits into a diagram
			\[
				M_k/M_{k-1}  \hookrightarrow N \twoheadleftarrow \mathrm{ad}_{E_i}^{n_k}(N_k).
			\]
			Since $\tau_{E_i,(-)}$ is injective for the three modules in this diagram, by Proposition \ref{exactweak} we can apply $\mathrm{ad}_{E_i}$ to this diagram to get
			\[
				\mathrm{ad}_{E_i}(M_k/M_{k-1})  \hookrightarrow \mathrm{ad}_{E_i}(N) \twoheadleftarrow \mathrm{ad}_{E_i}^{n_k+1}(N_k).
			\]
			So $\mathrm{ad}_{E_i}(M_k)/\mathrm{ad}_{E_i}(M_{k-1})$ is a subquotient of $\mathrm{ad}_{E_i}^{n_k+1}(N_k)$. Hence $\mathrm{ad}_{E_i}(M) \in \mathcal{H}'[i]$, and the proof is complete.
		\end{itemize} 
	\end{proof}

	Now Corollary \ref{injforserre} and Theorem \ref{genforhi} imply Theorem \ref{tauinj}.
	
	\subsection{Generating objects} As an immediate consequence of Theorem \ref{genforhi} and Proposition \ref{nderivation}, we get the following generating result for $\mathcal H[i]$.
	
	\begin{cor}\label{gen2}
		As a Serre and monoidal full subcategory of $\mathcal H$, $\mathcal H[i]$ is generated by the objects $\mathrm{ad}_{E_i}^{(n)}(E_j)$ for $n\geqslant 0$ and $j \in I \setminus \lbrace i \rbrace$.
	\end{cor}
	
	Corollary \ref{gen2} is a categorification of the definition of $U^+[i]$ as the subalgebra generated by the $\mathrm{ad}_{e_i}^{(n)}(e_j)$ for $n\geqslant 0$ and $j \in I \setminus \lbrace i \rbrace$. Lusztig also proves that $U^+ = \sum_{n \geqslant 0} U^+[i]e_i^{(n)}$ \cite[Lemma 38.1.2]{Lu}. We finish this section by proving a similar result for $\mathcal H$.
	
	\begin{prop} The Serre subcategory of $\mathcal H$ generated by the subcategories $\mathcal H[i]E_i^{(n)}$ for $n\geqslant 0$ is equal to $\mathcal H$.
	\end{prop}
	
	\begin{proof}
		Let $\mathcal H'$ be the Serre subcategory of $\mathcal H$ generated by the subcategories $\mathcal H[i]E_i^{n}$ for $n\geqslant0$. We have $E_j \in \mathcal H'$ for all $j \in I$. So to prove that $\mathcal H = \mathcal H'$, it suffices to prove that $\mathcal H'$ is also a monoidal subcategory.		
		We start by proving that
		\begin{equation}\label{eim}
			\text{if } M \in \mathcal{H}',  \text{ then } E_iM \in \mathcal H'.
		\end{equation}
		Given $M \in \mathcal H'$, there is a filtration $0=M_0 \subseteq M_1 \subseteq \cdots \subseteq M_r = M$ such that each quotient $M_{k}/M_{k-1}$ is a subquotient of an object of $\mathcal H[i]E_i^{n_k}$. We want to prove $E_iM \in \mathcal H'$. By induction on $r$, it suffices to prove the result for $M=M'E_i^n$, with $M' \in \mathcal H[i]$. But then by Theorem \ref{tauinj} we have a short exact sequence
		\[
			0 \rightarrow M'E_i^{n+1} \rightarrow E_iM \rightarrow \mathrm{ad}_{E_i}(M')E_i^n \rightarrow 0. 
		\]
		Since $M'E_i^{n+1},\mathrm{ad}_{E_i}(M')E_i^n \in \mathcal{H}'$, we conclude that  $E_iM \in \mathcal H'$. Hence statement (\ref{eim}) is established. Since $\mathcal H[i]$ is monoidal, it is also clear that \begin{equation}\label{mn}\text{if } M\in \mathcal H' \text{ and }N \in \mathcal H[i], \text{ then } MN \in \mathcal H'.\end{equation}
		
		Now let $M,N \in \mathcal H'$, we want to prove that $MN \in \mathcal H'$. We have filtrations
		\[
			0=M_0 \subseteq M_1 \subseteq \cdots \subseteq M_{r-1} \subseteq M_r = M,
		\]
		\[
			0=N_0 \subseteq N_1 \subseteq \cdots \subseteq N_{s-1} \subseteq N_s = N,
		\]
		such that $M_k/M_{k-1}$ is a subquotient of an object of $\mathcal H[i]E_i^{m_k}$ and $N_k/N_{k-1}$ is a subquotient of an object of $\mathcal H[i]E_i^{n_k}$. There is an exact sequence
		\[
			0 \rightarrow MN_{s-1} \rightarrow MN \rightarrow M(N/N_{s-1}) \rightarrow 0,
		\]
		so by induction on $s$, it suffices to show the result for $s=1$. Similarly there is an exact sequence
		\[
			0 \rightarrow M_{r-1}N \rightarrow MN \rightarrow (M/M_{r-1})N \rightarrow 0,
		\]
		so by induction on $r$, it suffices to show the result for $r=1$. In that case, $M$ is a subquotient of $M'E_i^m$ and $N$ is a subquotient of $N'E_i^n$ with $M',N' \in \mathcal H[i]$. Then by exactness of the tensor product, $MN$ is a subquotient of $M'E_i^mN'E_i^n$. But since $N'E_i^n \in \mathcal H'$, applying statement (\ref{eim}) $m$ times gives $E_i^mN'E_i^n \in \mathcal H'$. Then statement (\ref{mn}) gives $M'E_i^mN'E_i^n \in \mathcal H'$. Thus $MN \in \mathcal H'$.
	\end{proof}
	
	\bigskip
	
	\section{Categorical $\mathfrak{sl}_2$ action} In this section we finish proving that there is a categorical action of $\mathfrak{sl}_2$ on $\mathcal H[i]$. We start by recalling the definition of categorical $\mathfrak{sl}_2$ actions. With what we have already proved, the only axiom left to check is the categorical $[e_i,f_i]=h_i$ relation. The proof mirrors that of \cite{kk} for cyclotomic KLR algebras, and is based on two ingredients: the short exact sequence of Theorem \ref{tauinj} and the Mackey decompositions for KLR algebras. One difference, as in the proof of Theorem \ref{tauinj}, is that we only do explicit computations for modules of the form $\mathrm{ad}_{E_i}^{n}(E_{j_r}\ldots E_{j_1})$, which is enough as they generate $\mathcal{H}[i]$ by Theorem \ref{genforhi}.
	
	\subsection{2-representations}
	There are various setups for 2-representations (on additive, abelian or triangulated categories). In our framework, we use a version for abelian categories. Since it is not important for our purpose, we will not introduce the 2-category attached to $\mathfrak{sl}_2$, but rather just define what a representation of it is, as in the original approach of \cite{ChR}.
	
	\begin{defi}\label{2repsl2}A \textit{2-representation} of $\mathfrak{sl}_2$ is the data of
	\begin{itemize}
		\item a graded, abelian, $K$-linear category $\mathcal V$ together with a decomposition into \textit{weight subcategories} $\mathcal V = \oplus_{w \in \mathbb Z} \mathcal V_w$,
		\item endofunctors $\mathcal E, \mathcal F$ of $\mathcal V$ restricting to $\mathcal E1_w: \mathcal V_w \rightarrow \mathcal V_{w+2}$ and $\mathcal F1_{w}:V_w\rightarrow V_{w-2}$ for all $w \in \mathbb Z$,
		\item natural transformations $x : \mathcal{E} \rightarrow \mathcal E$ of degree 2, $\tau : \mathcal E^2 \rightarrow \mathcal E^2$ of degree -2, $\eta : 1_w \rightarrow \mathcal F\mathcal E1_w$ of degree $1+w$ and $\varepsilon : \mathcal E\mathcal F1_w \rightarrow 1_w$ of degree $1-w$ for all $w \in \mathbb{Z}$,
	\end{itemize}

	subject to the following conditions
	\begin{enumerate}
		\item the functors $\mathcal E$, $\mathcal F$ are exact,
		\item the natural transformation $\varepsilon, \eta$ are the counit and unit respectively of adjunctions $(\mathcal E1_w,q^{w+1}\mathcal F1_{w+2})$ for all $w\in \mathbb{Z}$,
		\item the natural transformations $x$ and $\tau$ induce actions of affine nil Hecke algebras on powers of $\mathcal E$; which means that the following equalities hold:
		\[
		\left \{\begin{array}{l}
			\tau^2=0,\\
			\tau \circ x\mathcal E - \mathcal Ex\circ \tau = x\mathcal E\circ \tau - \tau\circ \mathcal Ex = \mathcal E^2, \\
			\tau \mathcal E\circ \mathcal E\tau \circ \tau \mathcal E - \mathcal E\tau \circ \tau \mathcal E \circ \mathcal E\tau =0,
		\end{array} \right.
		\]
		\item for all $w \in \mathbb Z$, there are isomorphisms
		\[
		\left \{ \begin{array}{l}
			\text{if} \ w \geqslant 0, \ \mathcal E
			\mathcal F 1_w \xrightarrow[\sim]{\rho_w} \mathcal F\mathcal E1_w \oplus [w]1_w, \\
			\text{if} \ w \leqslant 0, \ \mathcal E \mathcal F1_w \oplus [-w]1_w \xrightarrow[\sim]{\rho_w} \mathcal F \mathcal E1_w,
		\end{array} \right.
		\]
			where the $\rho_w$ are some explicit natural transformations defined below.
	\end{enumerate}
	To complete the definition, we need to define the natural transformations $\rho_w$. We start by defining a map $\varphi : EF \rightarrow FE$ as the composition
	\begin{equation}\label{varphi}
		\varphi = \left(\mathcal E\mathcal F \xrightarrow{\eta \mathcal E\mathcal F} \mathcal F\mathcal E\mathcal E\mathcal F \xrightarrow{\mathcal F\tau \mathcal F} \mathcal F\mathcal E\mathcal E\mathcal F \xrightarrow{\mathcal F\mathcal E\varepsilon} \mathcal F\mathcal E\right).
	\end{equation}
	If $w\geqslant 0$, we define the natural transformation $\rho_w$ as a column vector as follows
	\[
		 \rho_w = \left[ \begin{array}{c} 
			\varphi \\
			\varepsilon \\
			\varepsilon \circ x\mathcal F \\
			\vdots \\
			\varepsilon\circ x^{w-1}\mathcal F
		\end{array} \right].
	\]
	If $w \geqslant 0$, we define the natural transformation $\rho_w$ as a row vector as follows
	\[
		\rho_w = \left[\varphi \quad \eta \quad \mathcal Fx \circ \eta \quad \ldots \quad \mathcal Fx^{w-1}\circ \eta \right].
	\]
	This completes the definition of a categorical $\mathfrak{sl}_2$-action.
	\end{defi}
	
	\begin{rem}
		The degrees of the natural transformations can also be dilated by some factor. In our case, the $\mathfrak{sl}_2$ categorical action arises from an action of the subalgebra $U_i$, hence all the degrees will be multiplied by $d_i$.
	\end{rem}

	We want to prove that there is a structure of 2-representation of $\mathfrak{sl}_2$ on the abelian category $\mathcal H[i]$, induced by the functor $\mathrm{ad}_{E_i}$. Let us start by giving all the necessary pieces of data. The weight subcategories of $\mathcal H[i]$ are the subcategories $\mathcal H[i]_w$, $w \in \mathbb Z$, defined by
	\[
		\mathcal H[i]_w = \bigoplus_{\substack{\beta \in Q^+ \\ \left< i^{\vee},\beta\right>=w}} H_{\beta}^i\mathrm{-mod}.
	\]
	The adjoint pair of endofunctors of $\mathcal H[i]$ that we consider is $(\mathrm{ad}_{E_i},\mathrm{ad}_{F_i})$. The functor $\mathrm{ad}_{E_i}$ is the one we studied in Section \ref{cat}, the functor $\mathrm{ad}_{F_i}$ is defined to be its right adjoint, up to a grading shift chosen so that the unit $\eta$ and counit $\varepsilon$ of adjunction have the desired degree. Explicitly, $\mathrm{ad}_{F_i}$ is given by
	\begin{align*}
		\mathrm{ad}_{F_i} = \bigoplus_{\beta \in Q^+} q_i^{1-\left<i^{\vee},\beta\right>}\mathrm{res}_{H_{\beta-i}^i}^{H_{\beta}^i}.
	\end{align*}
	More precisely, if $M \in H_{\beta}^i\mathrm{-mod}$, we have $\mathrm{ad}_{F_i}(M) = q_i^{1-\left< i^{\vee},\beta\right>}1_{i,\beta-i}M$, viewed as an $H_{\beta-i}^i$-module via the right inclusion $H_{\beta-i}^i\rightarrow 1_{i,\beta-i}H_{\beta}^i1_{i,\beta-i}$. The fact that $\mathrm{ad}_{E_i}$ and $\mathrm{ad}_{F_i}$ restrict as desired on the weight subcategories follows simply from $\left< i^{\vee},i\right>=2$. Finally, we have to define $x \in \mathrm{End}(\mathrm{ad}_{E_i})$ and $\tau \in \mathrm{End}(\mathrm{ad}_{E_i}^2)$. These natural transformations are defined by Proposition \ref{adei}.
	
	\begin{thm}\label{2rep}
		This data defines a 2-representation of $\mathfrak{sl}_2$ on $\mathcal H[i]$.
	\end{thm}
	
	\begin{proof}
		We must prove that the conditions in Definition \ref{2repsl2} are satisfied. The fact that $\mathrm{ad}_{E_i}$ is exact was proved in Corollary \ref{exact}, and the exactness of $\mathrm{ad}_{F_i}$ follows similarly from Corollary \ref{proj}. The action of the affine nil Hecke algebra on powers of $\mathrm{ad}_{E_i}$ was proved in Proposition \ref{adei}. Hence only the invertibility of the maps $\rho_w$, $w \in \mathbb{Z}$, remains to be proved, which we do in Theorem \ref{weyl} below.
	\end{proof}
	
	\subsection{Mackey formulas}
	\subsubsection{The KLR algebra case} We follow \cite[Section 3]{kk}. Let $\beta \in Q^+$ and $i \in I$. Recall the left $i$-induction functor
	\[
		\left \{ \begin{array}{rcl}
			H_{\beta}\mathrm{-mod} & \rightarrow & H_{\beta+i}\mathrm{-mod}, \\
			M & \mapsto & E_iM = H_{\beta+i}1_{i,\beta} \otimes_{H_{i,\beta}} (E_i\otimes M).
		\end{array} \right.
	\]
	Since $E_i$ is a free module of rank 1 over $H_{i}$, we have $E_iM \simeq H_{\beta+i}1_{i,\beta} \otimes_{H_{\beta}} M$. We also defined the left $i$-restriction functor $F_i$ as the right adjoint of the left $i$-induction functor. It is given by
	\[
	\left \{ \begin{array}{rcl}
	H_{\beta+i}\mathrm{-mod} & \rightarrow & H_{\beta}, \\
	N & \mapsto & F_i(N) = 1_{i,\beta}N.
	\end{array} \right.
	\]
	There is a Mackey type result for left $i$-induction and left $i$-restriction.
	
	\begin{prop}[{\cite[Theorem 3.6]{kk}}]\label{mackey1}\label{ll}
		There is an isomorphism of graded $H_{\beta}$-bimodules
		\[
		\left \{ \begin{array}{rcl}
		q_i^{2}\left(H_{\beta}1_{i,\beta-i} \otimes_{H_{\beta-i}} 1_{i,\beta-i} H_{\beta} \right) \oplus H_{i,\beta} & \rightarrow & 1_{i,\beta} H_{\beta+i} 1_{i,\beta}, \\
		\left((y\otimes y'),z\right) & \mapsto &  (1_i\diamond y)\tau_n(1_i\diamond y') + z.
		\end{array} \right.
		\]
		It induces a natural isomorphism in $M \in H_{\beta}\mathrm{-mod}$
		\[
			q_i^{2}E_iF_i(M) \oplus M[X_i] \xrightarrow{\sim} F_i(E_iM)
		\]
		where $X_i$ is a variable of degree $2d_i$. More explicitly,
		\[
			M[X_i] = M\otimes K[X_i] \simeq \bigoplus_{k \geqslant 0} q_i^{-2k}M.
		\]
	\end{prop}
	
	Right $i$-induction and right $i$-restriction functors are defined similarly, and satisfy an analogous Mackey type decomposition. Finally, there is a Mackey type result for right $i$-induction and left $i$-restriction.
	
	\begin{prop}[{\cite[Theorem 3.9]{kk}}]\label{lr}
		For all $\beta \in Q^+$, there is a short exact sequence of graded $H_{\beta}$-bimodules
		\[
		0 \rightarrow H_{\beta}1_{\beta-i,i} \otimes_{H_{\beta-i}} 1_{i,\beta-i} H_{\beta} \xrightarrow{R} 1_{i,\beta}H_{\beta+i}1_{\beta,i} \xrightarrow{S} q_i^{\left<i^{\vee},\beta\right>}H_{\beta,i} \rightarrow 0.
		\]
		It induces a short exact sequence which is natural in $M \in H_{\beta}\mathrm{-mod}$
		\[
		0 \rightarrow F_i(M)E_i \xrightarrow{R_M} F_i(ME_i) \xrightarrow{S_M} q_i^{\left< i^{\vee},\beta\right>} M[X_i] \rightarrow 0.
		\]
	\end{prop}
	
	Let us describe explicitly what the morphisms $R$ and $S$ are. The morphism $R$ is given by
	\[
	\left \{ \begin{array}{rcl} 
	H_{\beta}1_{\beta-i,i} \otimes_{H_{\beta-i}} 1_{i,\beta-i} H_{\beta} & \rightarrow & 1_{i,\beta}H_{\beta+i}1_{\beta,i}, \\
	y\otimes y' & \mapsto & (1_i \diamond y)(y' \diamond 1_i).
	\end{array} \right.
	\]	
	To describe the morphism $S$, we use decomposition (\ref{coset}) to write
	\[
		H_{\beta+i}1_{\beta,i} = \bigoplus_{k=0}^n \tau_{[k\downarrow 1]} H_{\beta,i}.
	\]
	If $z \in 1_{i,\beta} H_{\beta+i}1_{\beta,i}$, we can thus decompose $z$ as
	\[
	z=\sum_{k=0}^n \tau_{[k\downarrow 1]} z_k
	\]
	for some unique $z_0,\ldots,z_n \in H_{\beta,i}$ Then $S(z)=z_n$. Equivalently, we can also compute $S(z)$ by decomposing along
	\[
	1_{i,\beta}H_{\beta+i} = \bigoplus_{k=1}^{n+1} H_{i,\beta} \tau_{[n\downarrow k]}
	\]
	and taking the coefficient of $\tau_{[n\downarrow 1]}$ (see \cite[Corollary 3.8]{kk}).
	
	\subsubsection{Mackey formulas in $\mathcal H[i]$} We can now prove the categorical $[e_i,f_i]=h_i$ relation.
	
	\begin{thm}\label{weyl}
		Let $M \in H_{\beta}^i\mathrm{-mod}$ with $\beta \in Q^+$, and let $w = \left<i^{\vee},\beta\right>$. Then we have isomorphisms
		\[
		\left \{ \begin{array}{l}
			\text{if} \ w\geqslant 0, \ \mathrm{ad}_{E_i}\mathrm{ad}_{F_i}(M) \xrightarrow[\sim]{\rho_w} \mathrm{ad}_{F_i}\mathrm{ad}_{E_i}(M) \oplus [w]_iM, \\
			\text{if} \, w\leqslant 0, \ \mathrm{ad}_{E_i}\mathrm{ad}_{F_i}(M) \oplus [-w]_iM \xrightarrow[\sim]{\rho_w} \mathrm{ad}_{F_i}\mathrm{ad}_{E_i}(M).
		\end{array} \right.
		\]
	\end{thm}
	
	We follow the approach of \cite[Theorem 5.2]{kk}, and using jointly Theorem \ref{tauinj} and the Mackey decompositions for KLR algebras.
	\begin{lem}
	Let $M \in H_{\beta}^i\mathrm{-mod}$ with $\beta \in Q^+$, and let $w = \left<i^{\vee},\beta\right>$. There is a commutative diagram with exact rows which is natural in $M$
		\begin{equation}\label{comdia}
			\xymatrix{
				0 \ar[r] & q_i^{-1}F_i(M)E_i \ar[r]^{R_M} \ar[d]_{\tau_{E_i,F_i(M)}} & q_i^{-1}F_i(ME_i) \ar[r]^{S_M} \ar[d]_{F_i(\tau_{E_i,M})} & q_i^{w-1}M[X_i] \ar[d]^{\Phi_M} \ar[r] & 0 \\
				0 \ar[r] & q_i^{1-w}E_iF_i(M) \ar[r]_{T_M} & q_i^{-1-w}F_i(E_iM) \ar[r]_{W_M} & q_i^{-1-w}M[X_i] \ar[r] & 0
			}
		\end{equation}	
	\end{lem}

	\begin{proof}
		The maps $R_M$ and $S_M$ are those of Proposition \ref{lr}, and the maps $T_M$ and $W_M$ are those induced by the isomorphism of Proposition \ref{ll}. The exactness of the two rows also follows from Propositions \ref{ll} and \ref{lr}. The map $\Phi_M$ is defined by commutativity of the diagram. So the only thing to prove is that the leftmost square of the diagram commutes.
		
		Let $y = 1_{\beta-i,i}\otimes_{H_{\beta-i}} 1_{i,\beta-i} m \in F_i(M)E_i$, with $m\in M$. We have
		\begin{align*}
			&\tau_{E_i,F_i(M)}(y) = \tau_{[1\uparrow\vert\beta\vert-1]}1_{i,\beta-i} \otimes_{H_{\beta-i}} 1_{i,\beta-i}m,\\
			&T_M(\tau_{E_i,F_i(M)}(y)) = \tau_{[1\uparrow\vert\beta\vert]}1_{2i,\beta-i} \otimes_{H_{\beta}} m,
		\end{align*}
		and
		\begin{align*}
			&R_M(y) = 1_{i,\beta-i,i} \otimes_{H_{\beta}} m, \\
			& F_i(\tau_{E_i,M})(f_M(y)) = \tau_{[1\uparrow \vert\beta\vert]} 1_{2i,\beta} \otimes_{H_{\beta}} m.
		\end{align*}
		Hence $T_M(\tau_{E_i,F_i(M)}(y))=F_i(\tau_{E_i,M})(R_M(y))$. Since elements $y\in F_i(M)E_i$ of this form generate $F_i(M)E_i$ as an $H_{\beta}$-module, we deduce that $T_M \circ\tau_{E_i,F_i(M)} = F_i(\tau_{E_i,M})\circ R_M$, and the proof is complete.
	\end{proof}

	For $M \in H_{\beta}^i\mathrm{-mod}$, the morphism $\tau_{E_i,F_i(M)}$ is injective by Theorem \ref{tauinj}, and its cokernel is $\mathrm{ad}_{E_i}(F_i(M)) = q_i^{\left< i^{\vee},\beta\right> -1}\mathrm{ad}_{E_i}(\mathrm{ad}_{F_i}(M))$. Since the left $i$-restriction functor $F_i$ is exact, we also know that $F_i(\tau_{E_i,M})$ is injective, and its cokernel is $F_i(\mathrm{ad}_{E_i}(M)) = q_i^{\left<i^{\vee},\beta\right>+1}\mathrm{ad}_{F_i}(\mathrm{ad}_{E_i}(M))$. Hence by the snake lemma there is a short exact sequence which is natural in $M \in H_{\beta}^i\mathrm{-mod}$
	\begin{equation}\label{snake}
		0 \rightarrow \mathrm{ker}(\Phi_M) \rightarrow \mathrm{ad}_{E_i}\mathrm{ad}_{F_i}(M) \xrightarrow{\overline{T}_M} \mathrm{ad}_{F_i}\mathrm{ad}_{E_i}(M) \rightarrow \mathrm{coker}(\Phi_M) \rightarrow 0.
	\end{equation}
	
	\begin{lem}\label{g=phi}
		For all $M \in H_{\beta}^i\mathrm{-mod}$ with $\beta \in Q^+$, the map $\overline{T}_M : \mathrm{ad}_{E_i}\mathrm{ad}_{F_i}(M) \rightarrow \mathrm{ad}_{F_i}\mathrm{ad}_{E_i}(M)$ in the sequence (\ref{snake}) is equal to the map $\varphi$ defined in equation (\ref{varphi}).
	\end{lem}

	\begin{proof}
		We prove this by chasing the diagram giving rise to the exact sequence (\ref{snake}). Let $y$ be an element of $\mathrm{ad}_{E_i}\mathrm{ad}_{F_i}(M)$ of the form $y=1_{i,\beta-i}\otimes_{H_{\beta-i}^i} 1_{i,\beta-i}m$, for $m \in M$. We can write $y$ as the image of the element $z=1_{i,\beta-i}\otimes_{H_{\beta-i}} 1_{i,\beta-i}m \in E_iF_i(M)$ by the canonical quotient morphism. Then we have
		\[
			T_M(z) = \tau_{\vert\beta\vert} 1_{2i,\beta-i}\otimes_{H_{\beta}} m.
		\]
		Hence
		\[
			\overline{T}_M(y) = \tau_{\vert\beta\vert} 1_{2i,\beta-i} \otimes_{H_{\beta}^i} m = \varphi(y).
		\]
		So $\varphi$ and $\overline{T}_M$ coincide on elements $y$ of the form given above. Since these generate $\mathrm{ad}_{E_i}\mathrm{ad}_{F_i}(M)$ as an $H_{\beta}^i$-module, we deduce that $\varphi=\overline{T}_M$.
	\end{proof}
	
	We now want to compute explicitly the kernel and cokernel of the map $\Phi_M$. An element $y$ of $M[X_i]$ will be written formally in the form
	\[
		y= \sum_{k=0}^{\ell} m_kX_i^{k}
	\]
	where $\ell\geqslant 0$, and $m_0,\ldots,m_{\ell} \in M$. Our discussion up to this point is valid for any module in $\mathcal H[i]$, but we now restrict to the generators of $\mathcal H[i]$ to simplify the computations.
	
	\begin{prop}\label{phi}
		Let $M=\mathrm{ad}_{E_i}^{(n)}(E_{j_r}\ldots E_{j_1})$ for some $n\geqslant 0$ and $j_1,\ldots, j_r \in I \setminus \lbrace i \rbrace$. Let $\beta = ni+j_1+\ldots+j_r$ and let $w=\left< i^{\vee},\beta\right>$. Then there exists an invertible element $b$ of $K$ such that for all $m \in M$ and $k \geqslant 0$ we have
		\[
			\Phi_M(mX_i^{k}) \in bmX_i^{k-w} + \sum_{\ell < k-w} MX_i^{\ell},
		\]
		where we put $X_i^{\ell}=0$ when $\ell<0$.
	\end{prop}
		
	\begin{proof}
		The module $M$ is cyclic, generated by the class $c$ of $\tau_{\omega_0\left[r+1,n+r\right]}1_{ni,j_r,\ldots,j_1}$ in $\pi_i(E_i^{(n)}E_{j_r}\ldots E_{j_1})$. So it suffices to prove the result for $m=c$.
		
		Let $c_k=x_{n+r+1}^k\tau_{[n+r\downarrow 1]}1_{\beta,i} \otimes_{H_{\beta}}c \in F_i(ME_i)$. Then $cX_i^{k} = S_M(c_k)$ by definition of $S$. Since the rightmost square in the commutative diagram (\ref{comdia}) commutes, $\Phi_M(cX_i^{k}) = W_M(\tau_{E_i,M}(c_k))$. We have
		\begin{align*}
			\tau_{E_i,M}(c_k) &=  x_{n+r+1}^k\tau_{[n+r\downarrow 1]}\tau_{[1\uparrow n+r]} 1_{i,\beta} \otimes_{H_{\beta}} c \\
			&= x_{n+r+1}^k\tau_{[n+r\downarrow 1]}\tau_{[1\uparrow n+r]}\tau_{\omega_0[r+1,n+r]} 1_{(n+1)i,j_r\ldots j_1} \otimes_{H_{\beta}} x_{r+2}\ldots x_{n+r}^{n-1}c
		\end{align*}		
		the second equality coming from the fact that $c=\tau_{\omega_0[r+1,n+r]}x_{r+2}\ldots x_{n+r}^{n-1}c$. Let
		\[
			c'_k =x_{n+r+1}^k\tau_{[n+r\downarrow 1]}\tau_{[1\uparrow n+r]}\tau_{\omega_0[r+1,n+r]} 1_{(n+1)i,j_r\ldots j_1}.
		\]
		To compute $W_M(\tau_{E_i,M}(c_k))$, we must find the component of $c'_k$ on $H_{i,\beta}$ in the direct sum decomposition
		\[
			1_{i,\beta}H_{\beta+i}1_{i,\beta} = H_{i,\beta} \oplus q_i^{2}(H_{\beta}1_{i,\beta-i}\otimes_{H_{\beta- i}}1_{i,\beta-i}H_{\beta})
		\]
		given by Proposition \ref{ll}. The following computations are done in $H_{\beta+i}1_{(n+1)i,j_r,\ldots,j_1}$, but we omit the idempotent $1_{(n+1)i,j_r,\ldots,j_1}$ on the right to help readability. The explanation of each equality is written below the computation. Let
		\[
			P = \left(\prod_{\ell=1}^{r} Q_{i,j_{\ell}}(x_{r+1},x_{\ell})\right)1_{(n+1)i,j_r,\ldots,j_1} \in P_{\beta+i}.
		\]
		We have
		\begin{align*}
			c'_k &= x_{n+r+1}^k\tau_{[n+r\downarrow 1]}\tau_{[1\uparrow n+r]}\tau_{\omega_0[r+1,n+r]}\\
			&=x_{n+r+1}^k\tau_{[n+r\downarrow r+1]} P\tau_{[r+1\uparrow n+r]}\tau_{\omega_0[r+1,n+r]}\\
			&= x_{n+r+1}^k\tau_{[n+r\downarrow r+1]} P\tau_{\omega_0[r+1,n+r+1]}\\
			&= x_{n+r+1}^k\partial_{[n+r\downarrow r+1]}(P)\tau_{\omega_0[r+1,n+r+1]}\\
			&= x_{n+r+1}^k\partial_{[n+r\downarrow r+1]}(P)\tau_{[r+1\uparrow n+r]}\tau_{\omega_0[r+1,n+r]}.
		\end{align*}		
		In the second equality, we have applied Lemma \ref{multistrand} to replace $\tau_{[r\downarrow 1]}\tau_{[1\uparrow r]}$ by $P$. For the third equality, we have used the fact that $\omega_0[r+1,n+r+1] = s_{[r+1\uparrow n+r]}\omega_0[r+1,n+r]$ to replace $\tau_{[r+1\uparrow n+r]}\tau_{\omega_0[r+1,n+r]}$ by $\tau_{\omega_0[r+1,n+r+1]}$. The fourth equality comes from the relations of the affine nil Hecke algebra proved in Lemma \ref{compnil}. Finally in the fifth equality, we have just reversed the third step.
		
		Before we continue the computation, we write $\partial_{[n+r\downarrow r+1]}(P)$ in a more explicit way. First, given the assumptions on the form of the polynomials $Q_{i,j}$, we can write $P$ in the form
		\[
			P \in tx_{r+1}^{2n-w} + \sum_{\ell<2n-w} x_{r+1}^{\ell}K[x_1,\ldots,x_r].
		\]
		where $t$ is an invertible element of $K$. Hence
		\[
			\partial_{[n+r\downarrow r+1]}(P) = \sum_{\ell \leqslant n-w} x_{n+r+1}^{\ell}p_{\ell}
		\]
		for some $p_{\ell} \in K[x_1,\ldots,x_{n+r}]$, with $p_{n-w} = (-1)^nt$. We can now resume the computation of $c'_k$. With this form for $\partial_{[n+r\downarrow r+1]}(P)$, we have 
		\begin{align*}
			c'_k &= \sum_{\ell \leqslant n-w}x_{n+r+1}^{k+\ell}p_{\ell}\tau_{[r+1\uparrow n+r]}\tau_{\omega_0[r+1,n+r]} \\
			&= \sum_{\ell \leqslant n-w}p_{\ell}\tau_{[r+1\uparrow n+r-1]}x_{n+r+1}^{k+l}\tau_{n+r}\tau_{\omega_0[r+1,n+r]} \\
			&= \sum_{\ell \leqslant n-w}p_l\tau_{[r+1\uparrow n+r]}x_{n+r}^{k+{\ell}}\tau_{\omega_0[r+1,n+r]} \\
			& \quad + \sum_{\substack{\ell \leqslant n-w \\ m \leqslant k+\ell-1}} p_{\ell}\tau_{[r+1\uparrow n+r-1]}x_{n+r+1}^{m}x_{n+r}^{k+\ell-1-m}\tau_{\omega_0[r+1,n+r]}.
		\end{align*}
		
		For the second equality, we have used relation (\ref{Taux}) of Definition \ref{klr} from to move the factor $x_{n+r+1}^{k+\ell}$ all the way to the left of $\tau_{n+r}$. The third equality is what we get from the commutation relation of $x_{n+r+1}^{k+\ell}$ and $\tau_{n+r}$ from Lemma \ref{compnil}. However
		\[
			\sum_{\ell \leqslant n-w}p_{\ell}\tau_{[r+1\uparrow n+r]}x_{n+r}^{k+\ell}\tau_{\omega_0[r+1,n+r]} \in H_{\beta}1_{i,\beta-i}\otimes_{H_{\beta- i}}1_{i,\beta-i}H_{\beta}.
		\]
		Hence, modulo $H_{\beta}1_{i,\beta-i}\otimes_{H_{\beta- i}}1_{i,\beta-i}H_{\beta}$ we have
		\begin{align*}
			c'_k &= \sum_{\substack{\ell \leqslant n-w \\ m \leqslant k+\ell-1}} p_{\ell}\tau_{[r+1\uparrow n+r-1]}x_{n+r+1}^{m}x_{n+r}^{k+\ell-1-m}\tau_{\omega_0[r+1,n+r]} \\
			&= \sum_{\substack{\ell \leqslant n-w \\ m \leqslant k+\ell-1}} p_{\ell}x_{n+r+1}^{m}\partial_{[r+1\uparrow n+r-1]}(x_{n+r}^{k+{\ell}-1-m})\tau_{\omega_0[r+1,n+r]}.\\
		\end{align*}
		We have
		\[
			\partial_{[r+1\uparrow n+r-1]}(x_{n+r}^{k+\ell-1-m}) = \left \{ \begin{array}{ll}
				0 & \text{if} \ k+\ell-1-m<n-1, \\
				1 & \text{if} \ k+\ell-1-m=n-1.
			\end{array} \right.
		\]
		Thus the previous equation can be written
		\begin{align*}
			c'_k &= \sum_{\substack{\ell \leqslant n-w \\ m \leqslant k+\ell-n}} p_{\ell}x_{n+r+1}^{m}\partial_{[r+1\uparrow n+r-1]}(x_{n+r}^{\ell-1-m})\tau_{\omega_0[r+1,n+r]} \\
			& \in (-1)^ntx_{n+r+1}^{k-w}\tau_{\omega_0[r+1,n+r]} + \sum_{\ell<k-w} x_{n+r+1}^{\ell} (1_i\diamond H_{\beta}).
		\end{align*}
		We can now conclude that
		\begin{align*}
			\Phi_M(cX_i^{k}) &= W_M( c'_k \otimes_{\beta} x^{r+2}\ldots x_{r+n}^{n-1}c)\\
			&\in (-1)^nt cX_i^{k-w} + \sum_{\ell<k-w} MX_i^{\ell}.
		\end{align*}
		If we let $b=(-1)^nt \in K^{\times}$, then the proof is complete.
	\end{proof}

	\begin{rem} Since the map $\Phi_M$ is natural in $M$, Proposition \ref{phi} actually holds for $M$ any subquotient of a module of the form $\mathrm{ad}_{E_i}^{(n)}(E_{j_r}\ldots E_{j_1})$.
	\end{rem}
		
	As an immediate consequence of Proposition \ref{phi}, we get an explicit description of the kernel and cokernel of $\Phi_M$.
	
	\begin{cor}\label{ckphi} Let $M$ be a subquotient of $\mathrm{ad}_{E_i}^{(n)}(E_{j_r}\ldots E_{j_1})$ for $n\geqslant 0$ and $j_1,\ldots,j_r \in I \setminus \lbrace i \rbrace$ and $w$ defined as above. Then we have
		\begin{equation*}
		\left \{ \begin{array}{l}
			\text{if } w \geqslant 0, \ \mathrm{ker}(\Phi_M) = q_i^{w-1}\left(\displaystyle\bigoplus_{k=0}^{w-1}MX_i^k\right) = [w]_iM \quad \text{and} \quad \mathrm{coker}(\Phi_M)= 0, \\
			\\
			\text{if } w \leqslant 0, \ \mathrm{ker}(\Phi_M) = 0 \quad \text{and} \quad \mathrm{coker}(\Phi_M) = q_i^{-w-1}\left(\displaystyle\bigoplus_{k=0}^{-w-1}MX_i^k\right)=[-w]_iM.
		\end{array} \right.
	\end{equation*}
	\end{cor}
	
	With this proved, we can now show that the maps $\rho_w$ are isomorphisms. We separate two cases depending on whether the weight is positive or negative.
	
	\begin{prop}
		Let $M$ be a subquotient of $\mathrm{ad}_{E_i}^{(n)}(E_{j_r}\ldots E_{j_1})$ for some $n\geqslant 0$ and $j_1,\ldots, j_r \in I \setminus \lbrace i \rbrace$. Let $\beta = ni+j_1+\ldots+j_r$ and let $w=\left< i^{\vee},\beta\right>$. Assume that $w \leqslant 0$. Then $\rho_w$ is an isomorphism.
	\end{prop}

	\begin{proof}
		Given Lemma \ref{g=phi} and Corollary \ref{ckphi}, the short exact sequence (\ref{snake}) becomes
		\[
			0\rightarrow \mathrm{ad}_{E_i}\mathrm{ad}_{F_i}(M) \xrightarrow{\varphi} \mathrm{ad}_{F_i}\mathrm{ad}_{E_i}(M) \rightarrow [-w]_iM \rightarrow 0.
		\]
		We also have a commutative diagram
		\[
			\xymatrix{
				q_i^{-1-w}F_i(E_iM) \ar[r]^{W_M} \ar[d] & q_i^{-1-w}M[X_i] \ar[d] & [-w]_iM \ar[l]_-{\supseteq} \ar[ld]_{\sim} \\
				\mathrm{ad}_{F_i}\mathrm{ad}_{E_i}(M) \ar[r] & \mathrm{coker}(\Phi_M) &
			}
		\]
		The map $W_M$ admits a right inverse $A_M$ given by the direct sum decomposition
		\[
			1_{i,\beta}H_{i+\beta} 1_{i,\beta} = H_{i,\beta} \oplus q_i^2(H_{\beta}1_{i,\beta-i}\otimes_{H_{\beta-i}}1_{i,\beta-i}H_{\beta}).
		\]
		This right inverse provides a map $\overline{A}_M : [-w]_iM \rightarrow \mathrm{ad}_{F_i}\mathrm{ad}_{E_i}(M)$ that splits the short exact sequence above. Let us compute the map $\overline{A}_M$. Let $y \in [-w]_iM$, we view $y$ as an element of $q_i^{-1-w}M[X_i]$ that we write in the form
		\[
			y = \sum_{k=0}^{-w-1} m_kX_i^{k}
		\]
		where $m_{0},\ldots,m_{-w-1} \in M$. Then we have
		\[
			A_M(y) = \sum_{k=0}^{-w-1} x_{n+r+1}^k1_{i,\beta}\otimes_{H_{\beta}} m_k
		\]
		So
		\begin{align*}
			\overline{A}_M(y) & = \sum_{k=0}^{-w-1} x_{n+r+1}^k1_{i,\beta}\otimes_{H_{\beta}^i} m_k \\
			&= \sum_{k=0}^{-w-1} x_{n+r+1}^k\eta_M(m_k).
		\end{align*}
		Hence $\overline{A}_M$ is exactly the component of $\rho_w$ on $[-w]_iM$, so $\rho_w$ is an isomorphism.
	\end{proof}

	\begin{prop}
		Let $M$ be a subquotient of $\mathrm{ad}_{E_i}^{(n)}(E_{j_r}\ldots E_{j_1})$ for some $n\geqslant 0$ and $j_1,\ldots, j_r \in I \setminus \lbrace i \rbrace$. Let $\beta = ni+j_1+\ldots+j_r$ and let $w=\left< i^{\vee},\beta\right>$. Assume that $w \geqslant 0$. Then $\rho_w$ is an isomorphism.
	\end{prop}

	\begin{proof}
		Given Lemma \ref{g=phi} and Corollary \ref{ckphi}, the short exact sequence (\ref{snake}) becomes
		\[
			0\rightarrow [w]_iM \xrightarrow{\delta} \mathrm{ad}_{E_i}\mathrm{ad}_{F_i}(M) \xrightarrow{\varphi} \mathrm{ad}_{F_i}\mathrm{ad}_{E_i}(M) \rightarrow 0
		\]
		where $\delta$ is the connecting map given by the snake lemma. Let $\Gamma$ be the component of $\rho_w$ which maps in $[w]_iM$. To prove that $\rho_w$ is an isomorphism, it suffices to prove that $\Gamma\delta$ is an isomorphism. To this end, we prove that there exists an invertible element $b\in K$ such that for all $m \in M$ and $k\in \left \{ 0,\ldots,w-1\right \}$, we have
		\begin{equation}\label{gamdel}
			\Gamma\delta(mX_i^{k})=bmX_i^{w-1-k} + \sum_{w-1-k<\ell\leqslant w-1} MX_{i}^{\ell}.
		\end{equation}
		By naturality, it suffices to prove this for $M=\mathrm{ad}_{E_i}^{(n)}(E_{j_r}\ldots E_{j_1})$. In this case, $M$ is cyclic, generated by the class $c$ of $\tau_{\omega_0[r+1,n+r]}1_{ni,j_r\ldots,j_1}$ in $\pi_i(E_i^{(n)}E_{j_r}\ldots E_{j_1})$, and it suffices to treat the case $m=c$. We start by computing $\delta(cX_i^{k})$, which we do by chasing the diagram (\ref{comdia}). We resume the computations done in the proof of Proposition \ref{phi}. Recall that $cX_i^{k} = S_M(c_k)$ with $c_k=x_{n+r+1}^k\tau_{[n+r\downarrow 1]}1_{\beta,i} \otimes_{H_{\beta}} x_{r+2}\ldots x_{r+n}^{n-1}c$. We proved that $\tau_{E_i,M}(c_k) =c'_k\otimes_{H_{\beta}}x_{r+2}\ldots x_{r+n}^{n-1}c$, where $c'_k$ is an element proved to be equal to
		\begin{align*}
			c'_k &= \sum_{\ell \leqslant n-w}p_{\ell}\tau_{[r+1\uparrow n+r]}x_{n+r}^{k+\ell}\tau_{\omega_0[r+1,n+r]} \\
			& \quad + \sum_{\substack{\ell \leqslant n-w \\ m \leqslant k+\ell-1}} p_{\ell}x_{n+r+1}^{m}\partial_{[r+1\uparrow n+r-1]}(x_{n+r}^{k+\ell-1-m})\tau_{\omega_0[r+1,n+r]}
		\end{align*}
		where $p_{\ell} \in K[x_1,\ldots,x_{n+r}]$, with $p_{n-w} \in K^{\times}$. Since $cX_i^k \in \mathrm{ker}(\Phi_M)$, the second sum vanishes (this can also be checked directly from the fact that $k\leqslant w-1$). Thus we have
		\[
			\tau_{E_i,M}(c_k) = T_M\left( \sum_{\ell\leqslant n-w} p_{\ell}\tau_{[r+1\uparrow n+r-1]}1_{i,\beta-i}\otimes_{H_{\beta-i}}x_{n+r}^{k+{\ell}}c\right),
		\]
		and
		\[
			\delta(cX_i^k) = \sum_{\ell\leqslant n-w} p_{\ell}\tau_{[r+1\uparrow n+r-1]}1_{i,\beta-i}\otimes_{H_{\beta-i}^i}x_{n+r}^{k+\ell}c.
		\]
		So
		\begin{align*}
			\Gamma\delta(cX_i^k) &= \sum_{\substack{\ell \leqslant n-w \\ u \leqslant w-1}} p_{\ell}\partial_{[r+1\uparrow n+r-1]}(x_{n+r}^{k+\ell+u})cX_i^u.
		\end{align*}
		We have $\partial_{[r+1\uparrow n+r-1]}(x_{n+r}^{k+\ell+u})=0$ unless $u \geqslant w-1-k$, and we get
		\[
			\Gamma\delta(cX_i^k) = p_{n-w}cX_i^{w-1-k} + \sum_{w-1-k<\ell\leqslant w-1} MX_{i}^{\ell}.
		\]
		Since $p_{n-w} \in K^{\times}$, we have established equation (\ref{gamdel}). So $\Gamma\delta$ is an isomorphism and the proof is complete.
	\end{proof}
	
	 Hence, we have proved that Theorem \ref{weyl} holds for all the subquotients of modules of the form $\mathrm{ad}_{E_i}^{(n)}(E_{j_r}\ldots E_{j_1})$. To conclude that Theorem \ref{weyl} holds for every module in $\mathcal H[i]$, we use the following obvious lemma.
	 
	\begin{lem}
		Let $\mathcal C, \mathcal D$ be abelian categories, $G,G':\mathcal C \rightarrow \mathcal D$ be exact functors, and $\alpha: G \rightarrow G'$ be a natural transformation. Assume that we have a collection $\mathcal M$ of objects of $\mathcal C$ such that $\mathcal C$ is generated by $\mathcal M$ as a Serre subcategory, and $\alpha_M$ is an isomorphism for any subquotient $M$ of an object of $\mathcal M$. Then $\alpha$ is an isomorphism.   
	\end{lem}

	By Theorem \ref{genforhi}, modules of the form $\mathrm{ad}_{E_i}^{(n)}(E_{j_r}\ldots E_{j_1})$ generate $\mathcal H[i]$ as a Serre subcategory. Since Theorem \ref{weyl} holds for their subquotients, we conclude that it holds for every module in $\mathcal H[i]$. 

	\bigskip
	
	\section{Projective resolutions}
	
	In $U$, there are simple formulas for $\mathrm{ad}_{e_i}^n$ and $\mathrm{ad}_{e_i}^{(n)}$ that can be obtained by induction. For all $y \in U$ of weight $\beta \in Q^+$ we have
	\begin{equation}\label{sum}\begin{split}
		&\mathrm{ad}_{e_i}^n(y) = \sum_{k=0}^{n} (-1)^kq_i^{k(n-1+\left<i^{\vee},\beta\right>)} \left[ \begin{array}{c} n \\ k \end{array} \right]_i e_i^{n-k}ye_i^k, \\
		&\mathrm{ad}_{e_i}^{(n)}(y) = \sum_{k=0}^{n} (-1)^kq_i^{k(n-1+\left<i^{\vee},\beta\right>)} e_i^{(n-k)}ye_i^{(k)},
	\end{split}\end{equation}
	where \[\left[ \begin{array}{c} n \\ k \end{array} \right]_i = \frac{[n]_i!}{[k]_i![n-k]_i!}\] is the $q_i$-binomial coefficient. In this section, we introduce complexes $\mathrm{Ad}_{E_i}^n(M)$ and $\mathrm{Ad}_{E_i}^{(n)}(M)$ for $M \in \mathcal H[i]$. These complexes categorify the alternating sums (\ref{sum}). We prove that their cohomology is concentrated in top degree, and is equal to $\mathrm{ad}_{E_i}^n(M)$ and $\mathrm{ad}_{E_i}^{(n)}(M)$ respectively. From this we obtain projective resolutions for the generators of $\mathcal H[i]$.
	
	The higher order quantum Serre relations in $U^+$ state that for $j \in I \setminus \lbrace i \rbrace$, $m\geqslant 0$ and $n>-mc_{i,j}$ we have
	\[
		\mathrm{ad}_{e_i}^{(n)}(e_j^m)=0.
	\]
	We can categorify this in two (equivalent) ways. Firstly, the module $\mathrm{ad}_{E_i}^{(n)}(E_j^m)$ is zero when $n>-mc_{i,j}$, as follows from Proposition \ref{vanish}. We also prove that the complex $\mathrm{Ad}_{E_i}^{(n)}(E_j^{m})$ is null-homotopic when $n>-mc_{i,j}$, which categorifies the vanishing of the alternating sum.
	
	\begin{conv}
		All the complexes we write are cochain complexes (so they have degree +1 differential). Given a complex $M = (\ldots \rightarrow M^r \xrightarrow{d^r} M^{r+1} \rightarrow \ldots)$ we denote by $H^r(M)= \mathrm{ker}(d^r)/\mathrm{Im}(d^{r-1})$ its $r^{\mathrm{th}}$ cohomology group. Given complexes $M,N$ with respective differentials $d_M,d_N$, and a morphism of complexes $f : M \rightarrow N$, its cone is the complex $\mathrm{Cone}(f)$ defined by
		\[
			\mathrm{Cone}(f)^r = M^{r+1} \oplus N^r,
		\] 
		with differential
		\[
			\left[\begin{array}{cc}
				-d_M^{r+1} & 0 \\
				f_r & d_N^{r}
			\end{array}\right] : M^{r+1} \oplus N^r \rightarrow M^{r+2} \oplus N^{r+1}.
		\]
		Then there is a long exact sequence in cohomology
		\[
			\cdots \rightarrow H^k(M) \xrightarrow{H^k(f)} H^k(N) \rightarrow H^k(\mathrm{Cone}(f)) \rightarrow H^{k+1}(M) \rightarrow \cdots.
		\]
	\end{conv}

	\medskip
	
	\subsection{Adjoint action as a complex} We start by constructing the complex $\mathrm{Ad}_{E_i}^n(M)$ for $M \in \mathcal H[i]$. To do this, we need to express the $q_i$-binomial coefficients in a combinatorial way. Let $\mathcal P_n^k$ be the set of subsets of $\lbrace1,\ldots,n\rbrace$ containing $k$ elements. For $S \in \mathcal P_n^k$, we let $\Sigma(S)$ be the sum of the elements of $S$. Then we have (see \cite[Theorem 6.1]{qcal})
	\[
		\left[ \begin{array}{c} n \\ k \end{array} \right]_i =q_i^{-k(n+1)} \sum_{S \in \mathcal P_n^k} q_i^{2\Sigma(S)}.
	\]
	With this expression, we can rewrite $\mathrm{ad}_{e_i}^n(y)$ as
	\begin{equation}
	\mathrm{ad}_{e_i}^n(y) = \sum_{k=0}^{n} (-1)^kq_i^{k\left< i^{\vee},\beta\right>}\left(\sum_{S \in \mathcal P_n^k} q_i^{2(\Sigma(S) -k)}\right)e_i^{n-k}ye_i^k.
	\end{equation}
	We construct a complex which categorifies this alternating sum. The construction is based on the following elementary lemma.
	
	\begin{lem}\label{braid}
		Let $\beta \in Q^+$ of height $n$. In $H_{n+2}$, the following relations hold
		\begin{align*}
			& \tau_{[2\uparrow n+1]}\tau_{[1\uparrow n]}\tau_{n+1}1_{2i,\beta} = \tau_1\tau_{[2\uparrow n+1]}\tau_{[1\uparrow n]}1_{2i,\beta}, \\
			& \tau_{n+1} \tau_{[n\downarrow 1]}\tau_{[n+1\downarrow 2]}1_{\beta,2i} = \tau_{[n\downarrow 1]}\tau_{[n+1\downarrow 2]}\tau_11_{\beta,2i}.
		\end{align*}
	\end{lem}

	\begin{proof}
		This just follows from applying the braid relation (\ref{Taubraid}) of Definition \ref{klr} repeatedly, noticing that we are always in the case where the polynomial error term is zero.
	\end{proof}

	\begin{defi}
		Let $M \in H_{\beta}^i\mathrm{-mod}$, $n \geqslant 0$ and put $w=\left< i^{\vee},\beta\right>$. We define a complex $\mathrm{Ad}_{E_i}^n(M)$ of $H_{\beta+ni}$-modules of the form
		\[
			0 \rightarrow q_i^{n(w+n-1)}ME_i^n \rightarrow \cdots \rightarrow \bigoplus_{S \in \mathcal{P}_n^k} q_i^{kw+2(\Sigma(S)-k)} E_i^{n-k}ME_i^k \rightarrow \cdots \rightarrow E_i^nM \rightarrow 0
		\]
		where $E_i^nM$ is in cohomological degree 0. To define the differential of $\mathrm{Ad}_{E_i}^n(M)$, we first define a map $d_{S,S'} : q_i^{kw+2(\Sigma(S)-k)} E_i^{n-k}ME_i^k \rightarrow q_i^{(k-1)w+2(\Sigma(S')-k+1)} E_i^{n+1-k}ME_i^{k-1}$ for any $S \in \mathcal{P}_n^k$ and $S' \in \mathcal P_n^{k-1}$ as follows:
		\begin{itemize}
			\item if $S'$ is not a subset of $S$, then $d_{S,S'}=0$,
			\item if $S'=S\setminus \lbrace \ell \rbrace$, let $a$ be the number of elements of $S$ which are $>\ell$. Then we define
			\begin{align*}
				d_{S,S'} &= (-1)^a\left(\tau_{[1\uparrow \ell-1-a]}ME_i^{k-1}\right) \circ \left(E_i^{n-k}\tau_{E_i,M}E_i^{k-1} \right) \circ \left(E_i^{n-k}M\tau_{[k-a\uparrow k-1]} \right) \\
				&= \left \{ \begin{array}{rcl}
					E_i^{n-k}ME_i^k & \rightarrow & E_i^{n-k+1}ME_i^{k-1}, \\
					1_{(n-k)i,\beta,i} \otimes_{H_{\beta}} m & \mapsto & (-1)^a \tau_{[k-a+1\uparrow k-a+\ell+\vert \beta \vert]} 1_{(n-k+1)i,\beta,(k-1)i} \otimes_{H_{\beta}} m.
				\end{array} \right.
			\end{align*}
			Notice that $d_{S,S'}$ has indeed the required degree $-d_i(w+2(l-1))$. 
		\end{itemize}
		Then, the differential of $\mathrm{Ad}_{E_i}^n(M)$ is defined to be the direct sum of all maps $d_{S,S'}$. This ends the definition of $\mathrm{Ad}_{E_i}^n(M)$.
	\end{defi}

	It can be checked directly using Lemma \ref{braid} that the differential we defined on $\mathrm{Ad}_{E_i}^n(M)$ indeed squares to 0. Alternatively, we can also notice that $\mathrm{Ad}_{E_i}^n(M)$ can be constructed inductively, as the following proposition explains.
	
	\begin{prop}\label{Adind}
		Let $M \in H_{\beta}^i\mathrm{-mod}$ and let $n \geqslant 0$. There is a morphism of complexes
		\[
			\tau_{E_i,\mathrm{Ad}_{E_i}^n(M)} : q_i^{\left<i^{\vee},\beta\right>+2n}\mathrm{Ad}_{E_i}^n(M)E_i \rightarrow E_i\mathrm{Ad}_{E_i}^n(M)
		\]
		defined on the component of cohomological degree $-k$ as the direct sum over $S \in \mathcal P_n^k$ of the maps
		\begin{align*}
		\tau_S &= \left(\tau_{[1\uparrow n-k]}ME_i^k\right) \circ \left(E_i^{n-k}\tau_{E_i,M}E_i^{k}\right) \circ \left(E_i^{n-k}M\tau_{[1\uparrow k]}\right)\\
		 &=\left \{ \begin{array}{rcl}
				E_i^{n-k}ME_i^{k+1} & \rightarrow & E_i^{n-k+1}ME_i^k,\\
				1_{(n-k)i,\beta,(k+1)i} \otimes_{H_{\beta}} m & \mapsto & \tau_{[1\uparrow n+\vert\beta\vert]}1_{(n-k+1)i,\beta,ki}\otimes_{H_{\beta}} m.
			\end{array} \right.
		\end{align*}
		The cone of this morphism is $\mathrm{Ad}_{E_i}^{n+1}(M)$.	
	\end{prop}

	\begin{proof}
		We will denote the differential of $\mathrm{Ad}_{E_i}^n(M)$ by $d^n$. Let us check that the defined map is indeed a morphism of complexes. We need to prove that for $S \in \mathcal P_n^k$ and $S' \in \mathcal P_n^{k-1}$ the diagram
		\[
			\xymatrix{
			E_i^{n-k}ME_i^{k+1} \ar[r]^{d^n_{S,S'}E_i} \ar[d]_{\tau_S} &  E_i^{n-k+1}ME_i^k \ar[d]^{\tau_{S'}} \\
			E_i^{n-k+1}ME_i^k \ar[r]_{E_id^n_{S,S'}} & E_i^{n-k+2}ME_i^{k-1}	
		}
		\]
		commutes. It is clear if $S'$ is not a subset of $S$, because then $d^n_{S,S'}=0$. If $S'=S\setminus \lbrace \ell \rbrace$, let $a$ be the number of elements in $S$ that are $>\ell$. We have
		\begin{align*}
			&\tau_{S'}\circ d^n_{S,S'}E_i = \left \{ \begin{array}{rcl}
			 E_i^{n-k}ME_i^{k+1} & \rightarrow & E_i^{n-k+2}ME_i^{k-1} \\
			 1_{(n-k)i,\beta,(k+1)i}\otimes_{H_{\beta}} m & \mapsto & (-1)^a \tau_{[k-a+2\uparrow k-a+\ell+\vert \beta \vert +1]} \tau_{[1\uparrow n+\beta]} 1_{(n-k+2)i,\beta,ki} \otimes_{H_{\beta}} m
			\end{array} \right. \\
			& E_id^n_{S,S'} \circ \tau_S = \left \{ \begin{array}{rcl}
			E_i^{n-k}ME_i^{k+1} & \rightarrow & E_i^{n-k+2}ME_i^{k-1} \\
			1_{(n-k)i,\beta,(k+1)i}\otimes_{H_{\beta}} m & \mapsto & (-1)^a \tau_{[1\uparrow n+\beta]}\tau_{[k-a+1\uparrow k-a+\ell+\vert \beta \vert]} 1_{(n-k+2)i,\beta,ki} \otimes_{H_{\beta}} m
			\end{array} \right.
		\end{align*}
		By Lemma \ref{braid}, we have
		\[
			\tau_{[k-a+2\uparrow k-a+\ell+\vert \beta \vert +1]} \tau_{[1\uparrow n+\beta]} 1_{(n-k+2)i,\beta,ki} = \tau_{[1\uparrow n+\beta]} \tau_{[1\uparrow n+\beta]}\tau_{[k-a+1\uparrow k-a+\ell+\vert \beta \vert]} 1_{(n-k+2)i,\beta,ki}.
		\]
		Hence the diagram commutes, and the morphism of complexes $\tau_{E_i,\mathrm{Ad}_{E_i}^n(M)}$ is well-defined.
		
		We now need to check that the cone of this morphism is $\mathrm{Ad}_{E_i}^{n+1}(M)$. Let $w=\left< i^{\vee},\beta\right>$. The term of the cone in cohomological degree $-k$ is given by
		\begin{align*}
			&q_i^{w+2n}\left(\bigoplus_{S \in \mathcal{P}_n^{k-1}} q_i^{(k-1)w+2(\Sigma(S)-k+1)}E_i^{n+1-k}ME_i^k \right) \bigoplus \left( \bigoplus_{S \in \mathcal{P}_n^k} q_i^{kw+2(\Sigma(S)-k)} E_i^{n-k}ME_i^k \right) \\
			=& \left(\bigoplus_{S \in \mathcal{P}_n^{k-1}} q_i^{kw+2(\Sigma(S)+n+1-k)}E_i^{n+1-k}ME_i^k \right) \bigoplus \left( \bigoplus_{S \in \mathcal{P}_n^k} q_i^{kw+2(\Sigma(S)-k)} E_i^{n-k}ME_i^k \right).
		\end{align*}
		The terms of the first sum can be indexed by the $S \in \mathcal{P}_{n+1}^k$ such that $n+1 \in S$, and the terms of the second sum by the $S \in \mathcal{P}_{n+1}^k$ such that $n+1 \notin S$, and we see that we indeed get the term of cohomological degree $-k$ of $\mathrm{Ad}_{E_i}^{n+1}(M)$. Finally, we check that the differential of the cone matches that of $\mathrm{Ad}_{E_i}^{n+1}(M)$. Let $S \in \mathcal P_{n+1}^k$ and $S' \in \mathcal P_{n+1}^{k-1}$. With the parametrization of the terms of the cone given above, we let $d'_{S,S'}$ be the differential of the cone from the term indexed by $S$ to the term indexed by $S'$. By definition of the cone, we have
		\[
			d'_{S,S'} = \left \{ \begin{array}{ll}
				\tau_{S} = d^{n+1}_{S,S\setminus \lbrace n+1\rbrace}& \text{if } n+1 \in S \text{ and } S'=S\setminus \lbrace n+1 \rbrace, \\
				0 & \text{if } n+1 \in S \text{ and } S'\neq S\setminus \lbrace n+1 \rbrace, \\
				E_id^n_{S,S'} & \text{if } n+1 \notin S,S', \\
				-d^n_{S\setminus \lbrace n+1 \rbrace, S' \setminus \lbrace n+1 \rbrace}E_i & \text{if } n+1 \in S,S'.
			\end{array} \right.
		\]
		From this, we see easily that $d'_{S,S'}=d^{n+1}_{S,S'}$.
	\end{proof}

	We can now prove the main result regarding the cohomology of $\mathrm{Ad}_{E_i}^n$.

	\begin{thm}\label{coho}
		Let $M \in H_{\beta}^i\mathrm{-mod}$ and let $n \geqslant 0$. Then
		\[
			H^k(\mathrm{Ad}_{E_i}^n(M)) = \left \{ \begin{array}{ll}
				\mathrm{ad}_{E_i}^n(M) & \text{ if } k=0, \\
				0 & \text{ otherwise.}
			\end{array} \right.
		\]
	\end{thm}

	\begin{proof}
		We proceed by induction on $n$. The result is clear if $n=0$. Assume that the theorem is proved for $\mathrm{Ad}_{E_i}^n(M)$. By Proposition \ref{Adind} there is a distinguished triangle
		\[
			q_i^{\left<i^{\vee},\beta\right>+2n}\mathrm{Ad}_{E_i}^n(M)E_i \rightarrow E_i\mathrm{Ad}_{E_i}^n(M) \rightarrow \mathrm{Ad}_{E_i}^{n+1}(M) \xrightarrow{[1]}.
		\]
		It gives rise to a long exact sequence in cohomology
		\[
			\cdots \rightarrow H^{k-1}(\mathrm{Ad}_{E_i}^{n+1}(M))\rightarrow q_i^{\left<i^{\vee},\beta\right>+2n}H^k(\mathrm{Ad}_{E_i}^n(M))E_i \rightarrow E_iH^k(\mathrm{Ad}_{E_i}^n(M)) \rightarrow H^k(\mathrm{Ad}_{E_i}^{n+1}(M)) \rightarrow \cdots
		\]
		We deduce from this exact sequence that $H^k(\mathrm{Ad}_{E_i}^{n+1}(M))=0$ if $k\neq0,-1$, and that there is an exact sequence
		\[
			0\rightarrow H^{-1}(\mathrm{Ad}_{E_i}^{n+1}(M))\rightarrow q_i^{\left<i^{\vee},\beta\right>+2n}\mathrm{ad}_{E_i}^n(M)E_i \rightarrow E_i\mathrm{ad}_{E_i}^n(M) \rightarrow H^0(\mathrm{Ad}_{E_i}^{n+1}(M)) \rightarrow 0.
		\]
		From the definition of the map $q_i^{\left<i^{\vee},\beta\right>+2n}\mathrm{Ad}_{E_i}^n(M)E_i \rightarrow E_i\mathrm{Ad}_{E_i}^n(M)$ in Proposition \ref{Adind}, we see that the induced map $q_i^{\left<i^{\vee},\beta\right>+2n}\mathrm{ad}_{E_i}^n(M)E_i \rightarrow E_i\mathrm{ad}_{E_i}^n(M)$ is simply $\tau_{E_i, \mathrm{ad}_{E_i}^n(M)}$. By Theorem \ref{tauinj}, this map is injective with cokernel $\mathrm{ad}_{E_i}^{n+1}(M)$. The result follows.
	\end{proof}

	\medskip

	\subsection{Divided powers complex} We now introduce a similar complex for the divided powers.
	
	\begin{defi}
		Let $M \in H_{\beta}^i\mathrm{-mod}$ with $\beta$ of height $r$, and $n \geqslant 0$. We define a complex $\mathrm{Ad}_{E_i}^{(n)}(M)$ of the form
		\[
			0 \rightarrow q_i^{n(n+w-1)} ME_i^{(n)} \rightarrow \cdots \rightarrow q_i^{k(n+w-1)} E_i^{(n-k)}ME_i^{(k)} \rightarrow \cdots \rightarrow E_i^{(n)}M \rightarrow 0
		\]
		where $E_i^{(n)}M$ is in cohomological degree 0 and $w=\left< i^{\vee},\beta\right>$. The differential of $\mathrm{Ad}_{E_i}^{(n)}(M)$ is defined by
		\begin{align*}
				d_k&= (-1)^{k-1}\left(\tau_{[1 \uparrow n-k]}ME_i^{k-1}\right)\circ\left(E_i^{n-k}\tau_{E_i,M} E_i^{k-1}\right)\\
				& =\left \{ \begin{array}{rcl} q_i^{k(n+w-1)} E_i^{(n-k)}ME_i^{(k)}&\rightarrow&q_i^{(k-1)(n+w-1)} E_i^{(n-k+1)}ME_i^{(k-1)} \\
				\tau_{\omega_0[k+r,n+r]}\tau_{\omega_0[1,k]} 1_{(n-k)i,\beta,ki} \otimes_{H_{\beta}}m &\mapsto& (-1)^{k-1} \tau_{\omega_0[k+r,n+r]}\tau_{\omega_0[1,k]}\tau_{[k\uparrow n+r]} 1_{(n-k+1)i,\beta,(k-1)i} \otimes_{H_{\beta}}m
				\end{array} \right.
		\end{align*}
	\end{defi}

	Let us explain why this is well-defined. First, by Lemma \ref{braid}, we have
	\[
		\tau_{\omega_0[k+r,n+r]}\tau_{\omega_0[1,k]}\tau_{[k\uparrow n+r]} 1_{(n-k+1)i,\beta,(k-1)i} = \tau_{[1\uparrow k+r-1]}\tau_{\omega_0[k-1+r,n+n]}\tau_{\omega_0[1,k-1]}1_{(n-k+1)i,\beta,(k-1)i}
	\]
	which proves that $d_k$ indeed takes values in the summand $E_i^{(n-k+1)}ME_i^{(k-1)}$ of $E_i^{n-k+1}ME_i^{k-1}$. Furthermore, in $H_{\beta+ni}1_{(n-k+2)i,\beta,(k-2)i}$ we have
	\begin{align*}
		\tau_{\omega_0[k+r,n+r]}\tau_{\omega_0[1,k]}\tau_{[k\uparrow n+r]}\tau_{[k-1\uparrow n+r]} &= \tau_{\omega_0[k+r,n+r]}\tau_{\omega_0[1,k]}\tau_{k-1}\tau_{[k\uparrow n+r]}\tau_{[k-1\uparrow n+r-1]} \\
		&=0
	\end{align*}
	where the first equality is obtained by using Lemma \ref{braid}. Hence $d_{k-1}d_k=0$, and we have indeed defined a complex.
	
	We now show, as in Theorem \ref{coho}, that the cohomology of $\mathrm{Ad}_{E_i}^{(n)}(M)$ is concentrated in degree zero. 
	\begin{thm}\label{cohodp}
		Let $M \in \mathcal H[i]$ and let $n\geqslant 0$. Then
		\[
			H^k(\mathrm{Ad}_{E_i}^{(n)}(M)) = \left \{ \begin{array}{ll}
				\mathrm{ad}_{E_i}^{(n)}(M) & \text{ if } k=0, \\
				0 & \text{ otherwise.}
			\end{array} \right.
		\]
	\end{thm}

	\begin{proof}
		Assume that $M \in H_{\beta}\mathrm{-mod}$ with $\beta \in Q^+$ of height $r$. The proof is by induction on $n$. The result is clear if $n=0$. Assume the result proved for some $n \geqslant 0$. There is a morphism
		\[
			\tau_{E_i,\mathrm{Ad}_{E_i}^{(n)}(M)} : q_i^{\left<i^{\vee},\beta\right> +2n}\mathrm{Ad}_{E_i}^{(n)}(M) \rightarrow E_i\mathrm{Ad}_{E_i}^{(n)}(M)
		\]
		defined on the component of cohomological degree $-k$ to be the map
		\begin{align*}
			\tau_S &= \left(\tau_{[1\uparrow n-k]}ME_i^k\right) \circ \left(E_i^{n-k}\tau_{E_i,M}E_i^{k}\right) \circ \left(E_i^{n-k}M\tau_{[1\uparrow k]}\right)\\
			&=\left \{ \begin{array}{rcl}
			E_i^{(n-k)}ME_i^{(k)}E_i & \rightarrow & E_iE_i^{(n-k)}ME_i^{(k)},\\
			\tau_{\omega_0[k+1+r,n+r+1]}\tau_{\omega_0[2,k+1]}1_{(n-k)i,\beta,(k+1)i} \otimes_{H_{\beta}} m & \mapsto & \tau_{\omega_0[k+1+r,n+r+1]}\tau_{\omega_0[2,k+1]}\tau_{[1\uparrow n+r]}1_{(n-k+1)i,\beta,ki}\otimes_{H_{\beta}} m.
			\end{array} \right.
		\end{align*}
		Note that this just the restriction of the morphism defined in Proposition \ref{Adind}, and we prove similarly that this is well-defined. Let $\varUpsilon$ be the cone of $\tau_{E_i,\mathrm{Ad}_{E_i}^{(n)}(M)}$. By the same argument as the proof of Theorem \ref{coho}, $\varUpsilon$ has its cohomology concentrated in degree zero and equal to $\mathrm{ad}_{E_i}(\mathrm{ad}_{E_i}^{(n)}(M))$.
		
		We now prove that $\mathrm{Ad}_{E_i}^{(n+1)}(M)$ is a direct summand of $\varUpsilon$, up to a shift. First, we define a map $G : \mathrm{Ad}_{E_i}^{(n+1)}(M) \rightarrow \varUpsilon$ as follows: on the component of cohomological degree $-k$, $G$ is the canonical inclusion of $E_i^{(n+1-k)}ME_i^{(k)}$ in the first summand of $E_i^{(n+1-k)}ME_i^{(k-1)}E_i \oplus E_iE_i^{(n-k)}ME_i^{(k)}$. It follows easily from the definitions that $G$ is a morphism of complexes.
		
		We need to define a left inverse to $G$. We will denote by $h_d$ the complete symmetric polynomial of degree $d$. Recall the idempotent $e_n = x_2\ldots x_n^{n-1}\tau_{\omega_0[1,n]} \in H_{ni}$. We define new idempotents $f_k,f_k'$ in $H_{\beta+(n+1)i}$ by
		\begin{align*}
			&f_k = (e_{n+1-k} \diamond 1_{\beta} \diamond e_k), \\
			& f'_k = (-1)^{k-1} h_{k-1}(x_1,x_{k+r+1},\ldots,x_{n+r+1}) (e_{n+1-k} \diamond 1_{\beta} \diamond ((e_{k-1} \diamond 1_i)\tau_{[1\uparrow k-1]})).
		\end{align*}
		Note that $f_kf_k'=f_k$ and $f_k'f_k=f_k'$. We have a morphism
		\[
			\left \{ \begin{array}{rcl}
				E_i^{n+1-k}ME_i^k & \rightarrow & E_i^{(n+1-k)}ME_i^{(k)} \\
				1_{(n+1-k)i,\beta,ki}\otimes_{H_{\beta}} m & \mapsto & f_k\otimes_{H_{\beta}} m
			\end{array} \right.
		\]
		that we will just denote $f_k$, and similarly for $f_k'$. Then we define a morphism $F :\varUpsilon \rightarrow \mathrm{Ad}_{E_i}^{(n+1)}(M)$ as follows: on the component of cohomological degree $-k$, is defined by the row matrix
		\[
			\left[f_k' \quad x_{n+r}^kf_k \right] : E_i^{(n+1-k)}ME_i^{(k-1)}E_i \bigoplus E_iE_i^{(n-k)}ME_i^{(k)} \rightarrow E_i^{(n+1-k)}ME_i^{(k)}.
		\]
		It is clear that $FG=1$, we just need to check that $F$ is indeed a morphism of complexes. This amounts to checking the following relations in $H_{\beta+(n+1)i}1_{(n+2-k)i,\beta,(k-1)i}\otimes_{H_{\beta}}H_{\beta}^i$:
		\begin{align*}
			&x_{n+r+1}^kf_k\tau_{[k\uparrow n+r]} = (1_i \circ e_{n-k}\diamond 1_{\beta} \circ e_k) \tau_{[k\uparrow n+r-1]} x_{n+r+1}^{k-1}f_{k-1},\\
			&f_k'\tau_{[k\uparrow n+r]} = (e_{n+1-k}\diamond 1_{\beta} \diamond e_{k-1} \circ 1_i)(\tau_{[1\uparrow n+r]}x_{n+r+1}^{k-1}f_{k-1} - \tau_{[k\uparrow n+r]}f_{k-1}').
		\end{align*}
		The first relation follows from the following computation
		\begin{align*}
			&(1_i\diamond e_{n-k} \diamond 1_{\beta} \diamond e_k) \tau_{[k\uparrow n+r-1]}x_{n+r+1}^{k-1}f_{k-1} \\
			&=x_{n+r+1}^{k-1}(x_{k+r+2}\ldots x_{n+r}^{n-k-1})(1_{(n+1-k)i} \diamond 1_{\beta} \diamond e_k) \tau_{\omega_0[k+r+1,n+r]} \tau_{[k\uparrow n+r-1]} (e_{n+2-k}\diamond 1_{\beta} \diamond e_{k-1}) \\
			&=x_{n+r+1}^{k-1}(x_{k+r+2}\ldots x_{n+r}^{n-k-1})(1_{(n+1-k)i} \diamond 1_{\beta} \diamond e_k) \tau_{[k\uparrow k+r-1]} \tau_{\omega_0[k+r,n+r]} (e_{n+2-k}\diamond 1_{\beta} \diamond e_{k-1}) \\
			&=x_{n+r+1}^{k-1}(x_{k+r+2}\ldots x_{n+r}^{n-k-1})(1_{(n+1-k)i} \diamond 1_{\beta} \diamond e_k) \tau_{[k\uparrow k+r-1]} x_{n+r+1}^{n+1-k} (\tau_{\omega_0[1,n+2-k]}\diamond 1_{\beta} \diamond e_{k-1}) \\
			&=x_{n+r+1}^{n}(x_{k+r+2}\ldots x_{n+r}^{n-k-1})(\tau_{\omega_0[1,n+1-k]} \diamond 1_{\beta} \diamond e_k) \tau_{[k\uparrow n+r]} \\
			&= x_{n+r+1}^k(e_{n+1-k}\diamond 1_{\beta} \diamond e_k) \tau_{[k\uparrow n+r]}.
		\end{align*}
		Note that this computation actually holds in $H_{\beta+(n+1)i}1_{(n+2-k)i,\beta,(k-1)i}$, without needing to tensor by $H_{\beta}^i$ on the right.
		
		For the second equality, we start by computing $(e_{n+1-k}\diamond 1_{\beta} \diamond e_{k-1} \circ 1_i)\tau_{[1\uparrow n+r]}x_{n+r+1}^{k-1}f_{k-1}$ and  $(e_{n+1-k}\diamond 1_{\beta} \diamond e_{k-1} \diamond 1_i)\tau_{[k\uparrow n+r]}f_{k-1}'$ separately. We have
		\begin{align*}
			&(e_{n+1-k}\diamond 1_{\beta} \diamond e_{k-1} \diamond 1_i)\tau_{[1\uparrow n+r]}x_{n+r+1}^{k-1}f_{k-1} \\
			&=(e_{n+1-k}\diamond 1_{\beta} \diamond e_{k-1} \diamond 1_i)\tau_{[1\uparrow n+r]}x_{n+r+1}^{k-1}(e_{n+2-k}\diamond 1_{\beta} \diamond e_{k-1}) \\
			&=((x_2\ldots x_{n+1-k}^{n-k})\diamond 1_{\beta} \diamond e_{k-1} \diamond 1_i)\tau_{[1\uparrow k+r-1]}h_{k-1}(x_{k+r},\ldots,x_{r+n+1})(\tau_{\omega_0[1,n+2-k]}\diamond 1_{\beta} \diamond e_{k-1}) \\
			&=(e_{n+1-k}\diamond 1_{\beta} \diamond e_k \circ 1_i) \tau_{[1\uparrow k+r-1]}h_{k-1}(x_{k+r},\ldots,x_{r+n+1}) \tau_{[k+r\uparrow n+r]}.
		\end{align*}
		Now in $H_{\beta+(n+1)i}1_{(n+2-k)i,\beta,(k-1)i}\otimes_{H_{\beta}}H_{\beta}^i$ we have
		\[
			\tau_{[k\uparrow k+r-1]}x_{k+r} = x_{k}\tau_{[k\uparrow k+r-1]}
		\]
		 because $(\tau_{[k\uparrow k+r-1]}x_{k+r} - x_{k}\tau_{[k\uparrow k+r-1]})1_{i,\beta,(k-1)i} \in (1_{i,\beta-i,i,(k-1)i})$ (this is the same argument that we used to prove that $\tau_{E_i,M}$ is well-defined in Proposition \ref{welldef}). Hence we conclude that
		 \[
		 	(e_{n+1-k}\diamond 1_{\beta} \diamond e_{k-1} \diamond 1_i)\tau_{[1\uparrow n+r]}x_{n+r+1}^{k-1}f_{k-1} = (e_{n+1-k}\diamond 1_{\beta} \diamond e_k \diamond 1_i) \tau_{[1\uparrow k-1]}h_{k-1}(x_{k},x_{k+r+1}\ldots,x_{n+r+1}) \tau_{[k\uparrow n+r]}.
		 \]
		 Similarly, we have
		 \[
		 	(e_{n+1-k}\diamond 1_{\beta} \diamond e_{k-1} \diamond 1_i)\tau_{[k\uparrow n+r]}f_{k-1}' = (e_{n+1-k}\diamond 1_{\beta} \diamond e_{k-1} \diamond 1_i) h_{k-2}(x_1,x_k,x_{k+r+1},\ldots,x_{n+r+1}) \tau_{[1\uparrow k-2]}\tau_{[k\uparrow n+r]}
		 \]
		 To conclude, we use the following formula in the affine nil Hecke algebra $H_k^0$:
		 \[
		 	\tau_{\omega_0[2,k]}\tau_{[1\uparrow k-1]}x_k^a = \tau_{\omega_0[2,k]}x_1^a\tau_{[1\uparrow k-1]} +\tau_{\omega_0[2,k]}f_{a-1}(x_1,x_k)\tau_{[1\uparrow k-2]}.
		 \]
		 This is easily derived by induction on $a$. Then we have
		 \begin{align*}
			 &(e_{n+1-k}\diamond 1_{\beta} \diamond e_{k-1} \diamond 1_i)\left(\tau_{[1\uparrow n+r]}x_{n+r+1}^{k-1}f_{k-1} - \tau_{[k\uparrow n+r]}f_{k-1}'\right)\\
			 &=(e_{n+1-k}\diamond 1_{\beta} \diamond e_{k-1} \diamond 1_i)\left(\tau_{[1\uparrow k-1]}h_{k-1}(x_{k},x_{k+r+1}\ldots,x_{n+r+1})  - h_{k-2}(x_1,x_k,x_{k+r+1},\ldots,x_{n+r+1}\right) \tau_{[1\uparrow k-2]})\tau_{[k\uparrow n+r]} \\
			 &= (e_{n+1-k}\diamond 1_{\beta} \diamond e_{k-1} \diamond 1_i)h_{k-1}(x_1,x_{k+r+1}\ldots,x_{n+r+1})\tau_{[1\uparrow n+r]} \\
			 &=f_{k}'\tau_{[k\uparrow n+r]}.
		 \end{align*}
		 This completes the proof that $FG=1$.
		 
		 So $\mathrm{Ad}_{E_i}^{(n+1)}(M)$ is a direct factor of $\varUpsilon$, hence has cohomology only in degree 0. Furthermore, on degree zero cohomology, the morphism $GF$ of $\varUpsilon$ induces the idempotent projecting on the summand $\mathrm{ad}_{E_i}^{(n+1)}(M)$ of $\mathrm{ad}_{E_i}(\mathrm{ad}_{E_i}^{(n)}(M))$. Hence $H^0(\mathrm{Ad}_{E_i}^{(n+1)}(M))=\mathrm{ad}_{E_i}^{(n+1)}(M)$.
	\end{proof}

	\medskip
	
	\subsection{Projective resolutions and Serre relations} Using Theorems \ref{coho} and \ref{cohodp}, we can construct projective resolutions.
	
	\begin{prop}\label{projres}
		Let $M$ be a module in $\mathcal H[i]$ which is projective as a module in $\mathcal H$. Then $\mathrm{Ad}_{E_i}^n(M)$ (resp. $\mathrm{Ad}_{E_i}^{(n)}(M)$) is a projective resolution of $\mathrm{ad}_{E_i}^n(M)$ (resp. $\mathrm{ad}_{E_i}^{(n)}(M)$) in $\mathcal H$.
	\end{prop}

	Typically, this applies to $M=E_{j_1}\ldots E_{j_r}$ for $j_1,\ldots,j_r \neq i$. In particular, we have obtained projective resolutions for the generators of $\mathcal{H}[i]$.
	
	We now state the categorical version of the quantum higher Serre relations.

	\begin{thm}\label{Serre}
		Let $j\in I \setminus \lbrace i \rbrace$ and let $m\geqslant 0$. If $n>-mc_{i,j}$, then $\mathrm{Ad}_{E_i}^{(n)}(E_j^m)$ is null-homotopic.
	\end{thm}

	\begin{proof}
		By Proposition \ref{projres}, $\mathrm{Ad}_{E_i}^{(n)}(E_j^m)$ is a projective resolution of $\mathrm{ad}_{E_i}^{(n)}(E_j^m)$. Assume $n>-mc_{i,j}$. By Proposition \ref{vanish}, $H_{ni+mj}^i=0$. Hence $\mathrm{ad}_{E_i}^{(n)}(E_j^m)=0$ and $\mathrm{Ad}_{E_i}^{(n)}(E_j^m)$ is null-homotopic.
	\end{proof}

	\bigskip
	
	\bibliographystyle{alpha}
	\bibliography{biblio}
	
	\bigskip
	
	{\small \textsc{UCLA Mathematics Department, Box 951555, Los Angeles, CA 90095-1555}}
	
	\textit{E-mail address:} \href{mailto:lvera@math.ucla.edu}{\texttt{lvera@math.ucla.edu}}
	
\end{document}